\documentclass[a4paper]{article}

\usepackage[english
]{babel}
\usepackage[utf8x]{inputenc}
\usepackage[T1]{fontenc}
\usepackage{bm}
\usepackage[a4paper,top=3cm,bottom=2cm,left=2cm,right=2cm,marginparwidth=1.75cm]{geometry}

\usepackage{amsmath}
\usepackage{amsthm}
\usepackage{graphicx}
\usepackage[colorinlistoftodos]{todonotes}
\usepackage[colorlinks=true, allcolors=blue]{hyperref}
\newtheorem{theorem}{Theorem}[section]
\newtheorem{corollary}{Corollary}[theorem]
\newtheorem{lemma}[theorem]{Lemma}
\usepackage{amssymb}
\usepackage{mathrsfs}
\newtheorem{definition}{Definition}
\newtheorem*{assumption}{Assumption}

\newenvironment{remark}[1][Remark]{\begin{trivlist}
\item[\hskip \labelsep {\bfseries #1}]}{\hfill\qed\end{trivlist}}
\usepackage{dsfont}

\def\XXint#1#2#3{{\setbox0=\hbox{$#1{#2#3}{\int}$ }
\vcenter{\hbox{$#2#3$ }}\kern-.6\wd0}}

\usepackage{mathtools}

\newcommand{\indep}{\perp \!\!\! \perp}

\graphicspath{ {images/} }

\usepackage[shortlabels]{enumitem}

\DeclareMathOperator*{\essinf}{ess\,inf}
\DeclareMathOperator*{\esssup}{ess\,sup}
\DeclareMathOperator*{\esslim}{ess\,lim}
\DeclareMathOperator*{\esslimsup}{ess\,limsup}
\DeclareMathOperator*{\essliminf}{ess\,liminf}

\usepackage{xcolor}

\newcommand{\tr}[0]{\text{tr}}

\title{A Viscosity Solution Theory of Stochastic Hamilton-Jacobi-Bellman equations in the Wasserstein Space}
\author{Hang Cheung \\ \href{}{hang.cheung@ucalgary.ca}
\and Jinniao Qiu \\\href{}{jinniao.qiu@ucalgary.ca}\and Alexandru Badescu\\\href{}{abadescu@ucalgary.ca}}
\date{}
\begin{document}
\maketitle
\begin{abstract}
    This paper is devoted to a viscosity solution theory of the stochastic Hamilton-Jacobi-Bellman equation in the Wasserstein spaces for the mean-field type control problem which allows for random coefficients and may thus be non-Markovian. 
The value function of the control problem is proven to be the unique viscosity solution. The major challenge lies in the mixture of the lack of local compactness of the Wasserstein spaces and the non-Markovian setting with random coefficients and  various techniques are used, including It\^o processes parameterized by random measures,  the conditional law invariance of the value function,  a novel tailor-made compact subset of measure-valued processes,  finite dimensional approximations via stochastic n-player differential games with common noises, and so on.
\end{abstract}
\section{Introduction}

In the last decade, the theory of mean field type (McKean-Vlasov type) control problems and related mean field games has experienced a significant surge in interest. Unlike standard control theory and stochastic differential games, this captivating research area introduces a novel perspective, allowing both the state process and cost functional to incorporate their own probability measures as arguments. This unique feature enables the modeling of large-scale systems with agents interacting symmetrically, leading to rich and complex dynamics.\\
\hfill\\
The origins of mean field control theory and mean field games can be traced back independently to Caines, Huang and Malhame \cite{cis/1183728987} and Lasry
and Lions \cite{lasry_mean_2007}. As the importance of mean field phenomena continues to be recognized, several comprehensive monographs on these topics have emerged, providing in-depth insights and a broader perspective; see e.g. Carmona and Delaure \cite{carmona_probabilistic_2018_1,carmona_probabilistic_2018}, Bensoussan, Frehse and Yam \cite{bensoussan_mean_2013}, and the references therein.\\
\hfill\\
In the realm of control theory and differential game literature, two principal approaches stand out: the Pontryagin Maximum Principle and the Hamilton-Jacobi-Bellman equation. The Maximum Principle aims to provide necessary conditions characterizing the optimal control, and these necessary conditions often resort to the solvability of a system of mean field forward-backward stochastic differential equations. Notable contributions on this topic include Andersson and Djehiche \cite{andersson_maximum_2011}, Li \cite{li_stochastic_2012}, Buckdahn, Djehiche and Li \cite{buckdahn_general_2011}, Carmona and Delaure \cite{Delaure_Carmona_ECP,Carmona_Delaure_AOP}, Bensoussan and Yam \cite{bensoussan_alain_control_2019}, Bensoussan, Graber and Yam \cite{bensoussan_control_2020}, Bensoussan, Tai and Yam \cite{bensoussan_mean_2023}, Bensoussan, Wong, Yam and Yuan \cite{bensoussan_theory_2023}, Gangbo, Mészáros, Mou and Zhang \cite{Gangbo_Mou_Zhang_AOP}, Mou and Zhang \cite{mou_wellposedness_2019}, Chassagneux, Crisan and Delarue \cite{Dan_Delaure_book}, among others.\\\hfill\\
In this paper, our focus lies on the latter approach, the main goal being to develop a theory of viscosity solution for the Stochastic Hamilton-Jacobi-Bellman (SHJB) equations associated to the mean field type control problems with random coefficients. Specifically, given  a probability space $(\Omega,\mathcal{F},\mathbb{P})$ , we consider the following optimal control problem
\begin{align}
\label{value_intro}
    \inf_{\alpha\in\mathcal{A}} \mathbb{E}\Bigg[\int_0^T f_s(X_s,\mathbb{P}_{X_s}^{W^0},\alpha_s)+g(X_T,\mathbb{P}_{X_T}^{W^0})\Bigg]
\end{align}
subject to
\begin{eqnarray}
\label{basic_dynamics_intro}
\begin{cases}
    dX_s = & b_s(X_s,\mathbb{P}_{X_s}^{W^0},\alpha_s)ds+\sigma_s(X_s,\mathbb{P}_{X_s}^{W^0},\alpha_s)dW_s 
    + \sigma^0_s(X_s,\mathbb{P}_{X_s}^{W^0},\alpha_s)dW^0_s,\quad s\in[0,T],\\
    X_0 \,\,\,= & \xi.
\end{cases}
\end{eqnarray}
Here, $T\in(0,\infty)$ is a fixed deterministic terminal time, $\xi\in L^2((\Omega,\mathcal{F},\mathbb{P});\mathbb{R}^d)$ such that $\mathcal{L}(\xi) = \mu$, where $\mathcal{L}(\xi)$ denotes the law of $\xi$, and $\mu$ is some probability measure with finite second moment. $W$, $W^0$ are two independent Brownian motions on $(\Omega,\mathcal{F},\mathbb{P})$ representing respectively idiosyncratic noise and common noise. Our filtration is taken to be $\{\mathcal{F}_s\}_{s\geq 0} := (\sigma(W_s)\vee \sigma(W^0_s)\vee \mathcal{G})_{s\geq 0}$, where $\mathcal{G}$ is a sub-$\sigma$-algebra of $\mathcal{F}$ such that all the above $\xi$ is $\mathcal{G}$-measurable, and it is independent of $(\sigma(W_s)\vee\sigma(W_s^0))_{s\geq 0}$. Let $A \subset \mathbb{R}^d$ be a nonempty compact set and $\mathcal{A}$ be the set of all $A$-valued and $(\mathcal{F}_t)_{t\geq 0}$-adapted processes. The process $(X_s)_{s\in [0,T]}$ is the state process and $\mathbb{P}_{X_t}^{W^0}$ denotes the conditional distribution of $X_t$ given $W^0$ (or equivalently, given $\mathcal{F}_t^0$, $\mathcal{F}_t^0$ being the $\sigma$-algebra generated by $W^0$). 
We write $X_t^{r,\xi,\alpha}$ for $0\leq r \leq t \leq T$ to indicate the dependence of the state process on the control $\alpha\in\mathcal{A}$, the initial time $r$ and the initial state $\xi\in L^2((\Omega,\mathcal{F},\mathbb{P});\mathbb{R}^d)$.\\
\hfill\\
Herein, we consider the non-Markovian case where the coefficient $b$, $\sigma$, $\sigma^0$, $f$, and $g$ depend not only on time, space, measure and control, but also explicitly on $\omega^0 \in \Omega^0$. Here, $\Omega^0$ is a subset of $\Omega$ supporting only the common noise. For $t\in[0,T]$, we also let $\{\mathcal{F}^t_s\}_{s\geq 0}:= \{\sigma(W_{s\vee t}-W_t)\vee \sigma(W^0_s)\vee \mathcal{G}\}_{s\geq 0}$, and $\mathcal{A}_t$ be the set of all $A$-valued and $\mathcal{F}^t$-adapted process. The dynamic cost functional is defined by
\begin{align}
\label{functional_intro}
    J(t,\xi,\alpha):=& \mathbb{E}\Bigg[\int_t^T f_s(X_s^{t,\xi,\alpha},\mathbb{P}_{X_s^{t,\xi,\alpha}}^{W^0},\alpha_s)+g(X_T^{t,\xi,\alpha},\mathbb{P}_{X_T^{t,\xi,\alpha}}^{W^0})\Big|\mathcal{F}^0_t\Bigg],\quad t\in[0,T],
\end{align}
and the value function is given by 
\begin{align}
\label{value_intro_2}
    v(t,\xi) = \essinf_{\alpha\in\mathcal{A}_t} J(t,\xi,\alpha),\quad t\in[0,T]. 
\end{align}
Due to the randomness of the coefficients, the value functions $v$ is a function of time $t$, random variables $\xi$ and $\omega^0\in\Omega^0$. However, from the structure of our mean field dynamics from~\eqref{basic_dynamics_intro}-\eqref{value_intro_2}, and our richness of the control space, it can be proven that $v$ is conditional law invariant, as opposed to the law invariance in Cosso, Gozzi, Kharroubi, Pham, and Rosestolato  \cite{cosso_optimal_2022}. In fact, $v$ is shown to be the unique viscosity solution (in a suitable sense) to the following Stochastic Hamilton-Jacobi-Bellman (SHJB) equation on the Wasserstein space:
\begin{align}
\label{intro_HJB}
    \begin{cases}
        -\Game_t u(t,\mu)-\mathbb{H}(t,\mu,\partial_\mu u,\partial_x\partial_\mu u,\partial_\mu^2 u,\partial_\mu\Game_w u) = 0,\quad (t,\mu)\in[0,T]\times\mathcal{P}_2(\mathbb{R}^d),\\
        u(T,\mu) = g(\mu),\quad \mu\in \mathcal{P}_2(\mathbb{R}^d),
    \end{cases}
\end{align}
with
\begin{align*}
    \mathbb{H}(t,\mu,P,Q,R,S):=& \int_{\mathbb{R}^d}\essinf_{\alpha,\alpha'\in A}\Bigg\{f_t(x,\mu,\alpha)+\Big\langle b_t(x,\mu,\alpha),P\Big\rangle+\frac{1}{2}\tr\Big\{ (\sigma_t\sigma_t^{\intercal}+\sigma_t^0\sigma_t^{0;\intercal})(x,\mu,\alpha)Q\Big\}\\
    &+\displaystyle\int_{\mathbb{R}^d}\frac{1}{2}\tr\Big\{\sigma_t^0(x,\mu,\alpha)\sigma_t^{0;\intercal}(x',\mu,\alpha')R\Big\}\mu(dx')+\tr\Big\{\sigma_t^{0;\intercal}(x,\mu,\alpha)S\Big\}\Bigg\}\mu(dx),
\end{align*}
and $\mathbb{H}:\Omega^0\times[0,T]\times\mathcal{P}_2(\mathbb{R}^d)\times\mathcal{L}^2(L^2(\mathbb{R}^d,\mathcal{B}(\mathbb{R}^d),\mu;\mathbb{R}^d))\times\mathcal{L}^2(L^2(\mathbb{R}^d,\mathcal{B}(\mathbb{R}^d),\mu;\mathbb{R}^{d\times d})\times\mathcal{L}^2(L^2(\mathbb{R}^d\times\mathbb{R}^d,\mathcal{B}(\mathbb{R}^d\times\mathbb{R}^d),\mu\otimes\mu;\mathbb{R}^{d\times d}))\times\mathcal{L}^2(L^2(\mathbb{R}^d,\mathcal{B}(\mathbb{R}^d),\mu;\mathbb{R}^{d\times d})) \to\mathbb{R}$. The aforementioned $\mathcal{P}_2(\mathbb{R}^d)$ is the Wasserstein space, namely the space of probability measures with finite second moments equipped with the $2$-Wasserstein distance (see Definition \ref{Wasserstein_space}); $\partial_\mu$, $\partial_\mu^2$ are correspondingly Lions' first and second derivative (see Definition \ref{Lions' derivative}). Also, the unknown adapted random field $u$ in (\ref{intro_HJB}) is confined to the following form:
\begin{align}
\label{u_doob_meyer}
    u(t,\mu) = u(T,\mu) - \int_t^T\Game_s u(s,\mu)ds - \int_t^T\Game_w u(s,\mu)dW_s^0,\quad\forall (t,\mu)\in[0,T]\times\mathcal{P}_2(\mathbb{R}^d).
\end{align}
The Doob-Meyer decomposition theorem indicates the uniqueness of the pair $(\Game_t u,\Game_w u)$ and thus the linear operators $\Game_t$ and $\Game_w$ may be well defined in certain spaces. The pair $(\Game_t u,\Game_w u)$ may also be defined as two differential operators; see \cite[Theorem 4.3]{Leao_Ohashi} for instance. By comparing $(\ref{intro_HJB})$ and $(\ref{u_doob_meyer})$, we may rewrite the SHJB $(\ref{intro_HJB})$ formally into
\begin{align}
\begin{cases}
    -d u(t,\mu)=\,\mathbb{H}\left(t,\mu,\partial_\mu u,\partial_x\partial_\mu u,\partial_\mu^2 u, \partial_x\psi\right) -\psi(t,\mu)dW^0_t,\quad (t,\mu)\in[0,T]\times\mathcal{P}_2(\mathbb{R}^d),\\
        u(T,\mu)= \, g(\mu),\quad \mu\in \mathcal{P}_2(\mathbb{R}^d),
\end{cases}
\end{align}
where the pair $(u,\psi) = (u,\Game_w u)$ is unknown.\\\hfill\\
Several works on viscosity solutions have already been conducted in the case when all the coefficients are deterministic.    Pham and Wei \cite{Pham_Wei_Dynamic_Programming} established a dynamic programming principle for the case where the control is simply $(\mathcal{F}_t^0)_{t\geq 0}$-adapted. Under this assumption, the law invariance is an immediate consequence. The HJB equation derived from their work is lifted to the Hilbert space of random variables, and the existing viscosity theory on Hilbert spaces is then applied. Utilizing the mature viscosity theory on the Hilbert space, as a corollary, one can obtain both the uniqueness and existence of the viscosity solution. A similar Hilbert space approach can also be found in Pham and Wei \cite{pham_bellman_2018}. 
The reason they lifted the problem to the Hilbert space is due to the lack of local compactness in the infinite-dimensional Wasserstein space. To avoid this lifting, Wu and Zhang in \cite{mean_field_path} introduced a notion of viscosity solution that differs from the traditional Crandall-Lions definition. Instead, they required the maximum/minimum condition to hold on some compact subset of the Wasserstein space, rather than just in a local neighborhood. By doing so, they successfully developed a comprehensive theory of viscosity solutions. To maintain consistency with the original Crandall-Lions definition and overcome the local compactness issue arising from the Wasserstein space, Cosso, Gozzi, Kharroubi, Pham, and Rosestolato in \cite{cosso_master_2022} built their strategy upon the Borwein-Preiss generalization of Ekeland’s variational principle on the Wasserstein space. By incorporating finite-dimensional approximations of the value function derived from the related cooperative $n$-player game, they successfully extended Crandall-Lions' definition of viscosity solution to the Wasserstein space while assuming boundedness and Lipschitz continuity of the coefficients. A large body of literature continues to explore this topic. For instance, Burzoni, Ignazio, Reppen, and Soner in \cite{soner_jump} developed a viscosity solution theory for processes with jumps, while Soner and Yan in \cite{soner_viscosity_2022} introduced the problem formulated under the intrinsic linear derivative on a torus.\\\hfill\\
For the case involving random coefficients without the measure term, it was initially introduced by Peng in \cite{peng_shjb}. In his work, Peng established the existence and uniqueness of weak solutions in Sobolev spaces for the superparabolic semilinear stochastic HJB equations, while acknowledging the well-posedness of general cases as an open problem. Our previous works Qiu \cite{qiu_stochastic_hjb,qiu_controlled_2022}, Qiu and Wei \cite{qiu_uniqueness_2019}, Qiu and Zhang \cite{qiu_stochastic_2023}, Qiu and Yang \cite{qiu_optimal_2023} addressed this problem, and extended it to several other cases, including the basic case on $\mathbb{R}^d$, the case of stochastic differential game on $\mathbb{R}^d$ and the case of controlled ordinary differential equations with random path-dependent coefficients in finite and infinite dimension.\\\hfill\\
To the best of our knowledge, there is currently no existing literature on the viscosity solution theory of SHJB equations on the Wasserstein space when the coefficients are random. Thus, our primary objective is to fill this important void in the field by proposing a  viscosity theory on the Wasserstein space for SHJB equations, for which we need to deal with not just the lack of local compactness of the underlying spaces but also the non-Markovian setting with random coefficients. \\\hfill\\
First, the lack of local compactness of the Wasserstein space poses serious technical difficulties when establishing the comparison principle. To address this issue, we draw inspiration from Wu and Zhang \cite{mean_field_path} and define our compact subset $\mathcal{P}_L^\tau(t_0,\rho_0)$ (see Appendix \ref{compact_subset}). We require our maximum/minimum conditions to hold in this compact set, rather than just in a local neighborhood. Unlike the approach in   \cite{mean_field_path}, our set is a compact subset of Wasserstein-space-valued random variables representing the conditional distribution of the state process with drift and diffusion bounded by a constant $L$. To prove our compactness result, we supplement it with the relatively new result from Wang, Zhu, and Kloeden in \cite{wang2019compactness}, instead of using the classic Prokhorov's theorem or the result in Zheng \cite{zheng_tightness_nodate}. Moreover, since our compact subset involves random variables of laws, this necessitates rigorous methods to ensure the validity and reliability of our results. A crucial note to mention is that the approach of Borwein-Preiss variational principle in \cite{cosso_master_2022} cannot be readily generalized to the case at hand due to the presence of the measurability issue. This aspect will be a focus of our future research.\\
\hfill\\
The second major challenge lies in defining the appropriate notion of a viscosity solution. The coefficients involved are merely measurable with respect to $\omega^0 \in \Omega^0$, lacking any equipped topology. Therefore, defining the viscosity solution pointwise with respect to $\omega^0 \in \Omega^0$ is not suitable. Instead, we choose the class of test functions to be functions in the form of random fields (\ref{u_doob_meyer}), possessing sufficient regularity regarding the measure variable. This choice allows us to apply the It\^o-Wentzell formula from \cite{dos2022ito}. Furthermore, in establishing uniqueness, we encounter a lack of regularity results regarding the measure variable in the literature of second-order parabolic equations on Wasserstein space. To overcome this, we employ an $n$-player games approximation of the value function, akin to \cite{cosso_master_2022}. However, we face a significant difference: our convergence scheme involves a game system with common noise, necessitating the use of the limiting theory from \cite{limit_theory_tan}, in contrast to the result in \cite{daniel_limit} used in \cite{cosso_master_2022}. Moreover, the  aforementioned convergence result requires our control set to be $\mathcal{F}^t$-adapted, not just $\mathcal{F}^0$-adapted as in \cite{Pham_Wei_Dynamic_Programming}, adding additional technical difficulties in proving the conditional law invariance of the value function.  
 \\
\hfill\\
This paper is structured as follows: In Section 2, we introduce some notations, present the standing assumption on coefficients, formulate the problem, sketch the conditional law invariance interpretation of the value function and prove the dynamic programming. Moving to Section 3, we first recall the readers the definition of the Lions' derivative and an It\^o-Wentzell formula that suits our case by \cite{dos2022ito}. After that, we define the viscosity solution and verify the value function as a viscosity solution. A comparison theorem is then provided, and we establish the uniqueness of the viscosity solution, being the value function of the control problem. In the Appendix, we provide supplementary details of the measurable selection theorem and offer a proof of the law invariance property of the value function. Additionally, we construct the compact subset $\mathcal{P}_L^\tau(t_0,\rho_0)$. At the end, we present the detailed finite-dimensional approximation utilized in our proof of uniqueness.
\section{Dynamic Programming Principle of the McKean Vlasov Control Problem}
\subsection{Setup of the Probability Space}
\label{setup}
Fix  a probability space $(\Omega, \mathcal{F},\mathbb{P})$ of the form $(\Omega^0\times\Omega^1,\mathcal{F}^0\otimes \mathcal{F}^1,\mathbb{P}^0\otimes \mathbb{P}^1)$. The random coefficients live on $(\Omega^0,\mathcal{F}^0,\mathbb{P}^0)$, which supports an $m$-dimensional Brownian motion $W^0$.\\
\hfill\\
For $(\Omega^1,\mathcal{F}^1,\mathbb{P}^1)$, it is of the form $(\tilde{\Omega}^1\times \hat{\Omega}^1,\mathcal{G}\otimes\hat{\mathcal{F}}^1,\tilde{\mathbb{P}}^1\otimes\hat{\mathbb{P}}^1)$. On $(\hat{\Omega}^1,\hat{\mathcal{F}}^1,\hat{\mathbb{P}}^1)$
there lives an $m$-dimensional Brownian motion $W$, we regard this $W$ as the idiosyncratic noise. $(\tilde{\Omega}^1,\mathcal{G},\tilde{\mathbb{P}}^1)$ is where the initial random variables lives. We assume that $(\tilde{\Omega}^1,\mathcal{G},\tilde{\mathbb{P}}^1)$ is rich enough to support all probability laws in $\mathbb{R}^d$, i.e., for any probability law $\mu$ in $\mathbb{R}^d$, there exists $X\in \tilde{\Omega}^1$ such that $\mathbb{P}_X = \mu$.\\
\hfill\\
We write $\omega\in\Omega$ as $\omega = (\omega^0,\omega^1)$, and we regard the Brownian motions $W(\omega) = W(\omega^1)$, $W^0(\omega) = W^0(\omega^0)$. We denote by $\mathbb{E}^0$ (resp., $\mathbb{E}^1$) the expectation under $\mathbb{P}^0$ (resp., $\mathbb{P}^1$), by $\mathbb{F} = (\mathcal{F}_s)_{s\geq 0}:=(\sigma(W^0_s)\vee\sigma(W_s)\vee\mathcal{G})_{s\geq 0}$, $\mathbb{F}^t = (\mathcal{F}_s^t)_{s\geq 0}:=(\sigma(W_s^0)\vee\sigma(W_{s\vee t}-W_t)\vee\mathcal{G})_{s\geq 0}$ and $\mathbb{F}^{W^0} = (\mathcal{F}_s^{W^0})_{s\geq 0}:=(\sigma(W_s^0))_{s\geq 0}$. Without loss of generality, we assume they are $\mathbb{P}$-complete.
\subsection{Function Spaces and Notations}
We denote by $\mathscr{P}$ the $\sigma$-algebra of the predictable sets on $\Omega^0\times[0,T]$ associated with $\{\mathcal{F}_t^0\}_{t\geq 0}$.
\begin{definition}
    \label{Wasserstein_space}
    We introduce over $\mathbb{R}^d$ the space of probability measures $\mathcal{P}(\mathbb{R}^d)$ and its subset $\mathcal{P}_p(\mathbb{R}^d)$ of those with finite $p$-th moment, $p\geq 1$. The space $\mathcal{P}_p(\mathbb{R}^d)$ is equipped with the $p$-Wasserstein distance
\begin{align}\label{Wasserstein}
\mathcal{W}_p(\mu, \nu)=\inf _{\pi \in \Pi(\mu, \nu)}\Bigg(\int_{\mathbb{R}^d \times \mathbb{R}^d}|x-y|^p \pi(\mathrm{d} x, \mathrm{~d} y)\Bigg)^{\frac{1}{p}}, \quad \mu, \nu \in \mathcal{P}_p(\mathbb{R}^d),
\end{align}
where $\Pi(\mu, \nu)$ is the set of probability measures on $\mathbb{R}^d\times\mathbb{R}^d$ such that for all $\pi \in \Pi(\mu,v)$, $\pi(\mathbb{R}^d\times\cdot) = \mu$ and $\pi(\cdot\times\mathbb{R}^d) = \nu$. We call $(\mathcal{P}_p(\mathbb{R}^d),\mathcal{W}_p)$ the $p$-th Wasserstein space over $\mathbb{R}^d$, and it is a Polish space. Note that we shall be working on $\mathcal{P}_2(\mathbb{R}^d)$ most of the time. Finally, we denote $\operatorname{Supp}(\mu)$ the support of $\mu \in \mathcal{P}(\mathbb{R}^d)$.
\end{definition}
\begin{definition}
(Function Spaces). Let $\mathbb{B}$ be a Banach space equipped with norm $\|\cdot\|_{\mathbb{B}}$. For each $t\in[0,T]$, denote by $L^0((\Omega,\mathcal{F}_t,\mathbb{P});\mathbb{B})$ the space of $\mathbb{B}$-valued $\mathcal{F}_t$-measurable random variables. For $p\geq 1$, denote $L^p((\Omega,\mathcal{F}_t,\mathbb{P});\mathbb{B})$ the Banach space of $p$-integrable $\mathbb{B}$-valued $\mathcal{F}_t$-measurable random variables equipped with the norm:
\begin{align*}
    \|X\|_p:=(\mathbb{E}\|X\|_{\mathbb{B}}^p)^{1/p}.
\end{align*}
For $p\in[1,\infty]$, $\mathcal{S}^p(\mathbb{B})$ is the set of all the $\mathbb{B}$-valued, $\mathscr{P}$-measurable continuous process $\{\mathcal{X}_t\}_{t\in[0,T]}$ such that
\begin{align*}
    \|\mathcal{X}\|_{\mathcal{S}^p(\mathbb{B})}:=\Bigg\|\sup_{t\in[0,T]}\|\mathcal{X}_t\|_{\mathbb{B}}\Bigg\|_{L^p(\Omega^0,\mathcal{F}^0,\mathbb{P}^0)}<\infty.
\end{align*}
Denote by $\mathcal{L}^p(\mathbb{B})$ the totality of all the $\mathbb{B}$-valued, $\mathscr{P}$-measurable process $\{\mathcal{X}_t\}_{t\in[0,T]}$ such that
\begin{align*}
    \|\mathcal{X}\|_{\mathcal{L}^p(\mathbb{B})}:=\Bigg\|\Bigg(\int_0^T\|\mathcal{X}_t\|^p_{\mathbb{B}}dt\Bigg)^{1/p}\Bigg\|_{L^p(\Omega^0,\mathcal{F}^0,\mathbb{P}^0)}<\infty.
\end{align*}
Obviously $(\mathcal{S}^p(\mathbb{B}),\|\cdot\|_{\mathcal{S}^p(\mathbb{B})})$ and $(\mathcal{L}^p(\mathbb{B}),\|\cdot\|_{\mathcal{L}^p(\mathbb{B})})$ are Banach spaces. If a function of such classes admits a version with better properties, we always denote this version by itself.\\\hfill\\
If $\mathbb{B}$ is not separable, for $p\in[1,\infty]$, we define the spaces, 
\begin{align*}
    \mathcal{B}\mathcal{S}^p(\mathbb{B}):=\Big\{(\mathcal{X}_t)_{t\in [0,T]}\in\mathcal{S}^p(\mathbb{B})\,|\,\mathcal{X}_t:\Omega\to\mathbb{B}\text{ is Bochner/strong measurable for all }t\in[0,T]  \Big\},
\end{align*}
and similarly,
\begin{align*}
    \mathcal{B}\mathcal{L}^p(\mathbb{B}):=\Big\{(\mathcal{X}_t)_{t\in [0,T]}\in\mathcal{L}^p(\mathbb{B})\,|\,\mathcal{X}_t:\Omega\to\mathbb{B}\text{ is Bochner/strong measurable for all }t\in[0,T]  \Big\}.
\end{align*}
\end{definition}
\begin{definition}
For any $\mu\in\mathcal{P}_2(\mathbb{R}^d)$, $\varphi\in L^2((\mathbb{R}^d,\mathcal{B}(\mathbb{R}^d),\mu);\mathbb{R}^d)$, $\mu\otimes\mu$, $\psi \in L^2(\mathbb{R}^d\times\mathbb{R}^d,\mathcal{B}(\mathbb{R}^d\times\mathbb{R}^d),\mu\otimes\mu)$, we define the notation,
\begin{align*}
    \mu(\varphi) := \int_{\mathbb{R}^d}\varphi(x)\mu(dx),\quad\mu\otimes\mu(\psi):=\int_{\mathbb{R}^d\times\mathbb{R}^d}\psi(x,x')\mu(dx)\mu(dx'). 
\end{align*}
\end{definition}

\subsection{The Set of Admissible Controls}
We are given a compact subset $A$ of the Euclidean space equipped with the distance $d_A$. Let $t>0$, denote $\mathcal{A}$ (resp. $\mathcal{A}_t$) the set of $\mathbb{F}$-progressive process (resp. $\mathbb{F}^{t}$-progressively measurable process) $\alpha$ valued in $A$. Note that $\mathcal{A}$ (resp. $\mathcal{A}_t$) is a separable metric space endowed with the Krylov distance $\Delta(\alpha,\beta) = \mathbb{E}^0[\int_0^T d_A(\alpha_r,\beta_r)dr]$ 
 (resp. $\Delta(\alpha,\beta) = \mathbb{E}^0[\int_t^T d_A(\alpha_r,\beta_r)dr]$). Denote by $\mathcal{B}_{\mathcal{A}}$ (resp. $\mathcal{B}_{\mathcal{A}_t}$) the Borel $\sigma$-algebra of $\mathcal{A}$ (resp. $\mathcal{A}_t$).\\
\hfill\\
We assume that $(\Omega^0, W^0, \mathbb{P}^0)$ is the canonical space, i.e., $\Omega^0 = C(\mathbb{R}_{+},\mathbb{R}^m)$. We introduce the class of shifted control processes constructed by concatenation of paths: for $\alpha \in \mathcal{A}$ (resp. $\mathcal{A}_t$), $(r,\hat{\omega}^{0}) \in [0,T]\times\Omega^0$, set
\begin{align*}
    \alpha_s^{r,\hat{\omega}^0}(\omega^0):= \alpha_s(\hat{\omega}^0\otimes_r\omega^0),\,\,\,(s,\omega^0)\in [0,T]\times\Omega^0, 
\end{align*}
where
\begin{align*}
    \hat{\omega}^0\otimes_r\omega^0(s) := \hat{w}^0(s)1_{s<r}+\big(\hat{w}^0(r)+\omega^0(s)-\omega^0(r)\big)1_{s\geq r}.
\end{align*}
Denote $\mathcal{T}_{t,T}^0$ the set of $\mathcal{F}^0$-stopping times valued in $[t,T]$. Let $\alpha \in \mathcal{A}$, for $\theta \in \mathcal{T}_{t,T}^0$, denote $\alpha^\theta$ the map
\begin{align*}
    \alpha^\theta:(\Omega^0,\mathcal{F}_\theta^t)\to(\mathcal{A},\mathcal{B}_\mathcal{A}),\\
    \omega^0\mapsto \alpha^{\theta(\omega^0),\omega^0},
\end{align*}
the case for $\alpha\in\mathcal{A}_t$ is defined analogously.
\subsection{Assumptions}
The following assumption will be used throughout our work.
\begin{assumption}[$\mathcal{A}$1]
\label{assumption_A1}
    For the coefficients $g$ and $h = f$, $b$, $\sigma$, $\sigma^{0}$, we assume the following:
    \begin{enumerate}
        \item $g \in L^\infty\Big((\Omega^0,\mathcal{F}_T^0,\mathbb{P}^0);C(\mathbb{R}^d\times\mathcal{P}_2(\mathbb{R}^d))\Big)$. 
        \item For all $a \in A$, $h(\cdot,\cdot,\cdot,a)$ is $\left(\mathcal{F}_t^0\right)_{t\geq 0}$-adapted.
        \item For almost all $(\omega^0,t)\in\Omega^0\times[0,T]$, $h(t,\cdot,\cdot,\cdot)$ is continuous on $\mathbb{R}^d\times\mathcal{P}_2(\mathbb{R}^d)\times A$.
        \item There exists $K>0$ such that for all $x$, $x'\in\mathbb{R}^d$, $t\in [0,T]$ and $\mu,\mu'\in\mathcal{P}_2(\mathbb{R}^d)$, there holds
        \begin{align*}
        \esssup_{\omega^0\in\Omega^0}|g(x,\mu)-g(x',\mu')|+\esssup_{\omega^0\in\Omega^0}\sup_{a\in A}|h(t,x,\mu,a)-h(t,x',\mu',a)|&\leq K(|x-x'|+\mathcal{W}_2(\mu,\mu')),\\
        \esssup_{\omega^0\in\Omega^0}|g(x,\mu)|+\esssup_{\omega^0\in\Omega^0}\sup_{a\in A}|h(t,x,\mu,a)|&\leq K.
        \end{align*}
    \end{enumerate}
\end{assumption}
\begin{assumption}[$\mathcal{A}$2]
\label{assumption_A2}
    The following assumption is needed to simplify technical issue in establishing the existence of our viscosity solution theory.
    \begin{enumerate}
        \item $\sigma^0 := \sigma^0(\omega^0,t,x,\mu)$, i.e., independent of control.
    \end{enumerate}
\end{assumption}
\begin{assumption}[$\mathcal{A}$3]
\label{assumption_A3}
\label{bochner_assumption}
    This following assumption is needed in establishing the uniqueness of our viscosity solution theory.
    \begin{enumerate}
        \item For $\mathbb{P}^0$ a.e. $\omega^0\in\Omega^0$, $f$, $b$ are continuous on $[0,T]\times\mathbb{R}^d\times\mathcal{P}_2(\mathbb{R}^d)\times A$.
        \item $f$, $b\in \mathcal{B}\mathcal{S}^1(C(\mathcal{P}_2(\mathbb{R}^d)\times\mathbb{R}^d))$.
        \item $g:(\Omega^0,\mathcal{F}_T^0,\mathbb{P}^0)\to C(\mathbb{R}^d\times\mathcal{P}_2(\mathbb{R}^d))$ is Bochner measurable.
    \end{enumerate}
\end{assumption}
\noindent As a direct consequence of the above Assumption ($\mathcal{A}$3), the following approximation is immediate and standard.
\begin{lemma}
\label{approximation_markovian}
For each $\varepsilon > 0$, there exists a partition $0 = t_0 <t_1 <\cdot<t_{N-1}<t_N = T$ for some $N>3$ and functions
\begin{align*}
    (g^N, f^N, b^N) \in C^3(\mathcal{P}_2(\mathbb{R}^d);\mathbb{R}^{m\times N}\times\mathbb{R}^d)\times C\Big(A\times\mathcal{P}_2(\mathbb{R}^d);C^3([0,T]\times\mathbb{R}^{m\times N}\times \mathbb{R}^d)\Big)\times\\
    \times C\Big(A\times\mathcal{P}_2(\mathbb{R}^d);C^3([0,T]\times\mathbb{R}^{m\times N}\times\mathbb{R}^d)\Big)
\end{align*}
such that 
\begin{align*}
    g^\varepsilon :=& \esssup_{(x,\rho)\in\mathbb{R}^d\times\mathcal{P}_2(\mathbb{R}^d)}\Big|g^N(W^0_{t_1},\ldots,W^0_{t_N},x,\rho)-g(x,\rho)\Big|,\\
    f_t^\varepsilon :=& \esssup_{(x,\rho,\alpha)\in\mathbb{R}^d\times\mathcal{P}_2(\mathbb{R}^d)\times A}\Big|f_t^N(W^0_{t_1\wedge t},\ldots,W^0_{t_N\wedge t},x,\rho,\alpha)-f_t(x,\rho,\alpha)\Big|,\\
    b_t^\varepsilon :=& \esssup_{(x,\rho,\alpha)\in\mathbb{R}^d\times\mathcal{P}_2(\mathbb{R}^d)\times A}\Big|b_t^N(W^0_{t_1\wedge t},\ldots,W^0_{t_N\wedge t},x,\rho,\alpha)-b_t(x,\rho,\alpha)\Big|
\end{align*}
are $\mathcal{F}^0_t$ adapted with
\begin{align*}
    \|g^\varepsilon\|_{L^{2}((\Omega^0,\mathcal{F}^0_T,\mathbb{P}^0);\mathbb{R})} + \|f^\varepsilon\|_{\mathcal{L}^2(\mathbb{R})} + \|b^\varepsilon\|_{\mathcal{L}^2(\mathbb{R})} < \varepsilon.
\end{align*}
and $g^N$, $f^N$ and $b^N$ are uniformly Lipschitz continuous in the variables $\rho$, $x$ with an identical Lipschitz constant $K$ independent of $N$ and $\varepsilon$.
\end{lemma}

\subsection{The Problem}
The dynamic of our control problem is
\begin{align}
\label{basic_dynamics}
\begin{cases}
    dX_s^{t,\xi,\alpha} =& b_s(X_s^{t,\xi,\alpha},\mathbb{P}_{X_s^{t,\xi,\alpha}}^{W^0},\alpha_s)ds+\sigma_s(X_s^{t,\xi,\alpha},\mathbb{P}_{X_s^{t,\xi,\alpha}}^{W^0},\alpha_s)dW_s \\
    &+ \sigma^0_s(X_s^{t,\xi,\alpha},\mathbb{P}_{X_s^{t,\xi,\alpha}}^{W^0},\alpha_s)dW^0_s,\,\,\,s\in[t,T],\\
    X_t^{t,\xi,\alpha} =& \xi,
\end{cases}
\end{align}
where $\xi \in L^2((\Omega,\mathcal{F}_t,\mathbb{P});\mathbb{R}^d)$. The cost functional on $[0,T]\times L^2((\Omega,\mathcal{F}_t,\mathbb{P});\mathbb{R}^d)\times\mathcal{A}_t\times\Omega^0$ is given by
\begin{align}
    J(t,\xi,\alpha):=& \mathbb{E}\Bigg[\int_t^T f_s(X_s^{t,\xi,\alpha},\mathbb{P}_{X_s^{t,\xi,\alpha}}^{W^0},\alpha_s)+g(X_T^{t,\xi,\alpha},\mathbb{P}_{X_T^{t,\xi,\alpha}}^{W^0})\Big|\mathcal{F}^0_t\Bigg].
\end{align}
The value function on $[0,T]\times L^2((\Omega,\mathcal{F}_t,\mathbb{P});\mathbb{R}^d)\times\Omega^0$ will then be
\begin{align}
    v(t,\xi) = \essinf_{\alpha\in\mathcal{A}_t} J(t,\xi,\alpha) = \essinf_{\alpha\in\mathcal{A}_t} \mathbb{E}\Bigg[\int_t^T f_s(X_s^{t,\xi,\alpha},\mathbb{P}_{X_s^{t,\xi,\alpha}}^{W^0},\alpha_s)+g(X_T^{t,\xi,\alpha},\mathbb{P}_{X_T^{t,\xi,\alpha}}^{W^0})\Big|\mathcal{F}_t^0\Bigg].
\end{align}
Below are some standard properties of the strong solutions for SDEs.
\begin{lemma}
\label{SDE_property}
    Let ($\mathcal{A}$1) hold. Given $\alpha \in \mathcal{A}_r$, for the strong solution of the SDE (\ref{basic_dynamics}), there exists a constant $C>0$ depending on the constant $K$, $p$ and $T$ such that, for any $0\leq r\leq t\leq s\leq T$, and $\xi \in L^p((\Omega_r,\mathcal{F}_r,\mathbb{P});\mathbb{R}^d)$, we have
    \begin{enumerate}[(1)]
        \item the two processes $(X_s^{r,\xi,\alpha})_{t\leq s\leq T}$ and $(X_s^{t,X_t^{r,\xi,\alpha},\alpha})_{t\leq s\leq T}$ are indistinguishable.
        \item $\mathbb{E}_{\mathcal{F}_r^0}\max_{r\leq l \leq T}|X_l^{r,\xi,\alpha}|^p\leq C(1+\mathbb{E}_{\mathcal{F}_r^0}|\xi|^p)$ $\mathbb{P}^0$\text{-}a.e..
        \item $\mathbb{E}_{\mathcal{F}_r^0}|X_s^{r,\xi,\alpha}-X_t^{r,\xi,\alpha}|^p \leq C(1+\mathbb{E}_{\mathcal{F}_r^0}|\xi|^p)(s-t)^{p/2}$ $\mathbb{P}^0$\text{-}a.e..
        \item Given another $\eta\in L^p((\Omega,\mathcal{F}_r,\mathbb{P});\mathbb{R}^d)$,
        \begin{align*}
        \mathbb{E}_{\mathcal{F}_r^0}\max_{r\leq l\leq T}|X_l^{r,\xi,\alpha}-X_l^{r,\eta,\alpha}|^p \leq C\mathbb{E}_{\mathcal{F}_r^0}|\xi-\eta|^p\quad\mathbb{P}^0 \text{-a.e.}.
        \end{align*}
    \end{enumerate}
\end{lemma}
\noindent We also have the following regular properties for the cost functional $J$ and value function $v$. The results are similar to \cite{qiu_stochastic_hjb} Proposition 3.3 and the proof is omitted.
\begin{lemma}
\label{regularity_J_V}
    Let ($\mathcal{A}$1) hold. We have the following regularity properties regarding the cost functional $J$ and the value function $v$:
    \begin{enumerate}[(1)]
        \item For all $(t,\xi,\alpha)\in [0,T]\times L^2((\Omega,\mathcal{F}_t,\mathbb{P});\mathbb{R}^d)\times\mathcal{A}_t$, $(J(s,X_s^{t,\xi,\alpha},\alpha))_{t\leq s\leq T}$ is a continuous process. \label{Lemma_2_3_1}
        \item For all $(t,\xi)\in [0,T]\times L^2((\Omega,\mathcal{F}_t,\mathbb{P});\mathbb{R}^d)$, $(v(s,X_s^{t,\xi,\alpha}))_{t\leq s\leq T}$ is a continuous process.
        \label{Lemma_2_3_2}
        \item There exists $C>0$ depending on $K$ and $T$ such that for any $\alpha\in\mathcal{A}_t$, $\xi$, $\eta\in L^2((\Omega,\mathcal{F}_t,\mathbb{P});\mathbb{R}^d)$,
        \begin{align*}
            |v(t,\xi)-v(t,\eta)|+|J(t,\xi,\alpha)-J(t,\eta,\alpha)|\leq C\mathbb{E}_{\mathcal{F}_r^0}|\xi-\eta|^2\quad\mathbb{P}^0\text{-a.e.}
        \end{align*}
        \item With probability 1, $v(t,\xi)$ and $J(t,\xi,\alpha)$ for each $\alpha\in\mathcal{A}_t$ are continuous on $[0,T]\times L^2((\Omega,\mathcal{F}_t,\mathbb{P});\mathbb{R}^d)$ and
        \begin{align*}
            \esssup_{(t,\xi)\in [0,T]\times L^2(\Omega,\mathcal{F}_t,\mathbb{P};\mathbb{R}^d)} \max\{|v(t,\xi)|,|J(t,\xi,\alpha)|\}\leq K(T+1)\quad\mathbb{P}^0\text{-a.e..} 
        \end{align*}
    \end{enumerate}
\end{lemma}
\noindent We can now define the function $\bar{J}$ on $[0,T]\times L^2((\Omega^1,\mathcal{F}_t^1,\mathbb{P}^1);\mathbb{R}^d)\times \mathcal{A}_t\times\Omega^0$ by
\begin{align*}
    \bar{J}(t,\xi,\alpha) := J(t,\xi,\alpha) .
\end{align*}
We have the following lemma relating $J$ and $\bar{J}$.
\begin{lemma}
\label{nonrandom_random_functional}
Let ($\mathcal{A}$1) hold. For $\xi \in L^2((\Omega,\mathcal{F}_t,\mathbb{P});\mathbb{R}^d)$, we have
\begin{align*}
    J(t,\xi,\alpha)(\omega^0) = \bar{J}(t,\xi(\omega^0),\alpha)(\omega^0),\,\,\,\mathbb{P}^0\,a.e..
\end{align*}    
\end{lemma}
\begin{proof}
First of all we look at simple functions. For $i = 1,\cdots,n$, let $a_i\in\mathbb{R}^d$, $B_i^{0'}\times B_{i}^{1'}\in\mathcal{F}_t^0\times\mathcal{F}_t^1$, $\{B_i^{0'}\times B_{i}^{1'}\}_{i=1}^n$ a disjoint partition of $\Omega^0\times\Omega^1$, and we consider $\xi(\omega^0,\omega^1) = \sum_{i=1}^{n}a_{i}1_{B^{0'}_i}(\omega^0)1_{B^{1'}_i}(\omega^1)$. Note that $\xi$ could also be written as $\xi= \sum_{i=1}^{n}\Big(a_i 1_{B_i^{1'}}(\omega^1)\Big)1_{B_i^{0'}}(\omega^0)=\sum_{j=1}^m\xi_j(\omega^1)1_{B_j^0}(\omega^0)$, where for $j=1,\ldots,m$, $B_j^0\in\mathcal{F}_t^0$, $\{B_j^0\}_{j=1}^m$ is a disjoint partition of $\Omega^0$, $\xi_j\in L^2((\Omega^1,\mathcal{F}^1_t,\mathbb{P}^1);\mathbb{R}^d)$. Solving for $j=1,\ldots,m$,
\begin{align}
\begin{cases}
    dX_s^{t,\xi_j,\alpha} =& b_s(X_s^{t,\xi_j,\alpha},\mathbb{P}_{X_s^{t,\xi_j,\alpha}}^{W^0},\alpha_s)ds+\sigma_s(X_s^{t,\xi_j,\alpha},\mathbb{P}_{X_s^{t,\xi_j,\alpha}}^{W^0},\alpha_s)dW_s \\
    &+ \sigma^0_s(X_s^{t,\xi_j,\alpha},\mathbb{P}_{X_s^{t,\xi_j,\alpha}}^{W^0},\alpha_s)dW^0_s,\,\,\,s\in[t,T],\\
    X_t^{t,\xi_j,\alpha} \,\,\,=&  \xi_j,
\end{cases}
\end{align}
we have 
\begin{align*}
    &\sum_{j=1}^m X_s^{t,\xi_j,\alpha}1_{B_j^0}(\omega^0) \\
    =& \sum_{j=1}^m \xi_j1_{B_j^0}(\omega^0)+ \sum_{j=1}^m \Bigg(\int_t^s b_r(X_r^{t,\xi_j,\alpha},\mathbb{P}_{X_r^{t,\xi_j,\alpha}}^{W^0},\alpha_r)dr\Bigg)1_{B_j^0}(\omega^0)\\
    &+\sum_{j=1}^m \Bigg(\int_t^s\sigma_r(X_r^{t,\xi_j,\alpha},\mathbb{P}_{X_r^{t,\xi_j,\alpha}}^{W^0},\alpha_r)dW_r\Bigg)1_{B_j^0}(\omega^0) + \sum_{j=1}^m\Bigg(\int_t^s\sigma^0_r(X_r^{t,\xi_j,\alpha},\mathbb{P}_{X_r^{t,\xi_j,\alpha}}^{W^0},\alpha_r)dW^0_r\Bigg)1_{B_j^0}(\omega^0)\\
    =& \sum_{j=1}^m \xi_j1_{B_j^0}(\omega^0)+ \int_t^s  b_r\Big(\sum_{j=1}^m X_r^{t,\xi_j,\alpha}1_{B_j^0}, \mathbb{P}_{\sum_{j=1}^m X_r^{t,\xi_j,\alpha}1_{B_j^0}}^{W^0},\alpha_r\Big)dr\\
    &+ \int_t^s\sigma_r\Big(\sum_{j=1}^m X_r^{t,\xi_j,\alpha}1_{B_j^0},\mathbb{P}_{\sum_{j=1}^m X_r^{t,\xi_j,\alpha}1_{B_j^0}}^{W^0},\alpha_r\Big)dW_r + \int_t^s\sigma^0_r\Big(\sum_{j=1}^m X_r^{t,\xi_j,\alpha}1_{B_j^0},\mathbb{P}_{\sum_{j=1}^m X_r^{t,\xi_j,\alpha}1_{B_j^0}}^{W^0},\alpha_r\Big)dW^0_r.
\end{align*}
Thus by uniqueness of solution we have
\begin{align*}
    \sum_{j=1}^m X_s^{t,\xi_j,\alpha}1_{B_j^0}(\omega^0)
\end{align*}
solves
\begin{align}
\begin{cases}
    dX_s^{t,\xi,\alpha} = &b_s(X_s^{t,\xi,\alpha},\mathbb{P}_{X_s^{t,\xi,\alpha}}^{W^0},\alpha_s)ds+\sigma_s(X_s^{t,\xi,\alpha},\mathbb{P}_{X_s^{t,\xi,\alpha}}^{W^0},\alpha_s)dW_s \\
    &+ \sigma^0_s(X_s^{t,\xi,\alpha},\mathbb{P}_{X_s^{t,\xi,\alpha}}^{W^0},\alpha_s)dW^0_s,\,\,\,s\in[t,T],\\
    X_t^{t,\xi,\alpha} \,\,\,= &\xi.
\end{cases}
\end{align}
Therefore it holds that
\begin{align*}
    J(t,\xi,\alpha) 
    =& \mathbb{E}\Bigg[\int_t^T f_s(X_s^{t,\xi,\alpha},\mathbb{P}_{X_s^{t,\xi,\alpha}}^{W^0},\alpha_s)+g(X_T^{t,\xi,\alpha},\mathbb{P}_{X_T^{t,\xi,\alpha}}^{W^0})\Big|\mathcal{F}_t^0\Bigg]\\
    =&\mathbb{E}\Bigg[\int_t^T f_s(\sum_{j=1}^m X_s^{t,\xi_j,\alpha}1_{B_j^0}(\omega^0),\mathbb{P}_{\sum_{j=1}^m X_s^{t,\xi_j,\alpha}1_{B_j^0}(\omega^0)}^{W^0},\alpha_s)\\&+g(\sum_{j=1}^m X_T^{t,\xi_j,\alpha}1_{B_j^0}(\omega^0),\mathbb{P}_{\sum_{j=1}^m X_T^{t,\xi_j,\alpha}1_{B_j^0}(\omega^0)}^{W^0})\Big|\mathcal{F}_t^0\Bigg]\\
    =&\sum_{j=1}^m\mathbb{E}\Bigg[\int_t^T f_s( X_s^{t,\xi_j,\alpha},\mathbb{P}_{ X_s^{t,\xi_j,\alpha}}^{W^0},\alpha_s)+g(X_T^{t,\xi_j,\alpha},\mathbb{P}_{X_T^{t,\xi_j,\alpha}}^{W^0})\Big|\mathcal{F}_t^0\Bigg]1_{B_j^0}(\omega^0)\\
    =&\sum_{j=1}^m \bar{J}(t,\xi_j,\alpha)1_{B_j^0}(\omega^0)\\
    =&\bar{J}(t,\xi(\omega^0),\alpha).
\end{align*}
Now, for general $\xi\in L^2((\Omega,\mathcal{F}_t,\mathbb{P});\mathbb{R}^d)$, there exists a sequence of simple functions $\xi^n(\omega^0,\omega^1)$ converging to $\xi(\omega^0,\omega^1)$ pointwisely, and $\xi^n$ can be written in the form $\xi^n = \sum_{j=1}^m\xi_j^n(\omega^1)1_{B_j^{n;0}}(\omega^0)$. By Lemma \ref{regularity_J_V} we have 
\begin{align*}
    &\mathbb{E}^0\Big[|J(t,\xi,\alpha)-\bar{J}(t,\xi(\omega^0),\alpha)|^2\Big]\\
    \leq&C\mathbb{E}^0\Big[|J(t,\xi,\alpha)-J(t,\xi^n,\alpha)|^2\Big]+C\mathbb{E}^0\Big[J(t,\xi^n,\alpha)-\bar{J}(t,\xi(\omega^0),\alpha)|^2\Big]\\
    =&C\mathbb{E}^0\Big[|J(t,\xi,\alpha)-J(t,\xi^n,\alpha)|^2\Big]+C\mathbb{E}^0\Big[\bar{J}(t,\xi^n,\alpha)-\bar{J}(t,\xi(\omega^0),\alpha)|^2\Big]\\
    \leq &C\mathbb{E}[|\xi-\xi^n|^2],
\end{align*}
Taking $n\to\infty$ concludes the proof.
\end{proof}
\noindent Define the function $\bar{v}$ on $[0,T]\times L^2((\Omega^1,\mathcal{F}_t^1,\mathbb{P}^1);\mathbb{R}^d)\times\Omega^0$ by
\begin{align*}
    \bar{v}(t,\xi):=\essinf_{\alpha\in\mathcal{A}_t}\bar{J}(t,\xi,\alpha),
\end{align*}
we have a similar relation between $v$ and $\bar{v}$ as $J$ and $\bar{J}$.
\begin{lemma}
\label{random_nonrandom_value}
Let ($\mathcal{A}$1) hold. Let $\xi\in L^2((\Omega,\mathcal{F}_t,\mathbb{P});\mathbb{R}^d)$, we have that
\begin{align*}
    v(t,\xi)(\omega^0) = \bar{v}(t,\xi(\omega^0))(\omega^0),\,\,\,\mathbb{P}^0\,a.e..
\end{align*}
\end{lemma}
\begin{proof}
    Again, we start with simple functions. From the beginning of the proof of Lemma \ref{nonrandom_random_functional} we can assume $\xi(\omega^0,\omega^1) =\sum_{j=1}^m\xi_j(\omega^1)1_{B_j^0}(\omega^0)$, where for $j=1,\ldots,m$, $B_j^0\in\mathcal{F}_t^0$, $\{B_j^0\}_{j=1}^m$ is a disjoint partition of $\Omega^0$, $\xi_j\in L^2((\Omega^1,\mathcal{F}^1_t,\mathbb{P}^1);\mathbb{R}^d)$, and it holds that
\begin{align*}
     \sum_{j=1}^m X_s^{t,\xi_j,\alpha}1_{B_j^0}(\omega^0)
\end{align*}
solves the SDE 
\begin{align}
\label{random_nonrandom_eq2}
\begin{cases}
    dX_s^{t,\xi,\alpha} = &b_s(X_s^{t,\xi,\alpha},\mathbb{P}_{X_s^{t,\xi,\alpha}}^{W^0},\alpha_s)ds+\sigma_s(X_s^{t,\xi,\alpha},\mathbb{P}_{X_s^{t,\xi,\alpha}}^{W^0},\alpha_s)dW_s \\
    &+ \sigma^0_s(X_s^{t,\xi,\alpha},\mathbb{P}_{X_s^{t,\xi,\alpha}}^{W^0},\alpha_s)dW^0_s,\,\,\,s\in[t,T],\\
    X_t^{t,\xi,\alpha} = \xi.
\end{cases}
\end{align}
Thus,
\begin{align*}
    J(t,\xi,\alpha)
    = &\,\mathbb{E}\Bigg[\int_t^T f_s(X_s^{t,\xi,\alpha},\mathbb{P}^{W^0}_{X_s^{t,\xi,\alpha}},\alpha_s)+g(X_T^{t,\xi,\alpha},\mathbb{P}^{W^0}_{X_T^{t,\xi,\alpha}})\Big|\mathcal{F}_t^0\Bigg]\\
    =&\sum_{j=1}^m\mathbb{E}\Bigg[\int_t^T f_s(X_s^{t,\xi_j,\alpha},\mathbb{P}^{W^0}_{X_s^{t,\xi_j,\alpha}},\alpha_s)+g(X_T^{t,\xi_j,\alpha},\mathbb{P}^{W^0}_{X_T^{t,\xi_j,\alpha}})\Big|\mathcal{F}_t^0\Bigg]1_{B_j^0}(\omega^0)\\
    \geq &\sum_{j=1}^m \bar{v}(t,\xi_j)1_{B_j^0}(\omega^0)\\
    =&  \bar{v}(t,\sum_{j=1}^m\xi_j 1_{B_j^0}(\omega^0))
\end{align*}
By arbitrariness of $\alpha$ we have
\begin{align*}
    v(t,\xi)\geq \bar{v}(t,\sum_{j=1}^m\xi_j 1_{B_j^0}(\omega^0)).
\end{align*}
On the other hand, for each $j= 1,\ldots,m$, there exists $\alpha_{ji}\in\mathcal{A}_t$, $i\in\mathbb{N}$ such that 
\begin{align*}
    \lim_{i\to\infty} J(t,\xi_j,\alpha_{ji})= \bar{v}(t,\xi_j),\quad\mathbb{P}^0\text{-}a.e..
\end{align*}
Define $\hat{\alpha}_i := \sum_{j=1}^m \alpha_{ji}1_{B_j^0}$. Then
\begin{align*}
    v(t,\xi) 
    \leq&J(t,\xi,\hat{\alpha}_i)
    = \sum_{j=1}^m\mathbb{E}\Bigg[\int_t^T f_s(X_s^{t,\xi_j,\alpha_{ji}},\mathbb{P}^{W^0}_{X_s^{t,\xi_j,\alpha_{ji}}},\alpha_{ji,s})+g(X_T^{t,\xi_j,\alpha_{ji}},\mathbb{P}^{W^0}_{X_T^{t,\xi_j,\alpha_{ji}}})\Big|\mathcal{F}_t^0\Bigg]1_{B_j^0}(\omega^0)\\
    \to&\sum_{j=1}^m\bar{v}(t,\xi_j)1_{B_j^0},\text{ as }i\to\infty.
\end{align*}
Thus we have
\begin{align*}
    v(t,\xi)\leq\bar{v}(t,\sum_{j=1}^m\xi_j 1_{B_j^0}(\omega^0)).
\end{align*}
Finally, for general $\xi\in L^2((\Omega,\mathcal{F},\mathbb{P});\mathbb{R}^d)$, the proof is essentially the same as the last part in Lemma \ref{nonrandom_random_functional}.  
\end{proof}
\begin{lemma}
     Assume ($\mathcal{A}$1) holds. Letting $\theta\in\mathcal{T}_{t,T}^0$, for all $\xi\in L^2((\Omega,\mathcal{F}_\theta,\mathbb{P});\mathbb{R}^d)$, we have
    \begin{align}
    \label{pseudo_markov}
        J(\theta,X_\theta^{t,\xi,\alpha},\alpha^\theta)=&\, \mathbb{E}\Bigg[\int_\theta^T f_s(X_s^{t,\xi,\alpha},\mathbb{P}_{X_s^{t,\xi,\alpha}}^{W^0},\alpha_s)+g(X_T^{t,\xi,\alpha},\mathbb{P}_{X_T^{t,\xi,\alpha}}^{W^0})\Big|\mathcal{F}^0_\theta\Bigg],\,\,\,\mathbb{P}^0\text{-}a.e..
    \end{align}
\end{lemma}
\begin{proof}
    Recall that from the uniqueness of SDE we have the flow property:
\begin{align*}
    X_s^{t,\xi,\alpha}(\omega^0,\omega^1) = X_s^{\theta,X_{\theta}^{t,\xi,\alpha},\alpha^\theta}(\omega^0,\omega^1).
\end{align*}
Assume $\theta = \sum_{i=1}^n t_i 1_{B^0_i}$, where for $i=1,\cdots,n$, $t\leq t_i\leq T$, and $\{B^0_i\}_{i=1}^n$ is a disjoint partition of $\Omega^0$, thus $B_i^0 = \{\theta = t_i\} $, and $B_i^0\in\mathcal{F}_{t_i}^0$. Let $A\in\mathcal{F}_\theta^0$. Then 
\begin{align*}
    &\int_A \mathbb{E}\Bigg[\int_\theta^T f_s(X_s^{t,\xi,\alpha},\mathbb{P}_{X_s^{t,\xi,\alpha}}^{W^0},\alpha_s)+g(X_T^{t,\xi,\alpha},\mathbb{P}_{X_T^{t,\xi,\alpha}}^{W^0})\Big|\mathcal{F}^0_\theta\Bigg] d\mathbb{P}^0\\
    =&\int_A \Bigg(\int_\theta^T f_s(X_s^{t,\xi,\alpha},\mathbb{P}_{X_s^{t,\xi,\alpha}}^{W^0},\alpha_s)+g(X_T^{t,\xi,\alpha},\mathbb{P}_{X_T^{t,\xi,\alpha}}^{W^0}) \Bigg)\sum_{i=1}^n 1_{\{\theta = t_i\}}d\mathbb{P}^0\\
    =&\int_{\Omega^0} \Bigg(\int_\theta^T f_s(X_s^{t,\xi,\alpha},\mathbb{P}_{X_s^{t,\xi,\alpha}}^{W^0},\alpha_s)+g(X_T^{t,\xi,\alpha},\mathbb{P}_{X_T^{t,\xi,\alpha}}^{W^0}) \Bigg)\sum_{i=1}^n 1_{\{\theta = t_i\}\cap A}d\mathbb{P}^0\\
    =&\sum_{i=1}^n\int_{\Omega^0} \mathbb{E}\Bigg[\int_\theta^T f_s(X_s^{t,\xi,\alpha},\mathbb{P}_{X_s^{t,\xi,\alpha}}^{W^0},\alpha_s)+g(X_T^{t,\xi,\alpha},\mathbb{P}_{X_T^{t,\xi,\alpha}}^{W^0})\Big|\mathcal{F}^0_{t_i}\Bigg] 1_{\{\theta = t_i\}\cap A} d\mathbb{P}^0\\
    =&\sum_{i=1}^n\int_{\Omega^0} \mathbb{E}\Bigg[\int_{t_i}^T f_s(X_s^{t_i,X^{t,\xi,\alpha}_{t_i},\alpha^{t_i}},\mathbb{P}_{X_s^{t_i,X^{t,\xi,\alpha}_{t_i},\alpha^{t_i}}}^{W^0},\alpha^{t_i}_s)+g(X_T^{t_i,X^{t,\xi,\alpha}_{t_i},\alpha^{t_i}},\mathbb{P}_{X_T^{t_i,X^{t,\xi,\alpha}_{t_i},\alpha^{t_i}}}^{W^0})\Big|\mathcal{F}^0_{t_i}\Bigg] 1_{\{\theta = t_i\}\cap A} d\mathbb{P}^0\\
    =&\sum_{i=1}^n\int_{\Omega^0} J(t_i,X_{t_i}^{t,\xi,\alpha},\alpha^{t_i}) 1_{\{\theta = t_i\}\cap A}d\mathbb{P}^0\\
    =&\int_A J(\theta,X_\theta^{t,\xi,\alpha},\alpha^\theta) d\mathbb{P}^0.
\end{align*}
For the general case, let $\theta_i$ be an increasing sequence of discrete stopping times such that $\lim_i \theta_i = \theta$. By Theorem 4.6.10 in \cite{durrett2019probability} and Chapter 1, Exercises 24 and 27 in \cite{protter2005stochastic}, it holds that
\begin{align*}
    &\mathbb{E}\Bigg[\int_{\theta_i}^T f_s(X_s^{t,\xi,\alpha},\mathbb{P}_{X_s^{t,\xi,\alpha}}^{W^0},\alpha_s)+g(X_T^{t,\xi,\alpha},\mathbb{P}_{X_T^{t,\xi,\alpha}}^{W^0})\Big|\mathcal{F}_{\theta_i}^0\Bigg]\\
    \to&\mathbb{E}\Bigg[\int_{\theta}^T f_s(X_s^{t,\xi,\alpha},\mathbb{P}_{X_s^{t,\xi,\alpha}}^{W^0},\alpha_s)+g(X_T^{t,\xi,\alpha},\mathbb{P}_{X_T^{t,\xi,\alpha}}^{W^0})\Big|\mathcal{F}_{\theta}^0\Bigg]\quad\mathbb{P}^0\text{-}a.e.,
\end{align*}
which, together with Lemma \ref{regularity_J_V} \ref{Lemma_2_3_1}, implies that
\begin{align*}
    &J(\theta,X_\theta^{t,\xi,\alpha},\alpha^\theta)- \mathbb{E}\Bigg[\int_\theta^T f_s(X_s^{t,\xi,\alpha},\mathbb{P}_{X_s^{t,\xi,\alpha}}^{W^0},\alpha_s)+g(X_T^{t,\xi,\alpha},\mathbb{P}_{X_T^{t,\xi,\alpha}}^{W^0})\Big|\mathcal{F}^0_\theta\Bigg]\\
    =&J(\theta,X_\theta^{t,\xi,\alpha},\alpha^\theta)- J(\theta_i,X_{\theta_i}^{t,\xi,\alpha},\alpha^{\theta_i})\\
    &+J(\theta_i,X_{\theta_i}^{t,\xi,\alpha},\alpha^{\theta_i})-\mathbb{E}\Bigg[\int_\theta^T f_s(X_s^{t,\xi,\alpha},\mathbb{P}_{X_s^{t,\xi,\alpha}}^{W^0},\alpha_s)+g(X_T^{t,\xi,\alpha},\mathbb{P}_{X_T^{t,\xi,\alpha}}^{W^0})\Big|\mathcal{F}^0_\theta\Bigg]\\
    =&J(\theta,X_\theta^{t,\xi,\alpha},\alpha^\theta)- J(\theta_i,X_{\theta_i}^{t,\xi,\alpha},\alpha^{\theta_i})\\
    &+\mathbb{E}\Bigg[\int_{\theta_i}^T f_s(X_s^{t,\xi,\alpha},\mathbb{P}_{X_s^{t,\xi,\alpha}}^{W^0},\alpha_s)+g(X_T^{t,\xi,\alpha},\mathbb{P}_{X_T^{t,\xi,\alpha}}^{W^0})\Big|\mathcal{F}_{\theta_i}^0\Bigg]\\
    &-\mathbb{E}\Bigg[\int_\theta^T f_s(X_s^{t,\xi,\alpha},\mathbb{P}_{X_s^{t,\xi,\alpha}}^{W^0},\alpha_s)+g(X_T^{t,\xi,\alpha},\mathbb{P}_{X_T^{t,\xi,\alpha}}^{W^0})\Big|\mathcal{F}^0_\theta\Bigg]\\
    \to &0\,\,\,\mathbb{P}^0\,a.e.\text{ as }i\to \infty. 
\end{align*}
\end{proof}
\begin{theorem}
\label{DPP_THM}
(Dynamic Programming Principle). Let ($\mathcal{A}$1) hold. We have for all $(t,\xi)\in [0,T]\times L^2((\Omega,\mathcal{F}_t,\mathbb{P});\mathbb{R}^d)$,
\begin{align*}
    v(t,\xi) = &\essinf_{\alpha\in\mathcal{A}_t}\essinf_{\theta\in\mathcal{T}_{t,T}^0}\mathbb{E}\Bigg[\int_t^\theta f_s(X_s^{t,\xi,\alpha},\mathbb{P}^{W^0}_{X_s^{t,\xi,\alpha}},\alpha_s)+v(\theta,X_{\theta}^{t,\xi,\alpha}) \Big|\mathcal{F}_t^0\Bigg]\\
    =&\essinf_{\alpha\in\mathcal{A}_t}\esssup_{\theta\in\mathcal{T}_{t,T}^0}\mathbb{E}\Bigg[\int_t^\theta f_s(X_s^{t,\xi,\alpha},\mathbb{P}^{W^0}_{X_s^{t,\xi,\alpha}},\alpha_s)+v(\theta,X_{\theta}^{t,\xi,\alpha}) \Big|\mathcal{F}_t^0\Bigg].
\end{align*}
\end{theorem}
\begin{proof}
Let $(t,\xi,\alpha)\in [0,T]\times L^2((\Omega;\mathcal{F}_t,\mathbb{P});\mathbb{R}^d)\times\mathcal{A}_t$. From (\ref{pseudo_markov}) it follows that
\begin{align*}
    J(t,\xi,\alpha)
    =&\mathbb{E}\Bigg[\int_t^\theta f(X_s^{t,\xi,\alpha},\mathbb{P}^{W^0}_{X_s^{t,\xi,\alpha}},\alpha_s)+\mathbb{E}\Bigg[\int_\theta^T f(X_s^{t,\xi,\alpha},\mathbb{P}^{W^0}_{X_s^{t,\xi,\alpha}},\alpha_s) + g(X_T^{t,\xi,\alpha},\mathbb{P}^{W^0}_{X_T^{t,\xi,\alpha}})\Big|\mathcal{F}_\theta^0\Bigg]\Big|\mathcal{F}_t^0\Bigg]\\
    =&\mathbb{E}\Bigg[\int_t^\theta f(X_s^{t,\xi,\alpha},\mathbb{P}^{W^0}_{X_s^{t,\xi,\alpha}},\alpha_s)+J(\theta,X^{t,\xi,\alpha}_\theta,\alpha^\theta)\Big|\mathcal{F}_t^0\Bigg].
    \end{align*}
    Now assume that $\theta = \sum_{i=1}^n t_i 1_{B_i^0}$, where for $i=1,\cdots,n$, $t\leq t_i\leq T$, and $\{B^0_i\}_{i=1}^n$ is a disjoint partition of $\Omega^0$, then
    \begin{align*}
        &\mathbb{E}\Bigg[\int_t^\theta f(X_s^{t,\xi,\alpha},\mathbb{P}^{W^0}_{X_s^{t,\xi,\alpha}},\alpha_s)+J(\theta,X^{t,\xi,\alpha}_\theta,\alpha^\theta)\Big|\mathcal{F}_t^0\Bigg]\\
        =&\sum_{i=1}^n\mathbb{E}\Bigg[\Bigg[\int_t^{t_i} f(X_s^{t,\xi,\alpha},\mathbb{P}^{W^0}_{X_s^{t,\xi,\alpha}},\alpha_s)+J(t_i,X^{t,\xi,\alpha}_{t_i},\alpha^{t_i})\Bigg]1_{B_i^0}\Big|\mathcal{F}_t^0\Bigg]\\
        \geq&\sum_{i=1}^n\mathbb{E}\Bigg[\Bigg[\int_t^{t_i} f(X_s^{t,\xi,\alpha},\mathbb{P}^{W^0}_{X_s^{t,\xi,\alpha}},\alpha_s)+v(t_i,X^{t,\xi,\alpha}_{t_i})\Bigg]1_{B_i^0}\Big|\mathcal{F}_t^0\Bigg]\\
        =&\mathbb{E}\Bigg[\int_t^\theta f(X_s^{t,\xi,\alpha},\mathbb{P}^{W^0}_{X_s^{t,\xi,\alpha}},\alpha_s)+v(\theta,X^{t,\xi,\alpha}_{\theta})\Big|\mathcal{F}_t^0\Bigg].
    \end{align*}
    Thus if $\theta$ is simple, 
    \begin{align*}
    J(t,\xi,\alpha)\geq \mathbb{E}\Bigg[\int_t^\theta f(X_s^{t,\xi,\alpha},\mathbb{P}^{W^0}_{X_s^{t,\xi,\alpha}},\alpha_s)+v(\theta,X^{t,\xi,\alpha}_{\theta})\Big|\mathcal{F}_t^0\Bigg].
    \end{align*}
    Now, let $\theta\in\mathcal{T}_{t,T}^0$ and $\theta_i\to\theta$ a simple increasing approximation of stopping times of $\theta$, then we have
    \begin{align*}
        J(t,\xi,\alpha)\geq \mathbb{E}\Bigg[\int_t^{\theta_i} f(X_s^{t,\xi,\alpha},\mathbb{P}^{W^0}_{X_s^{t,\xi,\alpha}},\alpha_s)+v(\theta_i,X^{t,\xi,\alpha}_{\theta_i})\Big|\mathcal{F}_t^0\Bigg],\text{ for every }i\in\mathbb{N}.
    \end{align*}
    Take limit on the right hand side and from Lemma \ref{regularity_J_V} \ref{Lemma_2_3_2} we have
    \begin{align*}
        J(t,\xi,\alpha)\geq \mathbb{E}\Bigg[\int_t^{\theta} f(X_s^{t,\xi,\alpha},\mathbb{P}^{W^0}_{X_s^{t,\xi,\alpha}},\alpha_s)+v(\theta,X^{t,\xi,\alpha}_{\theta})\Big|\mathcal{F}_t^0\Bigg].
    \end{align*}
    Since $\theta\in\mathcal{T}_{t,T}^0$ and $\alpha\in\mathcal{A}_t$ are arbitrary, we have
    \begin{align*}
        v(t,\xi)\geq\essinf_{\alpha\in\mathcal{A}_t}\esssup_{\theta\in\mathcal{T}_{t,T}^0} \mathbb{E}\Bigg[\int_t^{\theta} f(X_s^{t,\xi,\alpha},\mathbb{P}^{W^0}_{X_s^{t,\xi,\alpha}},\alpha_s)+v(\theta,X^{t,\xi,\alpha}_{\theta})\Big|\mathcal{F}_t^0\Bigg].
    \end{align*}
    For the other side of inequality, again assume that $\theta = \sum_{i=1}^n t_i 1_{B_i^0}$, where for $i=1,\cdots,n$, $t\leq t_i\leq T$, and $\{B^0_i\}_{i=1}^n$ is a disjoint partition of $\Omega^0$. Let $\alpha\in\mathcal{A}_t$ be arbitrary. For each $t_i$, there exists $\alpha_{ij}\in\mathcal{A}_{t_i}$, $j\in\mathbb{N}$ such that
\begin{align*}
    \lim_{j\to\infty} J(t_i,X_{t_i}^{t,\xi,\alpha},\alpha_{ij}) = v(t_i,X_{t_i}^{t,\xi,\alpha}),\quad\mathbb{P}^0\text{-}a.e..
\end{align*}
Define $\alpha^*_{j;s} :=  \alpha_s 1_{s <\theta} + \sum_{i=1}^n \alpha_{ij;s}1_{s\geq t_i}1_{B_i^0}$.  Then we have
\begin{align*}
    v(t,\xi)
    \leq&J(t,\xi,\alpha^*_j)
    = \sum_{i=1}^n\mathbb{E}\Bigg[\Bigg[\int_t^{t_i} f(X_s^{t,\xi,\alpha},\mathbb{P}^{W^0}_{X_s^{t,\xi,\alpha}},\alpha_s)+J(t_i,X^{t,\xi,\alpha}_{t_i},\alpha_{ij})\Bigg]1_{B_i^0}\Big|\mathcal{F}_t^0\Bigg]\\
    \to&\sum_{i=1}^n\mathbb{E}\Bigg[\Bigg[\int_t^{t_i} f(X_s^{t,\xi,\alpha},\mathbb{P}^{W^0}_{X_s^{t,\xi,\alpha}},\alpha_s)+v(t_i,X^{t,\xi,\alpha}_{t_i})\Bigg]1_{B_i^0}\Big|\mathcal{F}_t^0\Bigg]\\
    =&\mathbb{E}\Bigg[\int_t^{\theta} f(X_s^{t,\xi,\alpha},\mathbb{P}^{W^0}_{X_s^{t,\xi,\alpha}},\alpha_s)+v(\theta,X^{t,\xi,\alpha}_{\theta})\Big|\mathcal{F}_t^0\Bigg].
\end{align*}
Again, using increasing simple function approximation to stopping times and from Lemma \ref{regularity_J_V}, we have for any $\theta\in\mathcal{T}_{t,T}^0$,
\begin{align*}
    v(t,\xi)\leq \mathbb{E}\Bigg[\int_t^{\theta} f(X_s^{t,\xi,\alpha},\mathbb{P}^{W^0}_{X_s^{t,\xi,\alpha}},\alpha_s)+v(\theta,X^{t,\xi,\alpha}_{\theta})\Big|\mathcal{F}_t^0\Bigg],
\end{align*}
and by arbitrariness of $\alpha\in\mathcal{A}_t$, $\theta\in\mathcal{T}_{t,T}^0$, we have
\begin{align*}
    v(t,\xi)\leq \essinf_{\alpha\in\mathcal{A}_t}\essinf_{\theta\in\mathcal{T}_{t,T}^0}\mathbb{E}\Bigg[\int_t^{\theta} f(X_s^{t,\xi,\alpha},\mathbb{P}^{W^0}_{X_s^{t,\xi,\alpha}},\alpha_s)+v(\theta,X^{t,\xi,\alpha}_{\theta})\Big|\mathcal{F}_t^0\Bigg],
\end{align*}
and we conclude the required equality.
\end{proof}
\subsection{Conditional Law Invariant Interpretation}
\begin{theorem}
\label{law_invariance}
    Let ($\mathcal{A}$1) hold. Let $\xi$, $\eta\in L^2((\Omega,\mathcal{F}_t,\mathbb{P});\mathbb{R}^d)$. If for $\mathbb{P}^0$-a.e. $\omega^0$, $\mathcal{L}(\xi(\omega^0,\cdot)) = \mathcal{L}(\eta(\omega^0,\cdot))$, then
    \begin{align*}
        v(t,\xi) = v(t,\eta)\quad\mathbb{P}^0\text{-}a.e..
    \end{align*}
\end{theorem}
\begin{proof}
  The proof is inspired by \cite[Theorem 3.6]{cosso_optimal_2022}. First of all, we suppose that both $\xi$, $\eta$ are discrete. By Lemma \ref{independent_rv_uniform_distribution}, there exists $U_\xi$ and $U_\eta$, which are $\mathcal{F}_t$ measurable, and when given $\mathcal{F}_t^0$ (which is characterized by the random variable $\zeta$ in Definition (\ref{zeta})), they are independent of $\xi$ and $\eta$ respectively, having the uniform distribution. By Theorem \ref{new_control}, there exists $b:[0,T]\times\Omega^0\times\Omega^1\times\mathbb{R}^d\times [0,1]\to A$, measurable with respect to $Prog((\mathcal{F}^{0} \vee \hat{\mathcal{F}}^{1;t}))\otimes\mathcal{B}(\mathbb{R}^d)\otimes\mathcal{B}([0,1])$ such that $(b_s(\xi,U_\xi))_{s\in[0,T]}$ is $\mathbb{F}^t$ progressively measurable, and
\begin{align*}
    &\mathcal{L}\Big(\xi,(\alpha_s)_{s\in[0,T]},(W_s^0)_{s\in [0,T]},(W_{s\vee t}-W_t)_{s\in [0,T]} \Big)\Bigg|_\zeta\\
    =&\mathcal{L}\Big(\xi,(b_s(\xi,U_\xi))_{s\in[0,T]},(W_s^0)_{s\in [0,T]},(W_{s\vee t}-W_t)_{s\in [0,T]} \Big)\Bigg|_\zeta\quad\mathbb{P}^0\text{-}a.e..
\end{align*}
Note that when $\zeta$ is given,
\begin{align*}
    &\mathcal{L}\Big(\xi,(b_s(\xi,U_\xi))_{s\in[0,T]},(W_s^0)_{s\in [0,T]},(W_{s\vee t}-W_t)_{s\in [0,T]} \Big)\Bigg|_\zeta\\
    =&\mathcal{L}\Big(\xi,(b_s(\xi,U_\xi))_{s\in[0,T]},(W_s^0)_{s\in [0,t]},(W_s^0-W_t^0)_{s\in [t,T]},(W_{s\vee t}-W_t)_{s\in [0,T]} \Big)\Bigg|_\zeta\\
    =&\mathcal{L}\Big(\eta,(b_s(\eta,U_\eta))_{s\in[0,T]},(W_s^0)_{s\in [0,t]},(W_s^0-W_t^0)_{s\in [t,T]},(W_{s\vee t}-W_t)_{s\in [0,T]} \Big)\Bigg|_\zeta\\
    =&\mathcal{L}\Big(\eta,(\beta_s)_{s\in[0,T]},(W_s^0)_{s\in [0,t]},(W_s^0-W_t^0)_{s\in [t,T]},(W_{s\vee t}-W_t)_{s\in [0,T]} \Big)\Bigg|_\zeta, 
\end{align*}
where $\beta_s:= b_s(\eta,U_\eta)$, and the second last equality follows from the fact that $\xi$, $U_\xi$ are $\mathcal{F}_t$-measurable, and therefore independent of $W_s^0 - W_t^0$, $W_s - W_t$, $s\geq t$ when given $\mathcal{F}_t^0$. By Lemma \ref{X_equal_in_law}, we have \begin{align*}
    \mathcal{L}\Big((X^{t,\xi,\alpha}_s)_{s\in[t,T]},(\alpha_s)_{s\in[0,T]},(W_s^0)_{s\in [0,T]}\Big)\Bigg|_\zeta = \mathcal{L}\Big((X^{t,\eta,\beta}_s)_{s\in[t,T]},(\beta_s)_{s\in[0,T]},(W_s^0)_{s\in [0,T]}\Big)\Bigg|_\zeta\,\,\mathbb{P}^0\text{-}a.e..
\end{align*}
Note that this also implies $\mathbb{P}^{W^0}_{X^{t,\xi,\alpha}_s}= \mathbb{P}^{W^0}_{X^{t,\eta,\beta}_s}$ $\mathbb{P}^0$-a.e., $\forall s\in [t,T]$. Therefore we have
\begin{align*}
    J(t,\xi,\alpha) = J(t,\eta,\beta)\quad\mathbb{P}^0\text{-}a.e..
\end{align*}
As a consequence we have that $v(t,\xi) = v(t,\eta)$. The general case follows from standard continuity arguments and the fact that for $\mathbb{P}^0$-a.e. $\omega^0$, $v(t,\cdot)$ is continuous.
\end{proof}
\noindent As we have assumed that $\mathcal{G}$ is rich enough to support all laws in $\mathcal{P}_2(\mathbb{R}^d)$, by leveraging the results obtained in the above theorem and Lemma \ref{random_nonrandom_value}, albeit with a slight abuse of notation, we can define the following function
\begin{align*}
    v:[0,T]\times\mathcal{P}_2(\mathbb{R}^d)\times\Omega^0\to\mathbb{R},\\
    v(t,\mu) := v(t,\xi),
\end{align*}
for all $\xi \in L^2((\Omega^1,\mathcal{F}_t^1,\mathbb{P}^1);\mathbb{R}^d)$ such that $\mathcal{L}(\xi):= \mu$. Immediately we have the following corollary from Theorem \ref{DPP_THM}:
\begin{corollary}
(Dynamic Programming Principle). Let ($\mathcal{A}$1) hold. For all $(t,\mu)\in [0,T]\times \mathcal{P}_2(\mathbb{R}^d)$,
\begin{align*}
    v(t,\mu) = &\essinf_{\alpha\in\mathcal{A}}\essinf_{\theta\in\mathcal{T}_{t,T}^0}\mathbb{E}\Bigg[\int_t^\theta f_s(X_s^{t,\xi,\alpha},\mathbb{P}^{W^0}_{X_s^{t,\xi,\alpha}},\alpha_s)+v(\theta,\mathbb{P}_{X_{\theta}^{t,\xi,\alpha}}^{W^0}) \Big|\mathcal{F}_t^0\Bigg]\\
    =&\essinf_{\alpha\in\mathcal{A}}\esssup_{\theta\in\mathcal{T}_{t,T}^0}\mathbb{E}\Bigg[\int_t^\theta f_s(X_s^{t,\xi,\alpha},\mathbb{P}^{W^0}_{X_s^{t,\xi,\alpha}},\alpha_s)+v(\theta,\mathbb{P}_{X_{\theta}^{t,\xi,\alpha}}^{W^0})) \Big|\mathcal{F}_t^0\Bigg],
\end{align*}
for all $\xi \in L^2(\mathcal{G};\mathbb{R}^d)$, $\mathcal{L}(\xi) = \mu$.
\end{corollary}
Moreover, for all $(t,\xi,\alpha)\in [0,T]\times L^2((\Omega,\mathcal{F}_t,\mathbb{P});\mathbb{R}^d)\times\mathcal{A}_t$, by solving the dynamics equation (\ref{basic_dynamics}), we obtain a process $(X_s^{t,\xi,\alpha})_{t\leq s\leq T}$. By taking conditional expectation with respect to $\Omega^0$, this process can actually be interpreted as a $\mathcal{P}_2(\mathbb{R}^d)$-valued process, and shares the following property from \cite{Pham_Wei_Dynamic_Programming}:
\begin{lemma}
\label{pham_wei_rho_process}
     Let ($\mathcal{A}$1) hold. For any $t \in[0, T], \mu \in \mathcal{P}_2(\mathbb{R}^d), \alpha \in \mathcal{A}$, the relation given by
$$
\rho_s^{t, \mu, \alpha}:=\mathbb{P}_{X_s^{t, \xi, \alpha}}^{W^0}, \quad t \leq s \leq T, \text { for } \xi \in L^2(\mathcal{F}_t ; \mathbb{R}^d) \text { s.t. } \mathbb{P}_{\xi}^{W^0}=\mu,
$$
defines a square integrable $\mathbb{F}^{W^0}$-progressive continuous process in $\mathcal{P}_2(\mathbb{R}^d)$. Moreover, the map $(s, t, \omega^0, \mu, \alpha) \in[0, T] \times[0, T] \times \Omega^0 \times \mathcal{P}_2(\mathbb{R}^d) \times \mathcal{A} \rightarrow \rho_s^{t, \mu, \alpha}(\omega^0) \in \mathcal{P}_2(\mathbb{R}^d)$ (with the convention that $\rho_s^{t, \mu, \alpha}=\mu$ for $s \leq t$) is measurable.
\end{lemma}
\begin{proof}
    See \cite{Pham_Wei_Dynamic_Programming} Lemma 3.1.
\end{proof}
\section{A Theory of Viscosity Solution}
\subsection{Lions Derivative, The Space $\mathcal{K}^{1,2}(\mathcal{P}_2(\mathbb{R}^d))$
and It\^o-Wentzell Formula}
\label{Ito_Wentzell}
The following is adapted from \cite{dos2022ito} but not in its fullest generality.
\begin{definition}
\label{Lions' derivative}
(Lions Derivative). Fix a probability space $(\Omega,\mathcal{F},\mathbb{P})$. We consider a canonical lifting of the function $u: \mathcal{P}_2(\mathbb{R}^d) \rightarrow \mathbb{R}$ to $\tilde{u}$ : $L^2((\Omega, \mathcal{F}, \mathbb{P}); \mathbb{R}^d) \ni X \rightarrow \tilde{u}(X)=u(\mathcal{L}(X)) \in \mathbb{R}$. We say that $u$ is $L$-differentiable at $\mu$, if $\tilde{u}$ is Frechèt differentiable (in $L^2$) at some $X$, such that $\mu=\mathcal{L}(X)$. Denoting the gradient by $D \tilde{u}$ and using a Hilbert structure of the $L^2$ space, we can identify $D \tilde{u}$ as an element of its dual, $L^2$ itself. It is well known that $D\tilde{u}$ is a $\sigma(X)$-measurable random variable and given by the function $D u(\mu, \cdot): \mathbb{R}^d \rightarrow \mathbb{R}^d$, depending on the law of $X$ and satisfying $D u(\mu, \cdot) \in L^2(\mathbb{R}^d, \mathcal{B}(\mathbb{R}^d), \mu ; \mathbb{R}^d)$. Hereinafter the $L$-derivative of $u$ at $\mu$ is the map $\partial_\mu u(\mu, \cdot): \mathbb{R}^d \ni x \rightarrow \partial_\mu u(\mu, x) \in \mathbb{R}^d$, satisfying $D \tilde{u}(X)=\partial_\mu u(\mu, X)$. We always denote $\partial_\mu u$ as the version of the $L$-derivative that is continuous in the product topology of all components of $u$. Moreover, let $\partial_\mu^2$ denote second derivative in measure and $\partial_v \partial_\mu u$ denote the derivative with respect to new variable arisen after applying derivative in measure.
\end{definition}
\begin{definition}
    (The Space $\mathcal{K}^{1,2}(\mathcal{P}_2(\mathbb{R}^d))$). In the context of the settings outlined in Section \ref{setup}, we say that $u \in \mathcal{K}^{1,2}(\mathcal{P}_2(\mathbb{R}^d))$ if 
    \begin{enumerate}
        \item $u\in \mathcal{S}^2(C(\mathcal{P}_2(\mathbb{R}^d)))$.
        \item There exists $(\Game_t u,\Game_w u)\in\mathcal{L}^2(C(\mathcal{P}_2(\mathbb{R}^d)))$ such that for $\mathbb{P}^0$-a.e. $\omega^0$,
    \begin{align*}
        u(t,\mu) = u(T,\mu) - \int_t^T\Game_s u(s,\mu)ds - \int_t^T\Game_w u(s,\mu)dW_s^0\quad\forall (t,\mu)\in[0,T]\times\mathcal{P}_2(\mathbb{R}^d).
    \end{align*}
    \item For any $t\in[0,T]$ the map $\mu\mapsto u(t,\mu)$ is $L$-differentiable $\mathbb{P}^0$-a.e. at every $\mu\in\mathcal{P}_2(\mathbb{R}^d)$, and $\partial_\mu u\in\mathcal{S}^\infty(C(\mathcal{P}_2(\mathbb{R}^d)\times\mathbb{R}^d))$.
    \item For any $(t,\mu)\in[0,T]\times\mathcal{P}_2(\mathbb{R}^d)$ the map $x\mapsto \partial_\mu u(\mu,x)$ is $\mathbb{R}^d$ differentiable $\mathbb{P}^0$-a.e. at every $x\in \text{Supp($\mu$)}$, and $\partial_x\partial_\mu u \in \mathcal{S}^\infty(C(\mathcal{P}_2(\mathbb{R}^d)\times\mathbb{R}^d))$.
    \item For any $(t,x)\in [0,T]\times\text{Supp($\mu$)}$, the map $\mu\mapsto \partial_\mu u(t,\mu,x)$ is $L$-differentiable $\mathbb{P}^0$-a.e. at every $\mu\in\mathcal{P}_2(\mathbb{R}^d)$, and $\partial_\mu^2 u \in \mathcal{S}^\infty(C(\mathcal{P}_2(\mathbb{R}^d)\times\mathbb{R}^d\times\mathbb{R}^d))$.
    \item For any $t\in[0,T]$, the map $\mu\mapsto \Game_w u(\mu)$ is $L$-differentiable $\mathbb{P}^0$-a.e. at every point $\mu\in\mathbb{P}_2(\mathbb{R}^d)$, and $\partial_\mu\Game_w u\in\mathcal{S}^\infty(C(\mathcal{P}_2(\mathbb{R}^d)\times\mathbb{R}^d))$.
    \end{enumerate}
\end{definition}
\begin{theorem}
\label{Ito_Wentzell_Formula}
    (It\^o-Wentzell Formula). Let $u\in \mathcal{K}^{1,2}(\mathcal{P}_2(\mathbb{R}^d))$, and $(X_t)_{0\leq t\leq T}$ solves (\ref{basic_dynamics}). For almost all $\omega^0\in\Omega^0$ take $(\rho_t)_{t\in[0,T]}:=\mathcal{L}(X_t(\omega^0,\cdot))_{t\in[0,T]}$. Then $(u_t(\rho_t))_{t\in [0,T]}$ is an It\^o process $\mathbb{P}^0$-a.e. satisfying the expansion
        \begin{align*}
            u_t(\rho_T) - u_0(\rho_0) =& \int_0^T\Game_s u(s,\rho_s)ds + \int_0^T \Game_w u(s,\rho_s) dW_s^0\\
            &+\int_0^T \tilde{\mathbb{E}}^1\Big[\partial_\mu u(s,\rho_s)(\tilde{X}_s)\cdot \tilde{b}_s\Big]ds + \int_0^T \tilde{\mathbb{E}}^1\Big[\tilde{\sigma_s}^{0;\intercal}\partial_\mu u(s,\rho_s)(\tilde{X}_s)\Big]\cdot dW_s^0\\
            &+\int_0^T\frac{1}{2}\tilde{\mathbb{E}}^1\Big[\tr\{\partial_x\partial_\mu u(s,\rho_s)(\tilde{X}_s)(\tilde{\sigma}_s^0\tilde{\sigma}_s^{0;\intercal}+\tilde{\sigma}_s\tilde{\sigma}_s^{\intercal})\}\Big]ds\\
            &+\int_0^T\frac{1}{2}\hat{\mathbb{E}}^1\Big[\tilde{\mathbb{E}}^1\Big[\tr\{\partial_\mu^2 u(s,\rho_s)(\tilde{X}_s,\hat{X}_s)\tilde{\sigma}_s^0\hat{\sigma}_s^{0;T}\}\Big]\Big]ds\\
            &+\int_0^T\tilde{\mathbb{E}}^1\Big[\tr\{\partial_\mu \Game_w u(s,\rho_s)(\tilde{X}_s)\sigma_s^{0;\intercal}\}\Big]ds,
        \end{align*}
        where the formula above $\tilde{\mathbb{E}}$ and $\hat{\mathbb{E}}$ denote the expectation acting on the model twin spaces $(\tilde{\Omega}, \tilde{\mathbb{F}}, \tilde{\mathbb{P}})$ and $(\hat{\Omega}, \hat{\mathbb{F}}, \hat{\mathbb{P}})$ respectively, and let the processes $(\tilde{X}_t, \tilde{b}_t, \tilde{\sigma}_t,\tilde{\sigma}_t^0)_{t \in[0, T]}$ and $(\hat{X}_t, \hat{b}_t, \hat{\sigma}_t,\hat{\sigma}^0_t)_{t \in[0, T]}$ be the independent twin processes of $(X_t, b_t, \sigma_t,\sigma_t^0)_{t \in[0, T]}$ respectively living within.
\end{theorem}
\begin{proof}
    We refer the readers to \cite{dos2022ito}.
\end{proof}
\subsection{Test Functions, Definition of Viscosity solution}
We are concerned with the following BSPDE: for $(t,\mu)\in [0,T]\times\mathcal{P}_2(\mathbb{R}^d)$, $v:\Omega^0\times [0,T]\times\mathcal{P}_2(\mathbb{R}^d)\to\mathbb{R}$ satisfies 
\begin{align}
\label{BSPDE}
\begin{cases}
    \displaystyle-\Game_t v(t,\mu) = \Bigg(\int_{\mathbb{R}^d}\essinf_{\alpha,\alpha'\in A}\Bigg\{f_t(x,\mu,\alpha)+\Big\langle b_t(x,\mu,\alpha),\partial_\mu v(t,\mu)(x)\Big\rangle+\frac{1}{2}\tr\Big\{ (\sigma_t\sigma_t^{\intercal}\\+\sigma_t^0\sigma_t^{0;\intercal})(x,\mu,\alpha)\partial_x\partial_\mu v(t,\mu)(x)\Big\}
    +\displaystyle\int_{\mathbb{R}^d}\frac{1}{2}\tr\Big\{\sigma_t^0(x,\mu,\alpha)\sigma_t^{0;\intercal}(x',\mu,\alpha')\partial^2_{\mu}v(t,\mu)(x,x')\Big\}\mu(dx')\\+
    \tr\Big\{\sigma_t^{0;\intercal}(x,\mu,\alpha)\partial_\mu\Game_w v(t,\mu)(x)\Big\}\Bigg\}\mu(dx)\Bigg)dt \displaystyle,\\
    \displaystyle v(T,\mu) = \int_{\mathbb{R}^d}g(x,\mu)\mu(dx).
\end{cases}
\end{align}
For ease of notation, we define the following function:
\begin{align*}
    \mathbb{H}(t,\mu,P,Q,R,S):=& \int_{\mathbb{R}^d}\essinf_{\alpha,\alpha'\in A}\Bigg\{f_t(x,\mu,\alpha)+\Big\langle b_t(x,\mu,\alpha),P\Big\rangle+\frac{1}{2}\tr\Big\{ (\sigma_t\sigma_t^{\intercal}+\sigma_t^0\sigma_t^{0;\intercal})(x,\mu,\alpha)Q\Big\}\\
    &+\displaystyle\int_{\mathbb{R}^d}\frac{1}{2}\tr\Big\{\sigma_t^0(x,\mu,\alpha)\sigma_t^{0;\intercal}(x',\mu,\alpha')R\Big\}\mu(dx')+\tr\Big\{\sigma_t^{0;\intercal}(x,\mu,\alpha)S\Big\}\Bigg\}\mu(dx),
\end{align*}
and $\mathbb{H}:\Omega^0\times[0,T]\times\mathcal{P}_2(\mathbb{R}^d)\times\mathcal{L}^2(L^2(\mathbb{R}^d,\mathcal{B}(\mathbb{R}^d),\mu;\mathbb{R}^d))\times\mathcal{L}^2(L^2(\mathbb{R}^d,\mathcal{B}(\mathbb{R}^d),\mu;\mathbb{R}^{d\times d})\times\mathcal{L}^2(L^2(\mathbb{R}^d\times\mathbb{R}^d,\mathcal{B}(\mathbb{R}^d\times\mathbb{R}^d),\mu\otimes\mu;\mathbb{R}^{d
\times d}))\times\mathcal{L}^2(L^2(\mathbb{R}^d,\mathcal{B}(\mathbb{R}^d),\mu;\mathbb{R}^{d\times d})) \to\mathbb{R}$.
\\\hfill\\
We say that $u\in\mathscr{S}$ if
\begin{enumerate}
    \item $u\in\mathcal{K}^{1,2}(\mathcal{P}_2(\mathbb{R}^d))$.
    \item There exists a constant $\beta\in(0,1)$ such that there exists a constant $L_{u,\beta}>0$ such that for $\mathbb{P}^0$ a.e. $\omega^0\in\Omega^0$, 
    \begin{enumerate}
        \item for all $t\in [0,T]$, $\mu,\mu'\in\mathcal{P}_2(\mathbb{R}^d)$, 
        \begin{align*}
            |\Game_t u(t,\mu) - \Game_t u(t,\mu')|\leq L_{u,\beta}\mathcal{W}_2^\beta(\mu,\mu'),
        \end{align*}
        \item for all $t\in [0,T]$, $x,x'\in\mathbb{R}^d$, $\mu,\mu'\in\mathcal{P}_2(\mathbb{R}^d)$, $h = \partial_\mu u, \partial_x\partial_\mu u$, $\partial_\mu\Game_w u$,
        \begin{align*}
            |h(t,\mu)(x) - h(t,\mu')(x')|\leq L_{u,\beta}\Big(\mathcal{W}_2^\beta(\mu,\mu')+|x-x'|^\beta\Big),
        \end{align*}
        \item for all $t\in [0,T]$, $x,x',y,y'\in\mathbb{R}^d$, $\mu,\mu'\in\mathcal{P}_2(\mathbb{R}^d)$, 
        \begin{align*}
            |\partial_\mu^2 u(t,\mu)(x,y) - \partial_\mu^2 u(t,\mu')(x',y')|\leq L_{u,\beta}\Big(\mathcal{W}_2^\beta(\mu,\mu')+|x-x'|^\beta+|y-y'|^\beta\Big).
        \end{align*}
    \end{enumerate}
\end{enumerate}
\begin{remark}
As a consequence that $u \in \mathcal{K}^{1,2}(\mathcal{P}_2(\mathbb{R}^d))$, there exists a constant $L_{u}>0$ such that for $\mathbb{P}^0$ a.e. $\omega^0$, for all $t\in[0,T]$, $x,x'\in\mathbb{R}^d$, $\mu\in\mathcal{P}_2(\mathbb{R}^d)$, 
    \begin{align*}
        |\partial_\mu u(t,\mu)(x)|+|\partial_x\partial_\mu u(t,\mu)(x)|+|\partial^2_\mu u(t,\mu)(x,x')|+|\partial_\mu\Game_w u(t,\mu)(x)|< L_u.
    \end{align*}
\end{remark}
Each $u \in \mathscr{S}$ can be thought of as an Itô process and thus a semi-martingale parameterized by $\mu \in \mathcal{P}_2(\mathbb{R}^d)$. The Doob-Meyer decomposition theorem ensures the uniqueness of the integrable pair $(\Game_t u, \Game_\omega u)$. 
\\\hfill\\
Recall for each $\mathcal{F}^0$-stopping time $\tau\leq T$, we denote by $\mathcal{T}_{t,T}^0$ the set of $\mathcal{F}^0$-stopping times $\tau$ valued in $[t,T]$ and by $\mathcal{T}_{t+}^0$ the subset of $\mathcal{T}^0_{t,T}$ such that $\tau>t$ for any $\tau\in\mathcal{T}_{t,T}^0$. When we are referring to $\mathcal{P}_2(\mathbb{R}^d)$-valued random variables, we will use bold fonts to avoid potential confusions. Before we proceed, readers are referred to Appendix \ref{compact_subset} for the definition of our compact subset $\mathcal{P}_{L}^\tau(t_0,\rho_0)$.
\begin{definition}
    (Set of test functions). Let $(t_0,\rho_0)\in [0,T]\times\mathcal{P}_2(\mathbb{R}^d)$, $\tau \in \mathcal{T}_{t,T}^0$, $\hat{\tau}\in\mathcal{T}_{\tau+}^0$, $L>0$, $\boldsymbol{\rho}\in\mathcal{P}_L^\tau(t_0,\rho_0)$, $\Omega_\tau^0\in\mathcal{F}_\tau^0$ with $\mathbb{P}^0(\Omega_\tau^0)>0$, we define the following set of test functions: 
    \begin{align*}
        \underline{\mathcal{G}}u(\tau,\boldsymbol{\rho},\hat{\tau};t_0,\rho_0;\Omega_\tau^0):=&\Bigg\{\phi \in \mathscr{S}\,\Big|\,(\phi - u)(\tau,\boldsymbol{\rho})1_{\Omega_\tau^0}=0 \\
        &= \essinf_{\bar{\tau}\in\mathcal{T}_{\tau,T}^0}\mathbb{E}_{\mathcal{F}^0_\tau}\Big[\essinf_{\boldsymbol{\mu}\in\mathcal{P}_L^{\bar{\tau}\wedge\hat{\tau}}(t_0,\rho_0)}(\phi-u)(\bar{\tau}\wedge\hat{\tau},\boldsymbol{\mu})\Big]1_{\Omega_\tau^0} a.e. \Bigg\},
    \end{align*}
    and
    \begin{align*}
        \overline{\mathcal{G}}u(\tau,\boldsymbol{\rho},\hat{\tau};t_0,\rho_0;\Omega_\tau^0):=&\Bigg\{\phi \in \mathscr{S}\,\Big|\,(\phi - u)(\tau,\boldsymbol{\rho})1_{\Omega_\tau^0}=0 \\
        &= \esssup_{\bar{\tau}\in\mathcal{T}_{\tau,T}^0}\mathbb{E}_{\mathcal{F}^0_\tau}\Big[\esssup_{\boldsymbol{\mu}\in\mathcal{P}_L^{\bar{\tau}\wedge\hat{\tau}}(t_0,\rho_0)}(\phi-u)(\bar{\tau}\wedge\hat{\tau},\boldsymbol{\mu})\Big]1_{\Omega_\tau^0} a.e.\Bigg\}.
    \end{align*}
\end{definition}
\begin{definition}
    (Viscosity solution). We say $u\in \mathcal{S}^2(C(\mathcal{P}_2(\mathbb{R}^d)))$ is a viscosity subsolution (resp. supersolution) of the BSPDE (\ref{BSPDE}) if there exists $L$ > 0 such that
    \begin{enumerate}[(i)]
        \item For the terminal condition, we have for all $\mu\in\mathcal{P}_2(\mathbb{R}^d)$,
        \begin{align*}
            u(T,\mu)\leq(\text{resp. }\geq)\, \int_{\mathbb{R}^d}g(x,\mu)\mu(dx)\,\,\,\mathbb{P}^0 a.e..
        \end{align*}
        \item For any $(t_0,\rho_0)\in [0,T]\times \mathcal{P}_2(\mathbb{R}^d)$, $\tau \in \mathcal{T}_{t,T}^0$, $\hat{\tau}\in\mathcal{T}_{\tau+}^0$, $\boldsymbol{\rho}\in\mathcal{P}_L^\tau(t_0,\rho_0)$, $\Omega_\tau^0 \in \mathcal{F}_\tau^0$ with $\mathbb{P}(\Omega_\tau^0)>0$ and any $\phi \in \underline{\mathcal{G}}u(\tau,\boldsymbol{\rho},\hat{\tau};t_0,\rho_0;\Omega_\tau^0)$ (resp. $\phi \in \overline{\mathcal{G}}u(\tau,\boldsymbol{\rho},\hat{\tau};t_0,\rho_0;\Omega_\tau^0)$), there holds
    \begin{align*}
        &\essliminf_{\substack{(s,\boldsymbol{\mu})\to (\tau^+,\boldsymbol{\rho}),\\ \boldsymbol{\mu}\in \mathcal{P}_L^s(t_0,\rho_0)}}\mathbb{E}_{\mathcal{F}^0_\tau}\Bigg[1_{s<\hat{\tau}}\Big[-\Game_s \phi(s,\boldsymbol{\mu})\\
        &- \mathbb{H}\Big(s,\boldsymbol{\mu},\partial_\mu \phi(s,\boldsymbol{\mu})(\cdot),\partial_x\partial_\mu \phi(s,\boldsymbol{\mu})(\cdot),\partial_\mu^2 \phi(s,\boldsymbol{\mu})(\cdot,\cdot),\partial_\mu\Game_w \phi(s,\boldsymbol{\mu})(\cdot)\Big)\Big]\Bigg]\leq 0,
    \end{align*}
    for almost all $\omega^0\in\Omega_\tau^0$ (resp.
    \begin{align*}
       &\esslimsup_{\substack{(s,\boldsymbol{\mu})\to (\tau^+,\boldsymbol{\rho}),\\ \boldsymbol{\mu}\in \mathcal{P}_L^s(t_0,\rho_0)}}\mathbb{E}_{\mathcal{F}^0_\tau}\Bigg[1_{s<\hat{\tau}}\Big[-\Game_s \phi(s,\boldsymbol{\mu})\\
        &- \mathbb{H}\Big(s,\boldsymbol{\mu},\partial_\mu \phi(s,\boldsymbol{\mu})(\cdot),\partial_x\partial_\mu \phi(s,\boldsymbol{\mu})(\cdot),\partial_\mu^2 \phi(s,\boldsymbol{\mu})(\cdot,\cdot),\partial_\mu\Game_w \phi(s,\boldsymbol{\mu})(\cdot)\Big)\Big]\Bigg]\geq 0,
    \end{align*}
    for almost all $\omega^0\in\Omega_\tau^0$).
    \end{enumerate}
    We say $u$ is a viscosity solution of the BSPDE (\ref{BSPDE}) if it is both viscosity subsolution and supersolution. Our definition is a simple extension of the case of $\mathbb{R}^d$, in which it is typical that the classical solution $u$ may not be differentiable in the time variable $t$, and $(\Game_t u,\Game_w u)$ may not be time-continuous but just measurable in $t$, this fact motivating us to use essential limits in the above.
    \begin{remark}
        For a progressively measurable function $H:[0,T]\times\mathcal{P}_2(\mathbb{R}^d)\times\Omega^0\to\mathbb{R}$, when we refer to $\displaystyle\essliminf_{\substack{(s,\boldsymbol{\mu})\to(\tau^+,\boldsymbol{\rho})\\\boldsymbol{\mu}\in\mathcal{P}_L^s(t_0,\rho_0)}}$ \Big(resp. $\displaystyle\esslimsup_{\substack{(s,\boldsymbol{\mu})\to(\tau^+,\boldsymbol{\rho})\\\boldsymbol{\mu}\in\mathcal{P}_L^s(t_0,\rho_0)}}$ \Big), we are taking the following limit for stopping time $s$, $s\geq \tau$ and $\mathcal{F}_{s}$-measurable $\mu$:
        \begin{align*}
            \displaystyle\essliminf_{\substack{(s,\boldsymbol{\mu})\to(\tau^+,\boldsymbol{\rho})\\\boldsymbol{\mu}\in\mathcal{P}_L^s(t_0,\rho_0)}}H(s,\boldsymbol{\mu}) &:= \esslim_{\delta\to 0} \Bigg(\essinf_{\substack{s\in B_{\delta}(\tau), s\geq \tau\\\boldsymbol{\mu}\in B_{\delta}(\boldsymbol{\rho})\cap\mathcal{P}_L^s(t_0,\rho_0)}} H(s,\boldsymbol{\mu})\Bigg),\\
            (\text{resp. } \displaystyle\esslimsup_{\substack{(s,\boldsymbol{\mu})\to(\tau^+,\boldsymbol{\rho})\\\boldsymbol{\mu}\in\mathcal{P}_L^s(t_0,\rho_0)}}H(s,\boldsymbol{\mu}) &:= \esslim_{\delta\to 0} \Bigg(\esssup_{\substack{s\in B_{\delta}(\tau), s\geq \tau\\\boldsymbol{\mu}\in B_{\delta}(\boldsymbol{\rho})\cap\mathcal{P}_L^s(t_0,\rho_0)}} H(s,\boldsymbol{\mu})\Bigg)).
        \end{align*}
    \end{remark}

\end{definition}
\subsection{Existence of the viscosity solution}
We introduce the following notation which will be useful in the proof of existence of viscosity solution. Write
\begin{align*}
    Q_r^+(t,\boldsymbol{\rho};t_0,\rho_0) := \big\{(s,\boldsymbol{\mu})\in [t,t+r^2)\times B_r(\boldsymbol{\rho})\,|\,\boldsymbol{\mu}\in\mathcal{P}_L^{s}(t_0,\rho_0)\big\}.
\end{align*}
Also, for convenience of notation, for $\phi \in \mathscr{S}$, $\alpha\in \mathcal{A}$, we denote
    \begin{align*}
        &\mathscr{L}^\alpha\phi(t,\mu)\\
        :=&
        \Game_t\phi(t,\mu) + \Bigg\{\int_{\mathbb{R}^d}\Big\langle b_t(x,\mu,\alpha),\partial_\mu v(t,\mu)(x)\Big\rangle+\frac{1}{2}\tr\Big\{ (\sigma_t\sigma_t^{\intercal}+\sigma_t^0\sigma_t^{0;\intercal})(x,\mu,\alpha)\partial_x\partial_\mu v(t,\mu)(x)\Big\}
    \\&+\displaystyle\int_{\mathbb{R}^d}\frac{1}{2}\tr\Big\{\sigma_t^0(x,\mu,\alpha)\sigma_t^{0;\intercal}(x',\mu,\alpha')\partial^2_{\mu}v(t,\mu)(x,x')\Big\}\mu(dx')+
    \tr\Big\{\sigma_t^{0;\intercal}(x,\mu,\alpha)\partial_\mu\Game_w v(t,\mu)(x)\Big\}\Bigg\}\mu(dx).
    \end{align*}
    Also, for $\mu\in\mathcal{P}_2(\mathbb{R}^d)$, $\alpha\in\mathcal{A}$, we write
    \begin{align*}
        \hat{f}_s(\mu,\alpha):= \int_{\mathbb{R}^d} f_s(x,\mu,\alpha)\mu(dx).
    \end{align*}
We have the following existence theorem:
\begin{theorem}
\label{existence_solution}
    Let ($\mathcal{A}$1), ($\mathcal{A}$2) hold. The value function $v$ is a viscosity solution of the stochastic HJB equation (\ref{BSPDE}).
\end{theorem}
\begin{proof}
In the following we shall take $L>K$. \\
    \textbf{Step 1}. ($v$ is a subsolution). Suppose to the contrary, that there exists $(t_0,\rho_0) \in [0,T]\times \mathcal{P}_2(\mathbb{R}^d)$, $\tau \in \mathcal{T}_{t,T}^0$, $\hat{\tau} \in \mathcal{T}_{\tau+}^0$, $\boldsymbol{\rho}\in\mathcal{P}_L^\tau(t_0,\rho_0)$, $\Omega_\tau^0\in\mathcal{F}_\tau^0$, $\phi\in\underline{\mathcal{G}}v(\tau,\boldsymbol{\rho},\hat{\tau};t_0,\rho_0;\Omega_\tau^0)$ such that there exists $\varepsilon,\tilde{\delta}>0$, and $\Omega^{0'}\in\mathcal{F}_\tau^0$, $\mathbb{P}(\Omega^{0'})>0$, with
    \begin{align*}
        &\essinf_{(s,\boldsymbol{\mu})\in Q_{\tilde{\delta}}^+(\tau,\boldsymbol{\rho};t_0,\rho_0)}\mathbb{E}_{\mathcal{F}^0_\tau}\Bigg[1_{s<\hat{\tau}}\Big[-\Game_s \phi(s,\boldsymbol{\mu})\\
        &- \mathbb{H}\Big(s,\boldsymbol{\mu},\partial_\mu \phi(s,\boldsymbol{\mu})(\cdot),\partial_x\partial_\mu \phi(s,\boldsymbol{\mu})(\cdot),\partial_\mu^2 \phi(s,\boldsymbol{\mu})(\cdot,\cdot),\partial_\mu\Game_w \phi(s,\boldsymbol{\mu})(\cdot)\Big)\Big]\Bigg]\geq 2\varepsilon,
    \end{align*}
    $\mathbb{P}^0$-a.e. in $\Omega^{0'}$. 
    By the measurable selection theorem (interested readers may refer to Appendix \ref{measurable_selection_construct}), there exists $\overline{\alpha}\in \mathcal{A}_\tau$ such that for almost all $\omega^0\in\Omega_\tau^0$,
    \begin{align}
    \label{overline_alpha_eq}
        \nonumber&-\mathscr{L}^{\overline{\alpha}}\phi(s,\boldsymbol{\rho})-\int_{\mathbb{R}^d}f_s(x,\boldsymbol{\rho},\overline{\alpha}_s)\boldsymbol{\rho}(dx)\\
        \geq &-\Game_s \phi(s,\boldsymbol{\rho})-\mathbb{H}\Big(s,\boldsymbol{\rho},\partial_\mu \phi(s,\boldsymbol{\rho})(\cdot),\partial_x\partial_\mu \phi(s,\boldsymbol{\rho})(\cdot),\partial_\mu^2 \phi(s,\boldsymbol{\rho})(\cdot,\cdot),\partial_\mu\Game_w \phi(s,\boldsymbol{\rho})(\cdot)\Big)-\varepsilon
    \end{align}
    for almost all $\tau \leq s < T$. Therefore,
    \begin{align*}
        \essinf_{\tau \leq s <(\tau+\tilde{\delta}^2)\wedge T}\mathbb{E}_{\mathcal{F}^0_\tau}\Bigg[1_{s\leq \hat{\tau}}\Big[-\mathscr{L}^{\overline{\alpha}}\phi(s,\boldsymbol{\rho})-\int_{\mathbb{R}^d}f_s(x,\boldsymbol{\rho},\overline{\alpha}_s)\boldsymbol{\rho}(dx)\Big]\Bigg]\geq \varepsilon,
    \end{align*}
    $\mathbb{P}^0$-a.e. in $\Omega^{0'}$. As $\boldsymbol{\rho}\in\mathcal{P}_L^\tau(t_0,\rho_0)$, there exists $\xi$ on $(\Omega,\mathcal{F}_\tau,\mathbb{P})$ such that $\mathcal{L}(\xi(\omega^0,\cdot)) = \boldsymbol{\rho}(\omega^0)$. Let $(X_s^{\tau,\xi;\overline{\alpha}})_{\tau\leq s\leq T}$ solves (\ref{basic_dynamics}) with inital time $\tau$, initial data $\xi$ and control $\overline{\alpha}$. By Lemma \ref{pham_wei_rho_process}, we can now examine the $\mathcal{P}_2(\mathbb{R}^d)$-valued process $(\rho_s^{\tau,\boldsymbol{\rho};\overline{\alpha}})_{\tau\leq s\leq T}:= \mathcal{L}(X_s^{\tau,\xi;\overline{\alpha}}(\omega^0,\cdot))_{\tau\leq s\leq T}$. \\
    \hfill\\
    By the dynamic programming principle (Theorem \ref{DPP_THM}) and It\^o-Wentzell formula (Theorem \ref{Ito_Wentzell_Formula}), for any $h\in(0,\tilde{\delta}^2/4)$, $h$ small enough and $\mathbb{P}^0$-a.e. $\omega^0 \in \Omega^{0'}$, we have
    \begin{align*}
        0\geq&\frac{1}{h} \mathbb{E}_{\mathcal{F}^0_\tau}\Big[(\phi-v)(\tau,\boldsymbol{\rho})-(\phi-v)\Big((\tau+h)\wedge\hat{\tau},\rho_{(\tau+h)\wedge \hat{\tau}}^{\tau,\boldsymbol{\rho};\overline{\alpha}}\Big)\Big]\\
        \geq&\frac{1}{h}\mathbb{E}_{\mathcal{F}^0_\tau}\Big[\phi(\tau,\boldsymbol{\rho})-\phi((\tau+h)\wedge \hat{\tau},\rho_{(\tau+h)\wedge \hat{\tau}}^{\tau,\boldsymbol{\rho};\overline{\alpha}})-\int_\tau^{(\tau+h)\wedge \hat{\tau}}\hat{f}_s(\rho_s^{\tau,\boldsymbol{\rho};\overline{\alpha}},\overline{\alpha}_s)ds\Big]\\
        =&\frac{1}{h}\mathbb{E}_{\mathcal{F}^0_\tau}\Big[\int_\tau^{(\tau+h)\wedge \hat{\tau}}-\mathscr{L}^{\overline{\alpha}}\phi(s,\rho_{s}^{\tau,\boldsymbol{\rho};\overline{\alpha}})-\hat{f}_{s}(\rho_{s}^{\tau,\boldsymbol{\rho};\overline{\alpha}},\overline{\alpha}_{s})ds\Big]\\
        \geq&\frac{1}{h}\mathbb{E}_{\mathcal{F}^0_\tau}\Big[\int_\tau^{(\tau+h)\wedge \hat{\tau}}-\mathscr{L}^{\overline{\alpha}}\phi(s,\boldsymbol{\rho})-\hat{f}_{s}(\boldsymbol{\rho},\overline{\alpha}_{s})ds\Big]\\
        &-\frac{1}{h}\mathbb{E}_{\mathcal{F}^0_\tau}\Big[\int_\tau^{(\tau+h)\wedge \hat{\tau}}\Big|-\mathscr{L}^{\overline{\alpha}}\phi(s,\rho_{s}^{\tau,\boldsymbol{\rho};\overline{\alpha}})-\hat{f}_{s}(\rho_{s}^{\tau,\boldsymbol{\rho};\overline{\alpha}},\overline{\alpha}_{s})ds+\mathscr{L}^{\overline{\alpha}}\phi(s,\boldsymbol{\rho})+\hat{f}_{s}(\boldsymbol{\rho},\overline{\alpha}_{s})\Big|ds\Big]\\
        \geq &\varepsilon-\frac{1}{h}\mathbb{E}_{\mathcal{F}^0_\tau}\Big[\int_\tau^{(\tau+h)\wedge \hat{\tau}}\Big|-\mathscr{L}^{\overline{\alpha}}\phi(s,\rho_{s}^{\tau,\boldsymbol{\rho};\overline{\alpha}})-\hat{f}_{s}(\rho_{s}^{\tau,\boldsymbol{\rho};\overline{\alpha}},\overline{\alpha}_{s})ds+\mathscr{L}^{\overline{\alpha}}\phi(s,\boldsymbol{\rho})+\hat{f}_{s}(\boldsymbol{\rho},\overline{\alpha}_{s})\Big|ds\Big]\\
        \geq &\varepsilon- C(K,L_\phi,L_{\phi,\beta})\mathbb{E}_{\mathcal{F}^0_\tau}\max_{\tau\leq s\leq \tau+h}\mathcal{W}_2^\beta(\rho_{s}^{\tau,\boldsymbol{\rho},\bar{\alpha}},\boldsymbol{\rho})\\
        \to &\varepsilon, \text{ as }h \to 0,
        \end{align*}
    where $C(K,L_\phi,L_{\phi,\beta})>0$ is a constant depending $K$, and $L_\phi$, $L_{\phi,\beta}$. The last two are constants from the fact that $\phi\in\mathscr{S}$. The above draws a contradiction, and hence $v$ is a viscosity subsolution.\\
    \hfill\\
    \textbf{Step 2}. Now we prove that $v$ is a viscosity supersolution. Suppose to the contrary, that there exists $(t_0,\rho_0) \in [0,T]\times \mathcal{P}_2(\mathbb{R}^d)$, $\tau \in \mathcal{T}^0_{t,T}$, $\hat{\tau} \in \mathcal{T}_{\tau+}^0$, $\boldsymbol{\rho}\in\mathcal{P}_L^\tau(t_0,\rho_0)$, $\Omega_\tau^0\in\mathcal{F}_\tau^0$, $\phi\in\overline{\mathcal{G}}v(\tau,\boldsymbol{\rho},\hat{\tau};t_0,\rho_0;\Omega_\tau^0)$ such that there exists $\varepsilon,\tilde{\delta}>0$, and $\Omega^{0'}\in\mathcal{F}_\tau^0$, $\mathbb{P}(\Omega^{0'})>0$, with
    \begin{align*}
        &\esssup_{(s,\boldsymbol{\mu})\in Q_{\tilde{\delta}}^+(\tau,\boldsymbol{\rho};t_0,\rho_0)}\mathbb{E}_{\mathcal{F}^0_\tau}\Bigg[1_{s<\hat{\tau}}\Big[-\Game_s \phi(s,\boldsymbol{\mu})\\
        &- \mathbb{H}\Big(s,\boldsymbol{\mu},\partial_\mu \phi(s,\boldsymbol{\mu})(\cdot),\partial_x\partial_\mu \phi(s,\boldsymbol{\mu})(\cdot),\partial_\mu^2 \phi(s,\boldsymbol{\mu})(\cdot,\cdot),\partial_\mu\Game_w \phi(s,\boldsymbol{\mu})(\cdot)\Big)\Big]\Bigg]\leq-\varepsilon,
    \end{align*}
    $\mathbb{P}^0$ a.e. in $\Omega^{0'}$. As in Step 1, there exists $\xi$ on $(\Omega,\mathcal{F}_\tau,\mathbb{P})$ such that $\mathcal{L}(\xi(\omega^0,\cdot)) = \boldsymbol{\rho}(\omega^0)$. 
    For each $\alpha\in\mathcal{A}_\tau$, define $\tau^\alpha :=\inf\{s>\tau:\rho_s^{\tau,\boldsymbol{\rho};\alpha}\notin B_{\tilde{\delta}/4}(\boldsymbol{\rho})\}$,
    by Lemma \ref{SDE_property} we have the estimate:
    \begin{align}
    \label{control_growth_estimate}    \mathbb{E}_{\mathcal{F}^0_\tau}[1_{\tau+h>\tau^\alpha}] = &\mathbb{E}\Big[1_{\max_{\tau\leq s\leq\tau+h}\mathcal{W}_2(\rho_s^{\tau,\boldsymbol{\rho},\alpha},\boldsymbol{\rho})>\tilde{\delta}/4}\Big]\nonumber\\
        \leq &C_{\tilde{\delta},T,K}\mathbb{E}_{\mathcal{F}^0_\tau}\max_{\tau\leq s\leq \tau+h}\mathcal{W}_2^2(\rho_s^{\tau,\boldsymbol{\rho},\alpha},\boldsymbol{\rho})\nonumber\\
        \leq &C_{\tilde{\delta},T,K}(1+\mathbb{E}_{\mathcal{F}^0_\tau}|\boldsymbol{\rho}|^2)h,
    \end{align}
    where $C_{\tilde{\delta},T,K}$ is a constant depending only on $\tilde{\delta}$, $T$ and $K$ and independent of the control $\alpha$.
    For $h\in(0,\tilde{\delta}^2/4)$ and $h$ small enough, by the dynamic programming principle and Ito-Wentzell formula, we have
    \begin{align*}
        0= &\frac{1}{h}\Big(v(\tau,\boldsymbol{\rho})-\phi(\tau,\boldsymbol{\rho})\Big)\\
         =&\frac{1}{h}\essinf_{\alpha\in\mathcal{A}_\tau}\mathbb{E}_{\mathcal{F}^0_\tau}\Bigg[\int_\tau^{\hat{\tau}\wedge(\tau+h)}\hat{f}_s(\rho_s^{\tau,\boldsymbol{\rho};\alpha},\alpha_s)ds+v(\hat{\tau}\wedge(\tau+h),\rho_{\hat{\tau}\wedge(\tau+h)}^{\tau,\boldsymbol{\rho};\alpha})-\phi(\tau,\boldsymbol{\rho})\Bigg]\\
        \geq &\frac{1}{h}\essinf_{\alpha\in\mathcal{A}_\tau}\mathbb{E}_{\mathcal{F}^0_\tau}\Bigg[\int_\tau^{\hat{\tau}\wedge(\tau+h)}\hat{f}_s(\rho_s^{\tau,\boldsymbol{\rho};\alpha},\alpha_s)ds+\phi(\hat{\tau}\wedge(\tau+h),\rho_{\hat{\tau}\wedge(\tau+h)}^{\tau,\boldsymbol{\rho};\alpha})-\phi(\tau,\boldsymbol{\rho})\Bigg]\\
        = &\frac{1}{h}\essinf_{\alpha\in\mathcal{A}_\tau}\mathbb{E}_{\mathcal{F}^0_\tau}\Bigg[\int_\tau^{\tau+h}\Big(\hat{f}_s(\rho_s^{\tau,\boldsymbol{\rho};\alpha},\alpha_s)+\mathscr{L}^\alpha\phi(s,\rho_s^{\tau,\boldsymbol{\rho},\alpha})\Big)1_{s\leq\hat{\tau}}ds\Bigg]\\
        \geq &\frac{1}{h}\essinf_{\alpha\in\mathcal{A}_\tau}\mathbb{E}_{\mathcal{F}^0_\tau}\Bigg[\int_\tau^{\tau+h}\Big(\hat{f}_{s}(\rho_{s\wedge\tau^\alpha}^{\tau,\boldsymbol{\rho};\alpha},\alpha_{s\wedge\tau^\alpha})+\mathscr{L}^\alpha\phi({s},\rho_{s\wedge\tau^\alpha}^{\tau,\boldsymbol{\rho},\alpha})\Big)1_{s\leq\hat{\tau}}1_{\tau+h\leq\tau^\alpha}ds\Bigg]\\
        &-\frac{1}{h}\essinf_{\alpha\in\mathcal{A}_\tau}\mathbb{E}_{\mathcal{F}^0_\tau}\Bigg[\int_\tau^{\hat{\tau}\wedge(\tau+h)}\Big|\hat{f}_s(\rho_s^{\tau,\boldsymbol{\rho};\alpha},\alpha_s)+\mathscr{L}^\alpha\phi(s,\rho_s^{\tau,\boldsymbol{\rho},\alpha})\Big|1_{\tau+h>\tau^\alpha}ds\Bigg]\\
        \geq &\frac{1}{h}\essinf_{\alpha\in\mathcal{A}_\tau}\mathbb{E}_{\mathcal{F}^0_\tau}\Bigg[\int_\tau^{\tau+h}\Big(\hat{f}_{s}(\rho_{s\wedge\tau^\alpha}^{\tau,\boldsymbol{\rho};\alpha},\alpha_{s\wedge\tau^\alpha})+\mathscr{L}^\alpha\phi({s},\rho_{s\wedge\tau^\alpha}^{\tau,\boldsymbol{\rho},\alpha})\Big)1_{s\leq\hat{\tau}}ds\Bigg]\\
        &-\frac{1}{h}\esssup_{\alpha\in\mathcal{A}_\tau}\mathbb{E}_{\mathcal{F}^0_\tau}\Bigg[\int_\tau^{\hat{\tau}\wedge(\tau+h)}\Big|\hat{f}_s(\rho_s^{\tau,\boldsymbol{\rho};\alpha},\alpha_s)+\mathscr{L}^\alpha\phi(s,\rho_s^{\tau,\boldsymbol{\rho},\alpha})\Big|1_{\tau+h>\tau^\alpha}ds\Bigg]\\
        &-\frac{1}{h}\esssup_{\alpha\in\mathcal{A}_\tau}\mathbb{E}_{\mathcal{F}^0_\tau}\Bigg[\int_\tau^{\tau+h}\Big|\hat{f}_{s}(\rho_{s\wedge\tau^\alpha}^{\tau,\boldsymbol{\rho};\alpha},\alpha_{s\wedge\tau^\alpha})+\mathscr{L}^\alpha\phi({s},\rho_{s\wedge\tau^\alpha}^{\tau,\boldsymbol{\rho},\alpha})\Big|1_{s\leq\hat{\tau}}1_{\tau+h>\tau^\alpha}ds\Bigg]\\
        \geq &\varepsilon - \frac{1}{h}\esssup_{\alpha\in\mathcal{A}_\tau}\sqrt{\mathbb{E}_{\mathcal{F}^0_\tau}\int_\tau^{\tau+h}\Big|\hat{f}_s(\rho_s^{\tau,\boldsymbol{\rho};\alpha},\alpha_s)+\mathscr{L}^\alpha\phi(s,\rho_s^{\tau,\boldsymbol{\rho},\alpha})\Big|^2ds}\sqrt{\mathbb{E}_{\mathcal{F}^0_\tau}\int_\tau^{\tau+h}\Big|1_{\tau+h>\tau^\alpha}\Big|^2ds}\\
        &-\frac{1}{h}\esssup_{\alpha\in\mathcal{A}_\tau}\sqrt{\mathbb{E}_{\mathcal{F}^0_\tau}\int_\tau^{\tau+h}\Big|\hat{f}_s(\rho_{s\wedge\tau^\alpha}^{\tau,\boldsymbol{\rho};\alpha},\alpha_{s\wedge\tau^\alpha})+\mathscr{L}^\alpha\phi(s,\rho_{s\wedge\tau^\alpha}^{\tau,\boldsymbol{\rho},\alpha})\Big|^2ds}\sqrt{\mathbb{E}_{\mathcal{F}^0_\tau}\int_\tau^{\tau+h}\Big|1_{\tau+h>\tau^\alpha}\Big|^2ds}\\
        \to&\varepsilon,\text{ as }h\to 0, \text{ with the help of (\ref{control_growth_estimate})}.
        \end{align*}
       The above draws a contradiction and hence the value function $v$ is a viscosity supersolution.
\end{proof}
\subsection{Uniqueness of the viscosity solution}
\subsubsection{A comparison result}
\begin{theorem}
\label{uniqueness_result}
    Let ($\mathcal{A}$1) hold. Let $u$ be a viscosity supersolution (resp. subsolution) of the BSPDE (\ref{BSPDE}). Let $\phi\in\mathcal{S}^2 (C(\mathcal{P}_2(\mathbb{R}^d)))$ be such that
    \begin{enumerate}
        \item [(i)] There exists $\phi_{n,m}\in\mathscr{S}$, $(n,m)\in\mathbb{N}\times\mathbb{N}$, such that for $\mathbb{P}^0$-a.e. $\omega^0$, for all $(t,\mu)\in [0,T]\times\mathcal{P}_2(\mathbb{R}^d)$ such that there exists $q>2$, $\mu\in\mathcal{P}_q(\mathbb{R}^d)$, we have
        \begin{align*}
            \lim_{n}\lim_m\phi_{n,m}(t,\mu) = \phi(t,\mu), a.s..
        \end{align*}
        \item [(ii)] There exists a continuous $F:\mathcal{P}_2(\mathbb{R}^d)\to\mathbb{R}$, for $(n,m)\in\mathbb{N}\times\mathbb{N}$, for $\mathbb{P}^0$-a.e. $\omega^0$, for all $(t_0,\rho_0)\in [0,T]\times\mathcal{P}_2(\mathbb{R}^d)$ such that there exists $q>2$, $\rho_0\in\mathcal{P}_q(\mathbb{R}^d)$, for all $\tau\in\mathcal{T}_{t_0,T}^0$, $\boldsymbol{\rho}\in\mathcal{P}_L^\tau(t_0,\rho_0)$, $\hat{\tau}\in\mathcal{T}_{\tau+}^0$, we have a.s.
    \begin{align*}
        &\esslimsup_{\substack{(s,\boldsymbol{\mu})\to (\tau^+,\boldsymbol{\rho})\\
        \boldsymbol{\mu}\in \mathcal{P}_L^s(t_0,\rho_0)}}\mathbb{E}_{\mathcal{F}^0_\tau}\Bigg[1_{s\leq\hat{\tau}}\Big[-\Game_s \phi_{n,m}\\
        &- \mathbb{H}\Big(s,\boldsymbol{\mu},\partial_\mu \phi_{n,m}(s,\boldsymbol{\mu})(\cdot),\partial_x\partial_\mu \phi_{n,m}(s,\boldsymbol{\mu})(\cdot),\partial_\mu^2 \phi_{n,m}(s,\boldsymbol{\mu})(\cdot,\cdot),\partial_\mu\Game_w \phi_{n,m}(s,\boldsymbol{\mu})(\cdot)\Big)-F(\boldsymbol{\mu})h_n - l_m\Big]\Bigg]\\
        &\leq 0,
    \end{align*}
    (resp. 
    \begin{align*}
        &\esslimsup_{\substack{(s,\boldsymbol{\mu})\to (\tau^+,\boldsymbol{\rho})\\
        \boldsymbol{\mu}\in \mathcal{P}_L^s(t_0,\rho_0)}}\mathbb{E}_{\mathcal{F}^0_\tau}\Bigg[1_{s\leq\hat{\tau}}\Big[-\Game_s \phi_{n,m}\\
        &- \mathbb{H}\Big(s,\boldsymbol{\mu},\partial_\mu \phi_{n,m}(s,\boldsymbol{\mu})(\cdot),\partial_x\partial_\mu \phi_{n,m}(s,\boldsymbol{\mu})(\cdot),\partial_\mu^2 \phi_{n,m}(s,\boldsymbol{\mu})(\cdot,\cdot),\partial_\mu\Game_w \phi_{n,m}(s,\boldsymbol{\mu})(\cdot)\Big)+F(\boldsymbol{\mu})h_n + l_m\Big]\Bigg]\\
        &\geq 0),
    \end{align*}
    where $h_n$, $l_m$ are sequences of real numbers, $h_n\to 0$ as $n\to\infty$ and $l_m\to 0$ as $m\to\infty$. In addition, we assume that for all $(t_0,\rho_0)\in[0,T]\times\mathcal{P}_2(\mathbb{R}^d)$, $\mathbb{E}F:\mathcal{P}_L^T(t_0,\rho_0)\to\mathbb{R}$ is continuous in the topology of convergence in probability. 
    \item [(iii)] For all $\mu\in\mathcal{P}_2(\mathbb{R}^d)$ such that there exists $q > 2$, $\mu\in\mathcal{P}_q(\mathbb{R}^d)$, we have
    \begin{align*}
    \displaystyle\phi_{n,m}(T,\mu)\leq \int_{\mathbb{R}^d}g(x,\mu)\mu(dx)+F(\mu)h_n+l_m,
    \end{align*}
    (resp.
    \begin{align*}
    \displaystyle\phi_{n,m}(T,\mu)\geq \int_{\mathbb{R}^d}g(x,\mu)\mu(dx)-F(\mu)h_n-l_m),
    \end{align*}
    where $F$, $h_n$ and $l_m$ are defined in (ii).
        \item [(iv)] $\{\phi_{n,m}(t,\mu)\}_{n,m\geq 1}$ is equicontinuous in $\mu$ uniformly in $\omega^0$.
        \item [(v)] $(\phi - u)^+ \in \mathcal{S}^\infty (C(\mathcal{P}_2(\mathbb{R}^d)))$, $(\phi_{n,m} - u)^+ \in \mathcal{S}^\infty (C(\mathcal{P}_2(\mathbb{R}^d)))$\\
        (resp. $(u - \phi)^+ \in \mathcal{S}^\infty (C(\mathcal{P}_2(\mathbb{R}^d)))$, $(u - \phi_{n,m})^+ \in \mathcal{S}^\infty (C(\mathcal{P}_2(\mathbb{R}^d)))$). 
    \end{enumerate}
    Then it holds that for all $(t,\mu)\in [0,T]\times\mathcal{P}_2(\mathbb{R}^d)$, $u(t,\mu)\geq(\text{resp. $\leq$}) \phi(t,\mu)$ for $\mathbb{P}^0$ a.e. $\omega^0$.
\end{theorem}
\begin{proof}\hfill
\hfill\\
\textbf{Step 1}
    Suppose to the contrary, that there exists constant $\kappa'>0$, $(t_0,\rho_0')\in [0,T]\times\mathcal{P}_2(\mathbb{R}^d)$, $\Omega_{t_0}'\in\mathcal{F}_{t_0}^0$, $\mathbb{P}^0(\Omega_{t_0}')>0$, $\kappa' \leq \phi(t_0,\rho_0')-u(t_0,\rho_0')$ for all $\omega^0\in\Omega_{t_0}'$. We claim that there exists $\kappa>0$, $\rho_0 \in \mathcal{P}_q(\mathbb{R}^d)$, $q>2$ and a $\Omega_{t_0}\in\mathcal{F}_{t_0}^0$, $\mathbb{P}^0(\Omega_{t_0})>0$ such that 
    \begin{align*}
         \kappa \leq \phi(t_0,\rho_0) - u(t_0,\rho_0)\text{ for all }\omega^0\in\Omega_{t_0}.
    \end{align*}
    Indeed, this is apparent from the following: First of all, there exists $\xi:(\tilde{\Omega}^1,\mathcal{G},\tilde{\mathbb{P}})\to\mathbb{R}^d$ such that $\mathcal{L}(\xi) = \rho_0'$. Define $\rho_{0;K} := \mathcal{L}(\xi 1_{|\xi|\leq K})$. Then by the $\mathbb{P}^0$-a.e. continuity of $u(t_0,\cdot)$ and $\phi(t_0,\cdot)$, 
    \begin{align*}
        \phi(t_0,\rho_{0;K})-u(t_0,\rho_{0;K})\to\phi(t_0,\rho_{0}')-u(t_0,\rho_{0}')\text{ for }\mathbb{P}^0\text{-a.e. } \omega^0\in\Omega_{t_0}'.
    \end{align*}
    Egorov's theorem tells us that there exists a $\Omega_{t_0}$, $\mathbb{P}^0(\Omega_{t_0})>0$, 
    \begin{align*}
        \phi(t_0,\rho_{0;K})-u(t_0,\rho_{0;K}) \to \phi(t_0,\rho_{0}')-u(t_0,\rho_{0}')\text{ uniformly for all }\omega^0\in\Omega_{t_0}, 
    \end{align*}
    thus there exists $K$ large enough such that defining $\rho_0:=\rho_{0;K}$, we have
    \begin{align*}
        0<\kappa:=\kappa'/2 \leq \phi(t_0,\rho_0) - u(t_0,\rho_0)\text{ for all }\omega^0\in\Omega_{t_0}.
    \end{align*}
    \textbf{Step 2} As our compact set $\mathcal{P}_L^{T}(0,\rho_0)$ (compact in the topology induced by the convergence in probability as proven in Theorem \ref{P_L_compact}) is a subset of Wasserstein-space-valued random variables, there is a canonical way of looking at the functions $\phi$, $u$:
    \begin{align*}
        &\boldsymbol{\phi}:[0,T]\times\mathcal{P}_L^{T}(0,\rho_0)\to L^2((\Omega^0,\mathcal{F}^0,\mathbb{P}^0);\mathbb{R}),\\
        &\boldsymbol{\phi}(t,\boldsymbol{\mu})(\omega^0)=\phi(t,\boldsymbol{\mu}(\omega^0),\omega^0).
    \end{align*}
    From Step 1 and $(v)$, there exists $\kappa > 0$ such that
    \begin{align*}
        \esssup_{\boldsymbol{\mu}\in\mathcal{P}_L^{t_0}(0,\rho_0)}\{\boldsymbol{\phi}(t_0,\boldsymbol{\mu})-\boldsymbol{u}(t_0,\boldsymbol{\mu})\}\geq \kappa\text{ for $\mathbb{P}^0$-a.e. }\omega^0 \in \Omega_{t_0}.
    \end{align*}
    First of all, we claim that this essential supremum can actually be attained. Given $\boldsymbol{\mu}$, $\boldsymbol{\mu}' \in \mathcal{P}_L^{t_0}(0,\rho_0)$, setting $\boldsymbol{\mu}^* := \boldsymbol{\mu} 1_{\boldsymbol{\phi}(t_0,\boldsymbol{\mu})-\boldsymbol{u}(t_0,\boldsymbol{\mu})>\boldsymbol{\phi}(t_0,\boldsymbol{\mu}')-\boldsymbol{u}(t_0,\boldsymbol{\mu}')} + \boldsymbol{\mu}'1_{\boldsymbol{\phi}(t_0,\boldsymbol{\mu})-\boldsymbol{u}(t_0,\boldsymbol{\mu})\leq \boldsymbol{\phi}(t_0,\boldsymbol{\mu}')-\boldsymbol{u}(t_0,\boldsymbol{\mu}')}$, then
    \begin{align*}
        \boldsymbol{\phi}(t_0,\boldsymbol{\mu}^*) -\boldsymbol{u}(t_0,\boldsymbol{\mu}^*) = \max\{\boldsymbol{\phi}(t_0,\boldsymbol{\mu}) -\boldsymbol{u}(t_0,\boldsymbol{\mu}),\boldsymbol{\phi}(t_0,\boldsymbol{\mu}') -\boldsymbol{u}(t_0,\boldsymbol{\mu}')\}.
    \end{align*}
    Note that our constructed $\mathcal{P}_L^{t_0}$ is closed under finite partition addition, therefore $\boldsymbol{\mu}^*\in\mathcal{P}_L^{t_0}$. Thus there exists a sequence $\{\boldsymbol{\mu}_k\}_{k\in \mathbb{N}}\in\mathcal{P}_L^{t_0}(0,\rho_0)$ such that
    \begin{align*}
        \lim_k \boldsymbol{\phi}(t_0,\boldsymbol{\mu}_k) -\boldsymbol{u}(t_0,\boldsymbol{\mu}_k) = \esssup_{\boldsymbol{\mu}\in\mathcal{P}_L^{t_0}(0,\rho_0)}\{\boldsymbol{\phi}(t_0,\boldsymbol{\mu})-\boldsymbol{u}(t_0,\boldsymbol{\mu})\}.
    \end{align*}
    Moreover, since $\mathcal{P}_L^{t_0}(0,\rho_0)$ is compact, there exists $\boldsymbol{\mu}_\infty$ such that (up to a subsequence) $\boldsymbol{\mu}_k\to\boldsymbol{\mu}_\infty$ in probability, and there exists a subsequence $\boldsymbol{\mu}_{n_k}\to\boldsymbol{\mu}_\infty$ $\mathbb{P}^0$-a.e. $\omega^0$. Since for $\mathbb{P}^0$-a.e. $\omega^0$, $\phi(t_0,\cdot)-u(t_0,\cdot)$ is continuous,
    \begin{align*}
        \esssup_{\boldsymbol{\mu}\in\mathcal{P}_L^{t_0}(0,\rho_0)}\{\boldsymbol{\phi}(t_0,\boldsymbol{\mu})-\boldsymbol{u}(t_0,\boldsymbol{\mu})\}(\omega^0) =& \lim_k \phi(t_0,\boldsymbol{\mu}_{n_k}(\omega^0),\omega^0) -u(t_0,\boldsymbol{\mu}_{n_k}(\omega^0),\omega^0) \\
        =& \phi(t_0,\boldsymbol{\mu}_\infty(\omega^0),\omega^0) -u(t_0,\boldsymbol{\mu}_\infty(\omega^0),\omega^0),\text{ for $\mathbb{P}^0$-a.e. $\omega^0\in\Omega_{t_0}$}. 
    \end{align*}
    Therefore the essential supremum is attained by $\boldsymbol{\mu}_\infty$.\\
    \hfill\\
    \textbf{Step 3} Moreover, for all $\boldsymbol{\mu} \in \mathcal{P}_L^{T}(0,\rho_0)$, as $\mathcal{P}_L^T(0,\rho_0)$ is compact, and by $(ii)$, there exists a constant $C(\mathcal{P}_L^{T}(0,\rho_0))$ depending on $\mathcal{P}_L^{T}(0,\rho_0)$ such that $\mathbb{E}F(\boldsymbol{\mu})h_n \leq C(\mathcal{P}_L^{T}(0,\rho_0))h_n$, for all $\boldsymbol{\mu}\in\mathcal{P}_L^T(0,\rho_0)$. Hence one can choose $N_1$ such that for all $n,m\geq N_1$, $\boldsymbol{\mu} \in \mathcal{P}_L^{T}(0,\rho_0)$:
    \begin{align*}
        \max\{\mathbb{E}F(\boldsymbol{\mu})h_n,l_m\} < \min\Big\{\frac{\kappa}{16(T-t_0)},\frac{\kappa}{8}\Big\}.
    \end{align*}
    \textbf{Step 4} We claim that
    \begin{align*}
        \lim_n \lim_m\esssup_{\boldsymbol{\mu}\in\mathcal{P}_L^{t_0}(0,\rho_0)}\{\boldsymbol{\phi}_{n,m}(t_0,\boldsymbol{\mu})-\boldsymbol{u}(t_0,\boldsymbol{\mu})\} = \esssup_{\boldsymbol{\mu}\in\mathcal{P}_L^{t_0}(0,\rho_0)}\{\boldsymbol{\phi}(t_0,\boldsymbol{\mu})-\boldsymbol{u}(t_0,\boldsymbol{\mu})\}\text{ in probability}. 
    \end{align*}
    To see this, note that
    \begin{enumerate}[(1)]
        \item  $\mathcal{P}_{L}^{t_0}(0,\rho_0)$ is compact in the topology of convergence in probability,
        \item For a metric space $(M,d)$, let $\xi$ and $\eta$ be $M$-valued random variables, then $d_P(\xi,\eta):=\mathbb{E}[d(\xi,\eta)\wedge 1]$ generates the topology of convergence in probability. Given $\varepsilon > 0$, by $(iv)$, there exists $\delta$ such that if $\mathcal{W}_2(\mu',\mu)< \delta$, $|\phi_{n,m}(t_0,\mu',\omega) - \phi_{n,m}(t_0,\mu,\omega)|< \varepsilon/2$, independent of $n$, $m$, $\mu$ and $\omega$. Without loss of generality, assume $\delta<1$. So if $d_P(\boldsymbol{\mu}',\boldsymbol{\mu}) = \mathbb{E}[\mathcal{W}_2(\boldsymbol{\mu},\boldsymbol{\mu}')\wedge 1]<\delta':=\frac{\varepsilon}{2}\delta$, then
        \begin{align*}
            \frac{\varepsilon}{2}\delta >& \mathbb{E}[\mathcal{W}_2(\boldsymbol{\mu}',\boldsymbol{\mu})\wedge 1]\\
            =& \mathbb{E}[(\mathcal{W}_2(\boldsymbol{\mu},\boldsymbol{\mu}')\wedge 1) 1_{\mathcal{W}_2(\boldsymbol{\mu},\boldsymbol{\mu}')<\delta}]+\mathbb{E}[(\mathcal{W}_2(\boldsymbol{\mu},\boldsymbol{\mu}')\wedge 1) 1_{\mathcal{W}_2(\boldsymbol{\mu},\boldsymbol{\mu}')\geq \delta}]\\
            \geq& \delta \mathbb{P}^0(\mathcal{W}_2(\boldsymbol{\mu},\boldsymbol{\mu}')\geq \delta).
        \end{align*}
        Thus
        \begin{align*}
            &\mathbb{E}[|\boldsymbol{\phi}_{n,m}(t_0,\boldsymbol{\mu}')-\boldsymbol{\phi}_{n,m}(t_0,\boldsymbol{\mu})|\wedge 1]\\
            =& \mathbb{E}[(|\boldsymbol{\phi}_{n,m}(t_0,\boldsymbol{\mu}')-\boldsymbol{\phi}_{n,m}(t_0,\boldsymbol{\mu})|\wedge 1) 1_{\mathcal{W}_2(\boldsymbol{\mu},\boldsymbol{\mu}')<\delta}]+\mathbb{E}[(|\boldsymbol{\phi}_{n,m}(t_0,\boldsymbol{\mu}')-\boldsymbol{\phi}_{n,m}(t_0,\boldsymbol{\mu})|\wedge 1) 1_{\mathcal{W}_2(\boldsymbol{\mu},\boldsymbol{\mu}')\geq\delta}]\\
            \leq& \varepsilon/2+\mathbb{E}[1_{\mathcal{W}_2(\boldsymbol{\mu},\boldsymbol{\mu}')\geq\delta}] \leq \varepsilon.
        \end{align*}
    \end{enumerate}
    Thus there exists $\delta'$ such that if $d_P(\boldsymbol{\mu},\boldsymbol{\mu}')<\delta'$, then $d_P(\boldsymbol{\phi}_{n,m}(t_0,\boldsymbol{\mu}),\boldsymbol{\phi}_{n,m}(t_0,\boldsymbol{\mu}'))<\varepsilon$, and it is independent of $n$, $m$ $\mu$, so $\boldsymbol{\phi}_{n,m}$'s are continuous in probability in $\boldsymbol{\mu}$ uniformly in $n$, $m$, $\omega^0$. Since $\mathcal{P}_{L}^{t_0}(0,\rho_0)$ is compact, there exists a finite cover $\cup_{i=1}^r B_{\delta'}^{d_P}(\boldsymbol{\mu}_i)$. Thanks to our Step 1, there exists $q>2$ such that $\boldsymbol{\mu}_i$'s are $\mathcal{P}_q(\mathbb{R}^d)$-valued random variables, as $\rho_0 \in\mathcal{P}_q(\mathbb{R}^d)$. By $(i)$, $\lim_n\lim_m \phi_{n,m}(t_0,\boldsymbol{\mu}_i(\omega^0),\omega^0) = \phi(t_0,\boldsymbol{\mu}_i(\omega^0),\omega^0)$ $\mathbb{P}^0$ a.e., so $\boldsymbol{\phi}_{n,m}(t_0,\boldsymbol{\mu}_i)\to\boldsymbol{\phi}(t_0,\boldsymbol{\mu}_i)$ in probability as $m\to\infty$, then $n\to\infty$. There exists $N$ such that for all $n\geq N$, there exists $M:=M(n)$ such that for all $m\geq M$, 
    \begin{align*}
        d_P(\boldsymbol{\phi}_{n,m}(t_0,\boldsymbol{\mu}_i),\boldsymbol{\phi}(t_0,\boldsymbol{\mu}_i)) < \varepsilon,\text{ for all $1 \leq i \leq r$.}
    \end{align*}
    Now let $\boldsymbol{\mu} \in \mathcal{P}_{L}^{t_0}(0,\rho_0)$ be arbitrary, there exists some $i'$ such that $\boldsymbol{\mu} \in B_{\delta'}^{d_P}(\boldsymbol{\mu}_{i'})$, then for all $n\geq N$, $m\geq M(n)$,
    \begin{align*}
        d_P(\boldsymbol{\phi}_{n,m}(t_0,\boldsymbol{\mu}),\boldsymbol{\phi}(t_0,\boldsymbol{\mu}))\leq d_P(\boldsymbol{\phi}_{n,m}(t_0,\boldsymbol{\mu}),\boldsymbol{\phi}_{n,m}(t_0,\boldsymbol{\mu}_{i'}))+d_P(\boldsymbol{\phi}_{n,m}(t_0,\boldsymbol{\mu}_{i'}),\boldsymbol{\phi}(t_0,\boldsymbol{\mu}_{i'}))\leq 2\varepsilon. 
    \end{align*}
    Hence the convergence of $(i)$ can be made uniform in $\boldsymbol{\mu}\in\mathcal{P}_{L}^{t_0}(0,\rho_0)$ in the sense of probability. It is then easy to see that 
    \begin{align*}
    \lim_n\lim_m\esssup_{\boldsymbol{\mu}\in\mathcal{P}_L^{t_0}(0,\rho_0)}\{\boldsymbol{\phi}_{n,m}(t_0,\boldsymbol{\mu})-\boldsymbol{u}(t_0,\boldsymbol{\mu})\} = \esssup_{\boldsymbol{\mu}\in\mathcal{P}_L^{t_0}(0,\rho_0)}\{\boldsymbol{\phi}(t_0,\boldsymbol{\mu})-\boldsymbol{u}(t_0,\boldsymbol{\mu})\}\,\,\,\text{in probability}.
    \end{align*}
    Indeed, we have
    \begin{align*}
        &\mathbb{P}\Bigg(\Big|\esssup_{\boldsymbol{\mu}\in\mathcal{P}_L^{t_0}(0,\rho_0)}\{\boldsymbol{\phi}_{n,m}(t_0,\boldsymbol{\mu})-\boldsymbol{u}(t_0,\boldsymbol{\mu})\} - \esssup_{\boldsymbol{\mu}\in\mathcal{P}_L^{t_0}(0,\rho_0)}\{\boldsymbol{\phi}(t_0,\boldsymbol{\mu})-\boldsymbol{u}(t_0,\boldsymbol{\mu})\}\Big|\geq \varepsilon\Bigg)\\
        \leq &\mathbb{P}\Bigg(\esssup_{\boldsymbol{\mu}\in\mathcal{P}_L^{t_0}(0,\rho_0)}\{\boldsymbol{\phi}_{n,m}(t_0,\boldsymbol{\mu})-\boldsymbol{u}(t_0,\boldsymbol{\mu})\} - \esssup_{\boldsymbol{\mu}\in\mathcal{P}_L^{t_0}(0,\rho_0)}\{\boldsymbol{\phi}(t_0,\boldsymbol{\mu})-\boldsymbol{u}(t_0,\boldsymbol{\mu})\}\geq \varepsilon\Bigg)\\
        &+\mathbb{P}\Bigg(\esssup_{\boldsymbol{\mu}\in\mathcal{P}_L^{t_0}(0,\rho_0)}\{\boldsymbol{\phi}_{n,m}(t_0,\boldsymbol{\mu})-\boldsymbol{u}(t_0,\boldsymbol{\mu})\} - \esssup_{\boldsymbol{\mu}\in\mathcal{P}_L^{t_0}(0,\rho_0)}\{\boldsymbol{\phi}(t_0,\boldsymbol{\mu})-\boldsymbol{u}(t_0,\boldsymbol{\mu})\}\leq -\varepsilon\Bigg)\\
        \leq &\mathbb{P}\Bigg(\esssup_{\boldsymbol{\mu}\in\mathcal{P}_L^{t_0}(0,\rho_0)}\{\boldsymbol{\phi}_{n,m}(t_0,\boldsymbol{\mu})-\boldsymbol{\phi}(t_0,\boldsymbol{\mu})\}\geq \varepsilon\Bigg)\\  &+\mathbb{P}\Bigg(\boldsymbol{\phi}_{n,m}(t_0,\boldsymbol{\mu}^*)-\boldsymbol{\phi}(t_0,\boldsymbol{\mu}^*)\leq -\varepsilon \Bigg),\\
        \to& 0 \text{ as }m\to \infty,\text{ then }n\to\infty,
    \end{align*}
    where $\boldsymbol{\phi}(t_0,\boldsymbol{\mu}^*)-\boldsymbol{u}(t_0,\boldsymbol{\mu}^*)=\esssup_{\boldsymbol{\mu}\in\mathcal{P}_L^{t_0}(0,\rho_0)}\{\boldsymbol{\phi}(t_0,\boldsymbol{\mu})-\boldsymbol{u}(t_0,\boldsymbol{\mu})\}$.\\
    \hfill\\
    \textbf{Step 5} For all $n,m$, there exists $\boldsymbol{\mu}_{t_0}^{n,m} \in \mathcal{F}_{t_0}^0$ such that
    \begin{align*}
        \boldsymbol{\phi}_{n,m}(t_0,\boldsymbol{\mu}_{t_0}^{n,m})-u(t_0,\boldsymbol{\mu}_{t_0}^{n,m}) = \esssup_{\boldsymbol{\mu}\in\mathcal{P}_L^{t_0}(0,\rho_0)}\{\boldsymbol{\phi}_{n,m}(t_0,\boldsymbol{\mu})-\boldsymbol{u}(t_0,\boldsymbol{\mu})\}, 
    \end{align*}
    and let 
    \begin{align*}
        \boldsymbol{\phi}(t_0,\boldsymbol{\mu}_{t_0})-u(t_0,\boldsymbol{\mu}_{t_0}) := \esssup_{\boldsymbol{\mu}\in\mathcal{P}_L^{t_0}(0,\rho_0)}\{\boldsymbol{\phi}(t_0,\boldsymbol{\mu})-\boldsymbol{u}(t_0,\boldsymbol{\mu})\}.
    \end{align*}
    With an argument similar to Step 1, up to subsequence, there exists $N_2$ such that $\forall n \geq N_2$, there exists $M_2(n)$ such that for $m\geq M_2$,
    \begin{align*}
        \boldsymbol{\phi}_{n,m}(t_0,\boldsymbol{\mu}_{t_0}^{n,m}) - \boldsymbol{u}(t_0,\boldsymbol{\mu}_{t_0}^{n,m}) \geq \esssup_{\boldsymbol{\mu}\in\mathcal{P}_L^{t_0}(0,\rho_0)}\{\boldsymbol{\phi}(t_0,\boldsymbol{\mu})-\boldsymbol{u}(t_0,\boldsymbol{\mu})\}-\kappa/2\geq \kappa/2>0\text{ for almost }\omega\in\Omega_{t_0}^{n,m},
    \end{align*}
    for some $\Omega_{t_0}^{n,m}$, $\mathbb{P}(\Omega_{t_0}^{n,m})>0$. Now we fix an $n \geq \max(N_1,N_2)$ and $m\geq \max(N_1,M_2(n))$. Define \begin{align*}
    \alpha:=\boldsymbol{\phi}_{n,m}(t_0,\boldsymbol{\mu}_{t_0}^{n,m}) - \boldsymbol{u}(t_0,\boldsymbol{\mu}_{t_0}^{n,m}),
    \end{align*}
    and we work on the induced $\Omega_{t_0}^{n,m}$. For each $s\in(t_0,T]$, get an $\mathcal{F}_s^0$ measurable r.v. $\boldsymbol{\mu}^{n,m}_s$ such that
    \begin{align*}
        (\boldsymbol{\phi}_{n,m}(s,\boldsymbol{\mu}_s^{n,m})-\boldsymbol{u}(s,\boldsymbol{\mu}_s^{n,m}))^+ = \esssup_{\boldsymbol{\mu}\in\mathcal{P}_L^s(0,\rho)}(\boldsymbol{\phi}_{n,m}(s,\boldsymbol{\mu})-\boldsymbol{u}(s,\boldsymbol{\mu}))^+.
    \end{align*}
    Set 
    \begin{align*}
        Y_s :=& -(\boldsymbol{\phi}_{n,m}(s,\boldsymbol{\mu}_s^{n,m})-\boldsymbol{u}(s,\boldsymbol{\mu}_s^{n,m}))^+-\frac{\alpha(s-t_0)}{2(T-t_0)},\\
        Z_s :=& \essinf_{\tau\in\mathcal{T}_{s,T}^0}\mathbb{E}_{\mathcal{F}_s^0}[Y_\tau].
    \end{align*}
    From our assumption, the time continuity of $(\boldsymbol{\phi}_{n,m}(s,\boldsymbol{\mu}^{n,m}_s)-\boldsymbol{u}(s,\boldsymbol{\mu}^{n,m}_s))^+$ follows. Therefore, the process $(Y_s)_{t_0\leq s\leq T}$ has continuous trajectories. Define $\tau:= \inf\{s\geq t_0 : Y_s=Z_s\}$. Note that
    \begin{align*}
        \boldsymbol{\phi}_{n,m}(T,\boldsymbol{\mu}_T^{n,m})\leq& \int_{\mathbb{R}^d}g(x,\boldsymbol{\mu}_T^{n,m})\boldsymbol{\mu}_T^{n,m}(dx) +F(\boldsymbol{\mu}_T^{n,m})h_n+l_m\\
        \leq&\boldsymbol{u}(T,\boldsymbol{\mu}_T^{n,m})+F(\boldsymbol{\mu}_T^{n,m})h_n+l_m
    \end{align*}
    implying
    \begin{align*}
        (\boldsymbol{\phi}_{n,m}(T,\boldsymbol{\mu}_T^{n,m})-\boldsymbol{u}(T,\boldsymbol{\mu}_{T}^{n,m}))^+\leq F(\boldsymbol{\mu}_T^{n,m})h_n+l_m,
    \end{align*}
    so
    \begin{align*}
        \mathbb{E}(\boldsymbol{\phi}_{n,m}(T,\boldsymbol{\mu}_T^{n,m})-\boldsymbol{u}(T,\boldsymbol{\mu}_T^{n,m}))^+< \kappa/4 < \mathbb{E}\alpha/2.
    \end{align*}
    By optimal stopping theory, in particular Snell's envelope, 
    \begin{align*}
        \mathbb{E}Y_T >-\mathbb{E}\alpha/2 - \mathbb{E}\alpha/2 = -\mathbb{E}\alpha = \mathbb{E}Y_{t_0}\geq \mathbb{E}Z_{t_0} = \mathbb{E}\mathbb{E}_{\mathcal{F}_{t_0}^0}Y_{\tau} = \mathbb{E}Y_\tau, 
    \end{align*}
    thus $\mathbb{P}(\tau < T) >0$. As
    \begin{align*}
        -(\boldsymbol{\phi}_{n,m}(\tau,\boldsymbol{\mu}_\tau^{n,m})-\boldsymbol{u}(s,\boldsymbol{\mu}_s^{n,m}))^+-\frac{\alpha(\tau-t_0)}{2(T-t_0)} = Z_\tau \leq \mathbb{E}_{\mathcal{F}^0_\tau}[Y_T] \leq -\frac{\alpha}{2},
    \end{align*}
    we have $\mathbb{P}((\boldsymbol{\phi}_{n,m}(\tau,\boldsymbol{\mu}_\tau^{n,m})-\boldsymbol{u}(\tau,\boldsymbol{\mu}_\tau^{n,m}))^+>0)>0$. Define $\hat{\tau} = \inf\{s\geq \tau : (\boldsymbol{\phi}_{n,m}(s,\boldsymbol{\mu}_s^{n,m})-\boldsymbol{u}(s,\boldsymbol{\mu}_s^{n,m}))^+\leq 0\}$. Put $\Omega_{\tau} = \{\tau<\hat{\tau}\}$, then $\Omega_{\tau} = \{\tau<\hat{\tau}\} \in \mathcal{F}_\tau^0$ and $\mathbb{P}(\Omega_\tau)>0$. Set $\Phi(s,\mu):= \phi_{n,m}(s,\mu)+\frac{\alpha(s-t_0)}{2(T-t_0)}+\mathbb{E}_{\mathcal{F}_s^0}[Y_\tau]$, then $\Phi \in \mathscr{S}$. For each $\bar{\tau}\in\mathcal{T}^0_{\tau,T}$, we have for almost all $\omega^0\in\Omega_\tau$,
    \begin{align*}
        (\Phi - u)(\tau,\boldsymbol{\mu}_\tau^{n,m}) = 0 = Y_\tau - Z_\tau \geq Y_\tau - \mathbb{E}_{\mathcal{F}^0_\tau}[Y_{\bar{\tau}\wedge\hat{\tau}}] = \mathbb{E}_{\mathcal{F}^0_\tau}\Bigg[\esssup_{\boldsymbol{\mu}\in\mathcal{P}_{L}^{\hat{\tau}\wedge\bar{\tau}}(0,\rho_0)}(\Phi-u)(\bar{\tau}\wedge\hat{\tau},\boldsymbol{\mu})\Bigg],
    \end{align*}
   which together with the arbitrariness of $\bar{\tau}$ implies that $\Phi\in \overline{\mathcal{G}}u(\tau,\boldsymbol{\mu}_\tau^{n,m},\hat{\tau};0,\rho_0;\Omega_\tau)$. As $u$ is a viscosity supersolution, we have that for almost all $\omega^0\in\Omega_\tau$,
    \begin{align*}
        0\leq &\esslimsup_{\substack{(s,\boldsymbol{\mu})\to (\tau^+,\boldsymbol{\mu}_\tau^{n,m})\\
        \boldsymbol{\mu}\in \mathcal{P}_L^s(t_0,\rho_0)}}\mathbb{E}_{\mathcal{F}^0_\tau}\Bigg[1_{s\leq\hat{\tau}}\Big[-\Game_s \Phi\\
        &- \mathbb{H}\Big(s,\boldsymbol{\mu},\partial_\mu \Phi(s,\boldsymbol{\mu})(\cdot),\partial_x\partial_\mu \Phi(s,\boldsymbol{\mu})(\cdot),\partial_\mu^2 \Phi(s,\boldsymbol{\mu})(\cdot,\cdot),\partial_\mu\Game_w \Phi(s,\boldsymbol{\mu})(\cdot)\Big)\Bigg]\\
        \leq &-\frac{\alpha}{2(T-t)}+\esslimsup_{\substack{(s,\boldsymbol{\mu})\to (\tau^+,\boldsymbol{\mu}_\tau^{n,m})\\
        \boldsymbol{\mu}\in \mathcal{P}_L^s(t_0,\rho_0)}}\mathbb{E}_{\mathcal{F}^0_\tau}\Bigg[1_{s\leq\hat{\tau}}\Big[-\Game_s \phi_{n,m}\\
        &- \mathbb{H}\Big(s,\boldsymbol{\mu},\partial_\mu \phi_{n,m}(s,\boldsymbol{\mu})(\cdot),\partial_x\partial_\mu \phi_{n,m}(s,\boldsymbol{\mu})(\cdot),\partial_\mu^2 \phi_{n,m}(s,\boldsymbol{\mu})(\cdot,\cdot),\partial_\mu\Game_w \phi_{n,m}(s,\boldsymbol{\mu})(\cdot)\Big)\Bigg]\\
        \leq &-\frac{\kappa}{4(T-t)}+F(\boldsymbol{\mu}_\tau^{n,m})h_n + l_m.
    \end{align*}
    Taking expectation on both sides leads to
    \begin{align*}
        0\leq &-\frac{\kappa}{4(T-t)}+\mathbb{E}F(\boldsymbol{\mu}_\tau^{n,m})h_n + l_m
        \leq -\frac{\kappa}{8(T-t)},
    \end{align*}
   which is a contradiction.
\end{proof}
\subsubsection{Uniqueness}
To proceed, we have to assume one of the following additional assumptions on the diffusion coefficients $\sigma$, $\sigma^0$:
\begin{assumption}
    \begin{enumerate}
    \item [($\mathcal{B}$1)]
    \label{uniqueness_assumption_degenerate} For possible degenerate coefficients $\sigma$, $\sigma^0$, the diffusion coefficient $\sigma$, $\sigma^0: [0,T]\to\mathbb{R}^{d\times m}$ does not depend on $(\omega^0,x,\mu,a)\in\Omega^0\times\mathbb{R}^d\times\mathcal{P}_2(\mathbb{R}^d)\times A$.
    \item [($\mathcal{B}$2)]
    \label{uniqueness_assumption_superparabolic} The diffusion coefficients $\sigma:[0,T]\times \mathbb{R}^d\times A \to\mathbb{R}^{d\times m}$, $\sigma^0:[0,T]\times \mathbb{R}^d\to\mathbb{R}^{d\times m}$, and does not depend on $\omega$, and there exists $\lambda > 0$ such that $\forall (t,x,a,\xi)\in [0,T]\times\mathbb{R}^d\times A\times\mathbb{R}^d$, 
    \begin{align*}
        \sum_{i,j=1}^d\sum_{k=1}^m \sigma^{ik}\sigma^{jk}(t,x,a)\xi^i\xi^j\wedge\sum_{i,j=1}^d\sum_{k=1}^m \sigma^{0;ik}\sigma^{0;jk}(t,x)\xi^i\xi^j\geq \lambda |\xi|^2.
    \end{align*}
     \end{enumerate}
   
\end{assumption}
\begin{theorem}
\label{uniqueness_full}
    Let ($\mathcal{A}1$), ($\mathcal{A}2$) and ($\mathcal{A}3$) hold. Under either ($\mathcal{B}1$) or ($\mathcal{B}2$), the value function $v$ is the unique viscosity solution.
\end{theorem}
\begin{proof}
    We will present the proof under Assumption ($\mathcal{B}$1). It is worth noting that the proof under Assumption ($\mathcal{B}$2) is analogous and simpler, as there is no necessity to incorporate additional auxiliary Brownian motions to handle potential degeneracy. Define 
    \begin{align*}
        \overline{\mathscr{V}}:=\Bigg\{\phi\in\mathscr{S}\,\Big|\,\phi \text{ satisfy the corresponding viscosity subsolution version of  }(i)-(v)\text{ in Theorem }\ref{uniqueness_result}\Bigg\},
    \end{align*}
    and
    \begin{align*}
        \underline{\mathscr{V}}:=\Bigg\{\phi\in\mathscr{S}\,\Big|\,\phi \text{ satisfy the corresponding viscosity supersolution version of }(i)-(v)\text{ in Theorem }\ref{uniqueness_result}\Bigg\}.
    \end{align*}
    Set 
    \begin{align*}
        \overline{u}:= \essinf_{\phi\in\overline{\mathscr{V}}} \phi,\,\,\,\underline{u}:= \esssup_{\phi\in\underline{\mathscr{V}}} \phi. 
    \end{align*}
    In view of Theorem \ref{uniqueness_result}, each viscosity solution $u$ satisfy $\underline{u}\leq u \leq \overline{u}$. So to establish the uniqueness, it suffices to show that $\underline{u} = v = \overline{u}$, where $v$ is our value function.\\
    \hfill\\
    \textbf{Step 1.} For each fixed $\varepsilon\in(0,1)$, choose $(g^\varepsilon,f^\varepsilon,b^\varepsilon)$ and $(g^N,f^N,b^N)$ as in Lemma \ref{approximation_markovian}. Recall that $\{t_i\}_{i=1,\ldots,N}$ denotes the time partition from Lemma \ref{approximation_markovian}. Without loss of generality, we focus our attention on the time interval $t \in [t_{N-1},t_{N}]$, as similar arguments can be applied to the time intervals $[t_{N-2},t_{N-1})$, and the results obtained on $[t_{N-1},t_{N}]$ can be utilized as the terminal value, and we could repeat until the last interval $[0,t_1)$. We enlarge our probability space as described in Appendix \ref{n_player_approximation}. Recall the definition of $\bar{v}^N_{\varepsilon^0,\varepsilon^1,n,m}$ in Theorem \ref{approximation_properties}. Applying the finite dimensional version of It\^o-Kunita Formula to $\bar{v}^N_{\varepsilon^0,\varepsilon^1,n,m}$, together with stochastic Fubini theorem we conclude that
    \begin{align*}
        &-d\bar{v}_{\varepsilon^0,\varepsilon^1,n,m}^N (t,x_1-\varepsilon^0\bar{B}^0_t-\varepsilon^1\bar{B}^1_t,\ldots,x_n-\varepsilon^0\bar{B}^0_t-\varepsilon^1\bar{B}^n_t,W_t^0)\\
        =&\Bigg(\sum_{i=1}^n\essinf_{a_i\in A}\Bigg\{\frac{1}{n}f_{n,m}^{N;i} (t,W_{t_1}^0,\ldots,W_{t_{N-1}}^0,W_t^0,x_1-\varepsilon^0\bar{B}^0_t-\varepsilon^1\bar{B}^1_t,\ldots,x_n-\varepsilon^0\bar{B}^0_t-\varepsilon^1\bar{B}^n_t,a_i) \\
            &\displaystyle+ \Big\langle b_{n,m}^{N;i}(t,W_{t_1}^0,\ldots,W_{t_{N-1}}^0,W_t^0,x_1-\varepsilon^0\bar{B}^0_t-\varepsilon^1\bar{B}^1_t,\ldots,x_n-\varepsilon^0\bar{B}^0_t-\varepsilon^1\bar{B}^n_t,a_i),\\
        &\,\,\,\,\,\,\,\,\,\partial_{x_i}\bar{v}^N_{\varepsilon^0,\varepsilon^1,n,m}(t,x_1-\varepsilon^0\bar{B}^0_t-\varepsilon^1\bar{B}^1_t,\ldots,x_n-\varepsilon^0\bar{B}^0_t-\varepsilon^1\bar{B}^n_t,W_t^0) \Big\rangle\Bigg\}\\
        &\displaystyle+ \sum_{i=1}^n \frac{1}{2}\tr\Big[\big(\sigma_t\sigma^\intercal_t+\sigma^0_t\sigma^{0;\intercal}_t
        \big)\partial_{x_i x_i}^2 \bar{v}^N_{\varepsilon^0,\varepsilon^1,n,m}(t,x_1-\varepsilon^0\bar{B}^0_t-\varepsilon^1\bar{B}^1_t,\ldots,x_n-\varepsilon^0\bar{B}^0_t-\varepsilon^1\bar{B}^n_t,W_t^0)\Big]
        \\
        &+\displaystyle\sum_{i=1}^n \tr\big(\sigma^0_t\partial_{x_i y}^2\bar{v}^N_{\varepsilon^0,\varepsilon^1,n,m}(t,x_1-\varepsilon^0\bar{B}^0_t-\varepsilon^1\bar{B}^1_t,\ldots,x_n-\varepsilon^0\bar{B}^0_t-\varepsilon^1\bar{B}^n_t,W_t^0)\big)\\
        &\displaystyle+\frac{1}{2}\sum_{i,j=1,i\neq j}^n\tr\Big[\big(\sigma^0_t\sigma^{0;\intercal}_t\big)\partial^2_{x_i x_j}\bar{v}^N_{\varepsilon^0,\varepsilon^1,n,m}(t,x_1-\varepsilon^0\bar{B}^0_t-\varepsilon^1\bar{B}^1_t,\ldots,x_n-\varepsilon^0\bar{B}^0_t-\varepsilon^1\bar{B}^n_t,W_t^0)\Big]\Bigg)dt\\
        &+\varepsilon^0\sum_{i=1}^n \partial_{x_i}\bar{v}_{\varepsilon^0,\varepsilon^1,n,m}^N(t,x_1-\varepsilon^0\bar{B}^0_t-\varepsilon^1\bar{B}^1_t,\ldots,x_n-\varepsilon^0\bar{B}^0_t-\varepsilon^1\bar{B}^n_t,W_t^0)d\bar{B}^0_t\\
        &+\varepsilon^1\sum_{i=1}^n\partial_{x_i} \bar{v}_{\varepsilon^0,\varepsilon^1,n,m}^N(t,x_1-\varepsilon^0\bar{B}^0_t-\varepsilon^1\bar{B}^1_t,\ldots,x_n-\varepsilon^0\bar{B}^0_t-\varepsilon^1\bar{B}^n_t,W_t^0) d\bar{B}^i_t\\
        &-\partial_y \bar{v}_{\varepsilon^0,\varepsilon^1,n,m}^N(t,x_1-\varepsilon^0\bar{B}^0_t-\varepsilon^1\bar{B}^1_t,\ldots,x_n-\varepsilon^0\bar{B}^0_t-\varepsilon^1\bar{B}^n_t,W_t^0) dW_t^0.
    \end{align*}
    Let us denote the normal distribution with mean $\mu$ and variance $\sigma^2$ by $N(\mu, \sigma^2)$. Now, we  take the expectation with respect to $\mu\otimes\ldots\otimes\mu$, and integrate out the Brownian motions $\bar{B}^0$, $\bar{B}^i$, where $i = 1,\ldots,n$. This means we are considering the expectation with respect to the probabilities $\mathbb{P}^{0'}$ and $\mathbb{P}^{1'}$ as outlined in Appendix \ref{n_player_approximation}. For simplicity, we denote the expectation with respect to $\mathbb{P}^{0'}$ as $\mathbb{E}^{0'}$ and the expectation with respect to $\mathbb{P}^{1'}$ as $\mathbb{E}^{1'}$. Defining \begin{align}
    \label{definition_mathfrak_v}
    \mathfrak{v}^{N}_{\varepsilon^0,\varepsilon^1,n,m}(t,\mu):=& \mathbb{E}^{0'}v^N_{\varepsilon^0,\varepsilon^1,n,m}\Big(t,\mu\ast N(\varepsilon^0 \bar{B}_t^0,(\varepsilon^1)^2t),W_t^0\Big)\\
    = &\mathbb{E}^{0
        '}\mathbb{E}^{1'}\int_{\mathbb{R}^{nd}}\bar{v}_{\varepsilon^0,\varepsilon^1,n,m}^N(t,x_1-\varepsilon^0\bar{B}^0_t-\varepsilon^1\bar{B}^1_t,\ldots,x_n-\varepsilon^0\bar{B}^0_t-\varepsilon^1\bar{B}^n_t,W_t^0)\mu(dx_1)\ldots\mu(dx_n),\nonumber
    \end{align}
    we have
    \begin{align}
    \label{def_bar_v}
        &-d\mathfrak{v}^{N}_{\varepsilon^0,\varepsilon^1,n,m}(t,\mu)\nonumber\\
        =&\Bigg(\mathbb{E}^{0'}\mathbb{E}^{1'}\int_{\mathbb{R}^{nd}}\sum_{i=1}^n\essinf_{a_i\in A}\Bigg\{\frac{1}{n}f_{n,m}^{N;i} (t,W_{t_1}^0,\ldots,W_{t_{N-1}}^0,W_t^0,x_1-\varepsilon^0\bar{B}^0_t-\varepsilon^1\bar{B}^1_t,\ldots,x_n-\varepsilon^0\bar{B}^0_t-\varepsilon^1\bar{B}^n_t,a_i)\nonumber \\
            &\nonumber\displaystyle+ \Big\langle b_{n,m}^{N;i}(t,W_{t_1}^0,\ldots,W_{t_{N-1}}^0,W_t^0,x_1-\varepsilon^0\bar{B}^0_t-\varepsilon^1\bar{B}^1_t,\ldots,x_n-\varepsilon^0\bar{B}^0_t-\varepsilon^1\bar{B}^n_t,a_i),\\
        &\nonumber\,\,\,\,\,\,\,\,\,\partial_{x_i}\bar{v}^N_{\varepsilon^0,\varepsilon^1,n,m}(t,x_1-\varepsilon^0\bar{B}^0_t-\varepsilon^1\bar{B}^1_t,\ldots,x_n-\varepsilon^0\bar{B}^0_t-\varepsilon^1\bar{B}^n_t,W_t^0) \Big\rangle\Bigg\}\mu(dx_1)\otimes\ldots\otimes\mu(dx_n)\\
        &\nonumber\displaystyle+ \mu\Big(\frac{1}{2}\tr\Big[\big(\sigma_t\sigma^\intercal_t+\sigma^0_t\sigma^{0;\intercal}_t
        \big)\partial_{x}\partial_{\mu} \mathfrak{v}^N_{\varepsilon^0,\varepsilon^1,n,m}(t,\mu)\Big]\Big)+\mu\Big(\tr\big(\sigma^0_t\partial_\mu\Game_w\mathfrak{v}^{N}_{\varepsilon^0,\varepsilon^1,n,m}(t,\mu)\big)\Big)\\
        &\displaystyle+(\mu\otimes\mu)\Big(\frac{1}{2}\tr\Big[\big(\sigma^0_t\sigma^{0;\intercal}_t\big)\partial^2_{\mu}\mathfrak{v}^N_{\varepsilon^0,\varepsilon^1,n,m}(t,\mu)\Big]\Big)\Bigg)dt-\Game_w \mathfrak{v}^{N}_{\varepsilon^0,\varepsilon^1,n,m}(t,\mu) dW_t^0,
    \end{align}
    where 
    \begin{align*}
        \Game_w \mathfrak{v}^{N}_{\varepsilon^0,\varepsilon^1,n,m}(t,\mu) = \mathbb{E}^{0
        '}\mathbb{E}^{1'}\int_{\mathbb{R}^{nd}}\partial_y \bar{v}_{\varepsilon^0,\varepsilon^1,n,m}^N(t,x_1-\varepsilon^0\bar{B}^0_t-\varepsilon^1\bar{B}^1_t,\ldots,x_n-\varepsilon^0\bar{B}^0_t-\varepsilon^1\bar{B}^n_t,W_t^0)\mu(dx_1)\ldots\mu(dx_n),
    \end{align*}
    and subject to the terminal condition that
     \begin{align*}
        &\mathfrak{v}^N_{\varepsilon^0,\varepsilon^1,n,m}(T,\mu)\\
        = &\mathbb{E}^{0'}\mathbb{E}^{1'}\frac{1}{n}\sum_{i=1}^n\int_{\mathbb{R}^{dn}}g_{n,m}^{N;i}(W_{t_1}^0,\ldots,W_{t_N}^0,x_1-\varepsilon^0\bar{B}^0_T-\varepsilon^1\bar{B}^1_T,\ldots,x_n-\varepsilon^0\bar{B}^0_T-\varepsilon^1\bar{B}^n_T)\mu(dx_1)\otimes\ldots\otimes\mu(dx_n).
    \end{align*}
    \textbf{Step 2.} We first check the terminal. With the help of Lemma \ref{finite_dim_approximation_error} and Lemma \ref{function_n_m_Lipschitz}, we obtain that
    \begin{align*}
        &\Bigg|\mathbb{E}^{0'}\mathbb{E}^{1'}\frac{1}{n}\sum_{i=1}^n\int_{\mathbb{R}^{dn}}g_{n,m}^{N;i}(W_{t_1}^0,\ldots,W_{t_N}^0,x_1-\varepsilon^0\bar{B}_T^0-\varepsilon^1\bar{B}_T^1,\ldots,x_n-\varepsilon^0\bar{B}_T^0-\varepsilon^1\bar{B}_T^n)\mu(dx_1)\otimes\ldots\otimes\mu(dx_n)\\
        & -\int_{\mathbb{R}^d}g(x,\mu)\mu(dx)\Bigg|\\
        \leq&\Bigg|\mathbb{E}^{0'}\mathbb{E}^{1'}\frac{1}{n}\sum_{i=1}^n\int_{\mathbb{R}^{dn}}g_{n,m}^{N;i}(W_{t_1}^0,\ldots,W_{t_N}^0,x_1-\varepsilon^0\bar{B}_T^0-\varepsilon^1\bar{B}_T^1,\ldots,x_n-\varepsilon^0\bar{B}_T^0-\varepsilon^1\bar{B}_T^n)\mu(dx_1)\otimes\ldots\otimes\mu(dx_n)\\
        &-\mathbb{E}^{0'}\mathbb{E}^{1'}\frac{1}{n}\sum_{i=1}^n\int_{\mathbb{R}^{dn}}g_{n,m}^{N;i}(W_{t_1}^0,\ldots,W_{t_N}^0,x_1,\ldots,x_n)\mu(dx_1)\otimes\ldots\otimes\mu(dx_n)\Bigg|\\        &+\Bigg|\mathbb{E}^{0'}\mathbb{E}^{1'}\frac{1}{n}\sum_{i=1}^n\int_{\mathbb{R}^{dn}}g_{n,m}^{N;i}(W_{t_1}^0,\ldots,W_{t_N}^0,x_1,\ldots,x_n)\mu(dx_1)\otimes\ldots\otimes\mu(dx_n)-\int_{\mathbb{R}^d}g^N(x,\mu)\mu(dx)\Bigg|\\
        &+\Bigg|\int_{\mathbb{R}^d}g^N(x,\mu)\mu(dx) - \int_{\mathbb{R}^d}g(x,\mu)\mu(dx)\Bigg|\\
        \leq & K\frac{2\sqrt{2}}{\sqrt{\pi}}\sqrt{T}(\varepsilon^0+\varepsilon^1)+c_dK\Bigg(\int_{\mathbb{R}^d}|x|^q\mu(dx)\Bigg)^{1/q}h_n + K\frac{2}{m}\int_{\mathbb{R}^d}|y_1|\Phi(y_1)dy_1+g^\varepsilon,
    \end{align*}
    where $q\in(1,2]$, $c_d\geq 0$ a constant depending only on $d$, $h_n$ a sequence of real numbers, $h_n\to 0$ as $n\to\infty$ and $\Phi :\mathbb{R}^d\to [0,\infty)$ a symmetric $C^\infty$ functions with compact support satisfying $\int_{\mathbb{R}^d}\Phi(y)dy = 1$. The last inequality also follows from  Lemma \ref{finite_dim_approximation_error}.\\
    \hfill\\
    \textbf{Step 3}. Moreover, we have
    \begin{align*}
        &\Bigg|-\Game_t \mathfrak{v}^N_{\varepsilon^0,\varepsilon^1,n,m}-\int_{\mathbb{R}^{d}}\essinf_{\alpha\in A}\Bigg\{ f_t(x,\mu,\alpha) + \langle b_t(x,\mu,\alpha), \partial_{\mu}\mathfrak{v}_{\varepsilon^0,\varepsilon^1,n,m}^{N}\rangle\Bigg\}\mu(dx)\\
        &\displaystyle- \mu\Big(\frac{1}{2}\tr\Big[\big(\sigma_t\sigma^\intercal_t+\sigma^0_t\sigma^{0;\intercal}_t
        \big)\partial_{x}\partial_{\mu} \mathfrak{v}^N_{\varepsilon^0,\varepsilon^1,n,m}(t,\mu)\Big]\Big)-\mu\Big(\tr\big(\sigma^0_t\partial_\mu\Game_w\mathfrak{v}^{N}_{\varepsilon^0,\varepsilon^1,n,m}(t,\mu)\big)\Big)\\
        &\displaystyle-(\mu\otimes\mu)\Big(\frac{1}{2}\tr\Big[\big(\sigma^0_t\sigma^{0;\intercal}_t\big)\partial^2_{\mu}\mathfrak{v}^N_{\varepsilon^0,\varepsilon^1,n,m}(t,\mu)\Big]\Big)\Bigg|\\
        \leq &\Bigg|-\Game_t \mathfrak{v}^N_{\varepsilon^0,\varepsilon^1,n,m}-\int_{\mathbb{R}^{d}}\essinf_{\alpha\in A}\Bigg\{ f_t(x,\mu,\alpha) + \langle b_t(x,\mu,\alpha), \partial_{\mu}\mathfrak{v}_{\varepsilon^0,\varepsilon^1,n,m}^{N}\rangle\Bigg\}\mu(dx)\\
        &+\Bigg(\mathbb{E}^{0'}\mathbb{E}^{1'}\int_{\mathbb{R}^{nd}}\sum_{i=1}^n\essinf_{a_i\in A}\Bigg\{\frac{1}{n}f_{n,m}^{N;i} (t,W_{t_1}^0,\ldots,W_{t_{N-1}}^0,W_t^0,x_1-\varepsilon^0\bar{B}^0_t-\varepsilon^1\bar{B}^1_t,\ldots,x_n-\varepsilon^0\bar{B}^0_t-\varepsilon^1\bar{B}^n_t,a_i) \\
            &\displaystyle+ \Big\langle b_{n,m}^{N;i}(t,W_{t_1}^0,\ldots,W_{t_{N-1}}^0,W_t^0,x_1-\varepsilon^0\bar{B}^0_t-\varepsilon^1\bar{B}^1_t,\ldots,x_n-\varepsilon^0\bar{B}^0_t-\varepsilon^1\bar{B}^n_t,a_i),\\
        &\,\,\,\,\,\,\,\,\,\partial_{x_i}\bar{v}^N_{\varepsilon^0,\varepsilon^1,n,m}(t,x_1-\varepsilon^0\bar{B}^0_t-\varepsilon^1\bar{B}^1_t,\ldots,x_n-\varepsilon^0\bar{B}^0_t-\varepsilon^1\bar{B}^n_t,y) \Big\rangle\Bigg\}\mu(dx_1)\otimes\ldots\otimes\mu(dx_n)\Bigg)\\
        &-\Bigg(\mathbb{E}^{0'}\mathbb{E}^{1'}\int_{\mathbb{R}^{nd}}\sum_{i=1}^n\essinf_{a_i\in A}\Bigg\{\frac{1}{n}f_{n,m}^{N;i} (t,W_{t_1}^0,\ldots,W_{t_{N-1}}^0,W_t^0,x_1-\varepsilon^0\bar{B}^0_t-\varepsilon^1\bar{B}^1_t,\ldots,x_n-\varepsilon^0\bar{B}^0_t-\varepsilon^1\bar{B}^n_t,a_i) \\
            &\displaystyle+ \Big\langle b_{n,m}^{N;i}(t,W_{t_1}^0,\ldots,W_{t_{N-1}}^0,W_t^0,x_1-\varepsilon^0\bar{B}^0_t-\varepsilon^1\bar{B}^1_t,\ldots,x_n-\varepsilon^0\bar{B}^0_t-\varepsilon^1\bar{B}^n_t,a_i),\\
        &\,\,\,\,\,\,\,\,\,\partial_{x_i}\bar{v}^N_{\varepsilon^0,\varepsilon^1,n,m}(t,x_1-\varepsilon^0\bar{B}^0_t-\varepsilon^1\bar{B}^1_t,\ldots,x_n-\varepsilon^0\bar{B}^0_t-\varepsilon^1\bar{B}^n_t,y) \Big\rangle\Bigg\}\mu(dx_1)\otimes\ldots\otimes\mu(dx_n)\Bigg)\\
         &\displaystyle- \mu\Big(\frac{1}{2}\tr\Big[\big(\sigma_t\sigma^\intercal_t+\sigma^0_t\sigma^{0;\intercal}_t
        \big)\partial_{x}\partial_{\mu} \mathfrak{v}^N_{\varepsilon^0,\varepsilon^1,n,m}(t,\mu)\Big]\Big)-\mu\Big(\tr\big(\sigma^0_t\partial_\mu\Game_w\mathfrak{v}^{N}_{\varepsilon^0,\varepsilon^1,n,m}(t,\mu)\big)\Big)\\
        &\displaystyle-(\mu\otimes\mu)\Big(\frac{1}{2}\tr\Big[\big(\sigma^0_t\sigma^{0;\intercal}_t\big)\partial^2_{\mu}\mathfrak{v}^N_{\varepsilon^0,\varepsilon^1,n,m}(t,\mu)\Big]\Big)\Bigg|\\
        \leq &\Bigg|-\int_{\mathbb{R}^{d}}\essinf_{\alpha\in A}\Bigg\{ f_t(x,\mu,\alpha) + \langle b_t(x,\mu,\alpha), \partial_{\mu}\mathfrak{v}_{\varepsilon^0,\varepsilon^1,n,m}^{N}\rangle\Bigg\}\mu(dx)\\
        &+\Bigg(\mathbb{E}^{0'}\mathbb{E}^{1'}\int_{\mathbb{R}^{nd}}\sum_{i=1}^n\essinf_{a_i\in A}\Bigg\{\frac{1}{n}f_{n,m}^{N;i} (t,W_{t_1}^0,\ldots,W_{t_{N-1}}^0,W_t^0,x_1-\varepsilon^0\bar{B}^0_t-\varepsilon^1\bar{B}^1_t,\ldots,x_n-\varepsilon^0\bar{B}^0_t-\varepsilon^1\bar{B}^n_t,a_i) \\
            &\displaystyle+ \Big\langle b_{n,m}^{N;i}(t,W_{t_1}^0,\ldots,W_{t_{N-1}}^0,W_t^0,x_1-\varepsilon^0\bar{B}^0_t-\varepsilon^1\bar{B}^1_t,\ldots,x_n-\varepsilon^0\bar{B}^0_t-\varepsilon^1\bar{B}^n_t,a_i),\\
        &\,\,\,\,\,\,\,\,\,\partial_{x_i}\bar{v}^N_{\varepsilon^0,\varepsilon^1,n,m}(t,x_1-\varepsilon^0\bar{B}^0_t-\varepsilon^1\bar{B}^1_t,\ldots,x_n-\varepsilon^0\bar{B}^0_t-\varepsilon^1\bar{B}^n_t,y) \Big\rangle\Bigg\}\mu(dx_1)\otimes\ldots\otimes\mu(dx_n)\Bigg)\Bigg|,
    \end{align*}
    where in the last line we made use of $-\Game_t \mathfrak{v}^N_{\varepsilon^0,\varepsilon^1,n,m}$ in (\ref{def_bar_v}). With the help of Lemma \ref{finite_dim_approximation_error}, Lemma \ref{function_n_m_Lipschitz} and Theorem \ref{approximation_properties}, we have
    \begin{align*}
    &\Bigg|-\Game_t \mathfrak{v}^N_{\varepsilon^0,\varepsilon^1,n,m}-\int_{\mathbb{R}^{d}}\essinf_{\alpha\in A}\Bigg\{ f_t(x,\mu,\alpha) + \langle b_t(x,\mu,\alpha), \partial_{\mu}\mathfrak{v}_{\varepsilon^0,\varepsilon^1,n,m}^{N}\rangle\Bigg\}\mu(dx)\\
        &\displaystyle- \mu\Big(\frac{1}{2}\tr\Big[\big(\sigma_t\sigma^\intercal_t+\sigma^0_t\sigma^{0;\intercal}_t
        \big)\partial_{x}\partial_{\mu} \mathfrak{v}^N_{\varepsilon^0,\varepsilon^1,n,m}(t,\mu)\Big]\Big)-\mu\Big(\tr\big(\sigma^0_t\partial_\mu\Game_w\mathfrak{v}^{N}_{\varepsilon^0,\varepsilon^1,n,m}(t,\mu)\big)\Big)\\
        &\displaystyle-(\mu\otimes\mu)\Big(\frac{1}{2}\tr\Big[\big(\sigma^0_t\sigma^{0;\intercal}_t\big)\partial^2_{\mu}\mathfrak{v}^N_{\varepsilon^0,\varepsilon^1,n,m}(t,\mu)\Big]\Big)\Bigg|\\
        \leq &\Bigg|-\int_{\mathbb{R}^{d}}\essinf_{\alpha\in A}\Bigg\{ f_t(x,\mu,\alpha) + \langle b_t(x,\mu,\alpha), \partial_{\mu}\mathfrak{v}_{\varepsilon^0,\varepsilon^1,n,m}^{N}\rangle\Bigg\}\mu(dx) \\
         &+ \int_{\mathbb{R}^{d}}\essinf_{\alpha\in A}\Bigg\{ f_t^N(W_{t_1}^0,\ldots,W_{t_{N-1}}^0,W_t^0,x,\mu,\alpha) + \langle b_t^N(W_{t_1}^0,\ldots,W_{t_{N-1}}^0,W_t^0,x,\mu,\alpha), \partial_{\mu}\mathfrak{v}_{\varepsilon^0,\varepsilon^1,n,m}^{N}\rangle\Bigg\}\mu(dx)\Bigg|\\
&+\Bigg|-\int_{\mathbb{R}^{d}}\essinf_{\alpha\in A}\Bigg\{ f_t^N(W_{t_1}^0,\ldots,W_{t_{N-1}}^0,W_t^0,x,\mu,\alpha) + \langle b_t^N(W_{t_1}^0,\ldots,W_{t_{N-1}}^0,W_t^0,x,\mu,\alpha), \partial_{\mu}\mathfrak{v}_{\varepsilon^0,\varepsilon^1,n,m}^{N}\rangle\Bigg\}\mu(dx)\\
&+\int_{\mathbb{R}^{nd}}\sum_{i=1}^n\essinf_{a_i\in A}\Bigg\{\frac{1}{n}f_{n,m}^{N;i} (t,W_{t_1}^0,\ldots,W_{t_{N-1}}^0,W_t^0,x_1,\ldots,x_n,a_i) \\
            &\displaystyle+ \Big\langle b_{n,m}^{N;i}(t,W_{t_1}^0,\ldots,W_{t_{N-1}}^0,W_t^0,x_1,\ldots,x_n,a_i),\partial_{\mu}\mathfrak{v}_{\varepsilon^0,\varepsilon^1,n,m}^{N}\Big\rangle\Bigg\}\mu(dx_1)\otimes\ldots\otimes\mu(dx_n)\Bigg|\\
&+\Bigg|-\Bigg(\int_{\mathbb{R}^{nd}}\sum_{i=1}^n\essinf_{a_i\in A}\Bigg\{\frac{1}{n}f_{n,m}^{N;i} (t,W_{t_1}^0,\ldots,W_{t_{N-1}}^0,W_t^0,x_1,\ldots,x_n,a_i) \\
            &\displaystyle+ \Big\langle b_{n,m}^{N;i}(t,W_{t_1}^0,\ldots,W_{t_{N-1}}^0,W_t^0,x_1,\ldots,x_n,a_i),\partial_{\mu}\mathfrak{v}_{\varepsilon^0,\varepsilon^1,n,m}^{N}\Big\rangle\Bigg\}\mu(dx_1)\otimes\ldots\otimes\mu(dx_n)\\
&+\Bigg(\mathbb{E}^{0'}\mathbb{E}^{1'}\int_{\mathbb{R}^{nd}}\sum_{i=1}^n\essinf_{a_i\in A}\Bigg\{\frac{1}{n}f_{n,m}^{N;i} (t,W_{t_1}^0,\ldots,W_{t_{N-1}}^0,W_t^0,x_1-\varepsilon^0\bar{B}^0_t-\varepsilon^1\bar{B}^1_t,\ldots,x_n-\varepsilon^0\bar{B}^0_t-\varepsilon^1\bar{B}^n_t,a_i) \\
            &\displaystyle+ \Big\langle b_{n,m}^{N;i}(t,W_{t_1}^0,\ldots,W_{t_{N-1}}^0,W_t^0,x_1-\varepsilon^0\bar{B}^0_t-\varepsilon^1\bar{B}^1_t,\ldots,x_n-\varepsilon^0\bar{B}^0_t-\varepsilon^1\bar{B}^n_t,a_i),\\
        &\,\,\,\,\,\,\,\,\,\partial_{x_i}\bar{v}^N_{\varepsilon^0,\varepsilon^1,n,m}(t,x_1-\varepsilon^0\bar{B}^0_t-\varepsilon^1\bar{B}^1_t,\ldots,x_n-\varepsilon^0\bar{B}^0_t-\varepsilon^1\bar{B}^n_t,W_t^0) \Big\rangle\Bigg\}\mu(dx_1)\otimes\ldots\otimes\mu(dx_n)\Bigg)\Bigg|\\
        \leq & f_t^\varepsilon + C_K b_t^\varepsilon+K(C_K+1)\Bigg(c_d\Big(\int_{\mathbb{R}^d}|x|^q\mu(dx)\Big)^{1/q}h_n + \frac{2}{m}\int_{\mathbb{R}^{d}}|x_1|\Phi(x_1)dx_1\Bigg) \\
        &+ 2K(1+C_K)\frac{2\sqrt{2t}}{\sqrt{\pi}}
        (\varepsilon^0+\varepsilon^1),
    \end{align*}
    where $C_K$ is a constant depending on $K$, independent of $m$, $n$, $\varepsilon^0$, $\varepsilon^1$, and $c_d$, $q$, $h_n$, $\Phi$ are defined as in Step 2.\\
    \hfill\\
    \textbf{Step 4.} Now, let $(Y^\varepsilon_{\varepsilon^0,\varepsilon^1},Z^\varepsilon_{\varepsilon^0,\varepsilon^1})$ be the solution to the following BSDE:
    \begin{align}
    \label{BSDE}    Y^\varepsilon_{\varepsilon^0,\varepsilon^1}(t) =& g^\varepsilon + K\frac{2\sqrt{2T}}{\sqrt{\pi}}(\varepsilon^0+\varepsilon^1) + \int_t^T \Bigg[f_s^\varepsilon + C_K b_s^\varepsilon + 2K(1+C_K)\frac{2\sqrt{2s}}{\sqrt{\pi}}
        (\varepsilon^0+\varepsilon^1)\Bigg]ds\\
    &-\int_t^T Z^\varepsilon_{\varepsilon^0,\varepsilon^1} dW_s^0.\nonumber
    \end{align}
    Define 
    \begin{align*}
        \overline{\mathfrak{v}}^\varepsilon_{\varepsilon^0,\varepsilon^1,n,m}:=\mathfrak{v}^N_{\varepsilon^0,\varepsilon^1,n,m}+Y_{\varepsilon^0,\varepsilon^1}^\varepsilon,\\        \underline{\mathfrak{v}}^\varepsilon_{\varepsilon^0,\varepsilon^1,n,m}:=\mathfrak{v}^N_{\varepsilon^0,\varepsilon^1,n,m}-Y_{\varepsilon^0,\varepsilon^1}^\varepsilon.
    \end{align*}
    In Step 2, it can be readily verified that 
    \begin{align}
    \label{terminal_bar_v_geq_0}
        \nonumber\underline{\mathfrak{v}}^\varepsilon_{\varepsilon^0,\varepsilon^1,n,m}(T) =& \mathfrak{v}^N_{\varepsilon^0,\varepsilon^1,n,m}(T)-g^\varepsilon(T) - K\frac{2\sqrt{2T}}{\sqrt{\pi}}(\varepsilon^0+\varepsilon^1)\\
        \leq& \int_{\mathbb{R}^d}g(x,\mu)\mu(dx)+c_dK\Bigg(\int_{\mathbb{R}^d}|x|^q\mu(dx)\Bigg)^{1/q}h_n + K\frac{2}{m}\int_{\mathbb{R}^d}|y_1|\Phi(y_1)dy_1.
    \end{align}
    Furthermore, based on the results obtained in Step 3, we can conclude that
    \begin{align}
    \label{interior_bar_v_geq_0}
        \nonumber&-\Game_s \underline{\mathfrak{v}}^\varepsilon_{\varepsilon^0,\varepsilon^1,n,m}\\
        \nonumber&- \mathbb{H}\Big(s,\boldsymbol{\mu},\partial_\mu \underline{\mathfrak{v}}^\varepsilon_{\varepsilon^0,\varepsilon^1,n,m}(s,\boldsymbol{\mu})(\cdot),\partial_x\partial_\mu \underline{\mathfrak{v}}^\varepsilon_{\varepsilon^0,\varepsilon^1,n,m}(s,\boldsymbol{\mu})(\cdot),\partial_\mu^2 \underline{\mathfrak{v}}^\varepsilon_{\varepsilon^0,\varepsilon^1,n,m}(s,\boldsymbol{\mu})(\cdot,\cdot),\partial_\mu\Game_w \underline{\mathfrak{v}}^\varepsilon_{\varepsilon^0,\varepsilon^1,n,m}(s,\boldsymbol{\mu})(\cdot)\Big)\\
        \nonumber&-K(C_K+1)\Bigg(c_d\Big(\int_{\mathbb{R}^d}|x|^q\mu(dx)\Big)^{1/q}h_n + \frac{2}{m}\int_{\mathbb{R}^{d}}|x_1|\Phi(x_1)dx_1\Bigg)\\
        \leq &0.
    \end{align}
    Therefore, by combining Theorem \ref{approximation_properties} with the above (\ref{terminal_bar_v_geq_0}), (\ref{interior_bar_v_geq_0}), it is easy to see that that $\lim_n\lim_m\overline{\mathfrak{v}}^\varepsilon_{\varepsilon^0,\varepsilon^1,n,m} \in \overline{\mathscr{V}}$, $\lim_n\lim_m\underline{\mathfrak{v}}^\varepsilon_{\varepsilon^0,\varepsilon^1,n,m} \in \underline{\mathscr{V}}$, and by Theorem \ref{uniqueness_result} we have
    \begin{align}
    \label{relation_uniqueness}
        \lim_n\lim_m\underline{\mathfrak{v}}^\varepsilon_{\varepsilon^0,\varepsilon^1,n,m}\leq v\leq \lim_n\lim_m\overline{\mathfrak{v}}^\varepsilon_{\varepsilon^0,\varepsilon^1,n,m}.   
    \end{align}
    \textbf{Step 5.} We estimate the difference of the upper bound and lower bound in the above equation. Note that 
    \begin{align*}
        \mathbb{E}[|\overline{\mathfrak{v}}^\varepsilon_{\varepsilon^0,\varepsilon^1,n,m}(s,\mu) - \underline{\mathfrak{v}}^\varepsilon_{\varepsilon^0,\varepsilon^1,n,m}(s,\mu)|]
        \leq 2\mathbb{E}[|Y_{\varepsilon^0,\varepsilon^1}^\varepsilon(s,\mu)|]
        \leq 2 \|Y^{\varepsilon}_{\varepsilon^0,\varepsilon^1}\|_{\mathcal{S}^2(\mathbb{R})}
        \leq \tilde{C}_K(\varepsilon+\varepsilon^0+\varepsilon^1),
    \end{align*}
    where the last line is from standard BSDE theory concerning the equation (\ref{BSDE}), and $\tilde{C}_K$ is independent of $n$, $m$ and $N$. Finally, from relation (\ref{relation_uniqueness}) and the arbitrariness of $\varepsilon$, $\varepsilon^0$ and $\varepsilon^1$ we conclude that 
    \begin{align*}
        \underline{u}= v = \overline{u}.
    \end{align*}
\end{proof}
\appendix
\section{Appendix}
\subsection{Measurable Selection Theorem}
\label{measurable_selection_construct}
Recall the measurable selection theorem in \cite{measurable_selection}:
\begin{theorem}
\label{Measurable_Selection_Theorem}
    Let $(\Lambda,\mathscr{M})$ be a measurable space equipped with a nonnegative measure $\mu$, and let $(\mathcal{O},\mathcal{B}(\mathcal{O}))$ be a Polish space. Suppose $F$ is a set-valued function from $\Lambda$ to $\mathcal{B}(\mathcal{O})$ satisfying the following:
    \begin{enumerate}
        \item for $\mu$-a.e. $\lambda\in\Lambda$, $F(\lambda)$ is a closed nonempty subset of $\mathcal{O}$,
        \item for any open set $O\subset\mathcal{O}$, $\{\lambda:F(\lambda)\cap O \neq \phi\}\in \mathscr{M}$.
    \end{enumerate}
    Then there exists a measurable function $f:(\Lambda,\mathscr{M})\to(\mathcal{O},\mathcal{B}(\mathcal{O}))$ such that for $\mu$-a.e. $\lambda\in\Lambda$, $f(\lambda)\in F(\lambda)$.
\end{theorem}
\noindent The above measurable selection theorem is needed in the Step 1 of Theorem \ref{existence_solution} to construct $\overline{\alpha}\in\mathcal{A}_\tau$ such that for almost all $\omega^0\in\Omega_\tau^0$, (\ref{overline_alpha_eq}) holds.  As $\boldsymbol{\rho}\in\mathcal{P}_L^\tau(t_0,\rho_0)$, there exists $\xi$ on $(\Omega,\mathcal{F},\mathbb{P})$ such that $\mathcal{L}(\xi(\omega^0,\cdot)) = \boldsymbol{\rho}(\omega^0)$. We take $\Lambda = \{(\omega,s):\omega = (\omega^0,\omega^1)\in\Omega_\tau^0\times\Omega^1\text{ and }\tau(\omega^0)\leq s\leq T\}$, $\mu = \mathbb{P}\otimes ds$, $(\mathcal{O},\mathcal{B}(\mathcal{O})) = (A,\mathcal{B}(A))$, and $\mathscr{M} = Prog((\mathcal{F}^{0} \vee \hat{\mathcal{F}}^{1;\tau}))$, $Prog((\mathcal{F}^{0} \vee \hat{\mathcal{F}}^{1;\tau}))$ denoting the progressive $\sigma$-algebra w.r.t. $\mathcal{F}^{0} \vee \hat{\mathcal{F}}^{1;\tau}$, where $\hat{\mathcal{F}}^{1;\tau} = \sigma(W_{s\vee \tau} - W_\tau)_{s\geq 0}$. By the continuity of the involved functions,
\begin{align*}
    F(\omega,s):=\Bigg\{a\in A : &\Big\langle b_s(\omega^0,\xi(\omega),\boldsymbol{\rho}(\omega^0),a),\partial_\mu v(\omega^0,s,\boldsymbol{\rho}(\omega^0)) (\xi(\omega)) \Big\rangle + f_s(\omega^0,\xi(\omega),\boldsymbol{\rho}(\omega^0),a) \\
    &+\frac{1}{2}\tr\Big\{ (\sigma_s\sigma_s^{\intercal})(\omega^0,\xi(\omega),\boldsymbol{\rho}(\omega^0),a)\partial_x\partial_\mu v(\omega^0,s,\boldsymbol{\rho}(\omega^0))(\xi(\omega))\Big\}\\
    \leq &\essinf_{a\in A}\Big(\Big\langle b_s(\omega^0,\xi(\omega),\boldsymbol{\rho}(\omega^0),a),\partial_\mu v(\omega^0,s,\boldsymbol{\rho}(\omega^0)) (\xi(\omega)) \Big\rangle + f_s(\omega^0,\xi(\omega),\boldsymbol{\rho}(\omega^0),a) \\
    &+\frac{1}{2}\tr\Big\{ (\sigma_s\sigma_s^{\intercal})(\omega^0,\xi(\omega),\boldsymbol{\rho}(\omega^0),a)\partial_x\partial_\mu v(\omega^0,s,\boldsymbol{\rho}(\omega^0))(\xi(\omega))\Big\}\Big)+\varepsilon \Bigg\}
\end{align*}
satisfies the hypothesis of Theorem \ref{Measurable_Selection_Theorem} and we can obtain an $\overline{\alpha}\in\mathcal{A}_\tau$. Taking expectation with respect to $\mathbb{P}^1$, we see that the constructed $\overline{\alpha}$ satisfies (\ref{overline_alpha_eq}).
\subsection{Conditional Law Invariance}
The below proof is inspired by \cite[Appendix B]{cosso_optimal_2022}. First of all we recall the definition of probability kernel.
\begin{definition}
    (Probability Kernel). Given two measurable spaces $(S, \mathcal{S})$ and $(T, \mathcal{T})$, a mapping $\mu: S \times \mathcal{T} \rightarrow$ $\overline{\mathbb{R}}_{+}$ is called a (probability) kernel from $S$ to $T$ if the function $\mu_s B=\mu(s, B)$ is $\mathcal{S}$-measurable in $s \in S$ for fixed $B \in \mathcal{T}$ and a (probability) measure in $B \in \mathcal{T}$ for fixed $s \in S$.
\end{definition}
Before we proceed, we recall the readers the notion of conditional independence. Consider sub $\sigma$-algebras $\mathcal{F}_1$, $\mathcal{F}_2$, $\mathcal{G} \subset \mathcal{F}$. We say that $\mathcal{F}_1$ and $\mathcal{F}_2$ are conditional independent given $\mathcal{G}$ if 
\begin{align*}
    \mathbb{P}(B_1\cap B_2 | \mathcal{G}) = \mathbb{P}(B_1| \mathcal{G})\mathbb{P}(B_2| \mathcal{G})\text{ a.s., }B_1\in\mathcal{F}_1, B_2\in\mathcal{F}_2.
\end{align*}
We denote the above by $\mathcal{F}_1\indep_\mathcal{G}\mathcal{F}_2$. This  notation is generalised to the case of random variable by considering their induced $\sigma$-algebras.
\begin{theorem}
\label{Thm_6_10_conditional_version}
Fix a probability space $(\Omega,\mathcal{F},\mathbb{P})$, measurable spaces $(K,\mathcal{K})$, $(S,\mathcal{S})$ and a Borel space $(T,\mathcal{B}(T))$. Define the random variables:
\begin{align*}
    \zeta&:(\Omega,\mathcal{F},\mathbb{P})\to(K,\mathcal{K}),\\
    \mathcal{\xi}, \tilde{\xi}&:(\Omega,\mathcal{F},\mathbb{P})\to(S,\mathcal{S}),\\
    \eta&:(\Omega,\mathcal{F},\mathbb{P})\to(T,\mathcal{B}(T)).
\end{align*}
Assume that 
\begin{enumerate}[(i)]
    \item For $\mathbb{P}$-a.e. $\omega$, $\forall B\in\mathcal{S}$, $\mathbb{P}(\xi\in B|\zeta) = \mathbb{P}(\tilde{\xi}\in B|\zeta)$, 
    \item there exists a random variable $\theta$ such that $\theta \indep_\zeta \tilde{\xi}$ and for $\mathbb{P}$-a.e. $\omega$, $\mathcal{L}(\theta|\zeta) = U(0,1)$, where $U(0,1)$ is the law of a Uniform(0,1) distribution.
\end{enumerate}
Then, there exists a measurable mapping $f:K\times S\times [0,1]\to T$ such that if $\tilde{\eta}:=f(\zeta,\tilde{\xi},\theta)$, then for $\mathbb{P}$ a.e. $\omega$, 
\begin{align*}
    \mathbb{E}[g(\tilde{\xi},\tilde{\eta})|\zeta](\omega) = \mathbb{E}[g(\xi,\eta)|\zeta](\omega),\text{ for any measurable $g:S\times T\to \mathbb{R}_+$}. 
\end{align*}
\end{theorem}
\begin{proof}
By \cite{kallenberg_foundations_2002} Theorem 6.3, there exists a probability kernel $\mu$ from $K\times S\to T$ satisfying for $\mathbb{P}$-a.e. $\omega$, $\forall B\in\mathcal{S}$, 
\begin{align*}
    \mu(\zeta,\xi,B) = \mathbb{P}[\eta\in B|\zeta,\xi].
\end{align*}
By \cite{kallenberg_foundations_2002} Lemma 3.22, we may choose a measurable function $f:K\times S\times [0,1]\to T$ such that when given $\zeta$, $f(k,s,\theta)$ has the distribution $\mu(k,s,\cdot)$ for every $(k,s)\in K\times S$. Define $\tilde{\eta}:=f(\zeta,\tilde{\xi},\theta)$. We have $\mathbb{P}$-a.e. $\omega$ for all measurable $g$,
\begin{align*}
    \mathbb{E}[g(\tilde{\xi},\tilde{\eta})|\zeta] = &\mathbb{E}\Big[g(\tilde{\xi},f(\zeta,\tilde{\xi},\theta))\big|\zeta\Big]\\
    =&\mathbb{E}\Big[\mathbb{E}\big[ g(\tilde{\xi},f(\zeta,\tilde{\xi},\theta))|\zeta,\tilde{\xi}\big]|\zeta\Big]\\
    =&\mathbb{E}\Big[\mathbb{E}\big[ g(\xi,f(\zeta,\xi,\theta))|\zeta,\xi\big]|\zeta\Big]\text{ (due to (i) and (ii))}\\
    =&\mathbb{E}\Big[\mathbb{E}\big[ g(\xi,\eta)|\zeta,\xi\big]|\zeta\Big]\\
    = &\mathbb{E}[g(\xi,\eta)|\zeta].
\end{align*}
\end{proof}
\noindent We introduce the setting for the next theorem. Let $t\in[0,T]$. Denote the $\sigma$-algebra
\begin{align*}
    \hat{\mathcal{F}}^{1;t}:=&\sigma(W_{s\vee t} - W_t:s\geq 0).
\end{align*}
Let
\begin{align*}
    \overline{\Omega} = [0,T]\times\Omega,\,\,\,\,\,\overline{\mathcal{F}} = \mathcal{B}([0,T])\otimes \mathcal{F}^{0}\otimes\hat{\mathcal{F}}^{1;t}\otimes\mathcal{G},\,\,\,\,\,\overline{\mathbb{P}} = \lambda_{[0,T]} \otimes \mathbb{P},
\end{align*}
with $\lambda_{[0,T]}$ being the uniform distribution on $([0,T],\mathcal{B}([0,T]))$, and $Prog((\mathcal{F}^{0} \vee \hat{\mathcal{F}}^{1;t}))$ denoting the progressive $\sigma$-algebra w.r.t. $\mathcal{F}^{0} \vee \hat{\mathcal{F}}^{1;t}$. Let $\zeta:(\overline{\Omega},\overline{\mathcal{F}})\to(\Omega^0,\mathcal{F}^{0}_t)$ be the projection map, i.e.,
\begin{align}
\label{zeta}
    \zeta(s,\omega^0,\omega^1):= \omega^0_{\cdot\wedge t}.
\end{align}
\begin{theorem}
\label{new_control}
Let ($\mathcal{A}1$) hold. Let $\alpha\in\mathcal{A}_t$, $\xi \in L^2((\Omega,\mathcal{F}_t,\mathbb{P});\mathbb{R}^d)$. Suppose that there exists an $\mathcal{F}_t$-measurable random variable $U_\xi$, for $\mathbb{P}$-a.e. $\omega$, $\mathcal{L}(U_\xi|\zeta) = U(0,1)$ and $U_\xi \indep_\zeta \xi$. Then, there exists a function
$b:[0,T]\times\Omega^0\times\hat{\Omega}^1\times\mathbb{R}^d\times [0,1]\to A$, measurable with respect to $Prog((\mathcal{F}^{0} \vee \hat{\mathcal{F}}^{1;t}))\otimes\mathcal{B}(\mathbb{R}^d)\otimes\mathcal{B}([0,1])$ such that $(b_s(\xi,U_\xi))_{s\in[0,T]}$ is $\mathbb{F}^t$-progressively measurable, and
\begin{align*}
    &\mathcal{L}\Big(\xi,(b_s(\xi,U_\xi))_{s\in[0,T]},(W_s^0)_{s\in [0,T]},(W_{s\vee t}-W_t)_{s\in [0,T]} \Big)\Bigg|_\zeta \\
    = &\mathcal{L}\Big(\xi,(\alpha_s)_{s\in[0,T]},(W_s^0)_{s\in [0,T]},(W_{s\vee t}-W_t)_{s\in [0,T]} \Big)\Bigg|_\zeta\,\,\,\mathbb{P}^0\text{-}a.e..
\end{align*}
\end{theorem}
\begin{proof}
Consider the canonical extension of $U_\xi$ to $\overline{\Omega}$, denoted by $\overline{U}_\xi$. Let also $(\bar{E},\bar{\mathscr{E}})$ be the measurable space defined as $\bar{E} = [0,T]\times\Omega^0\times\hat{\Omega}^1$ and $\bar{\mathscr{E}} = Prog((\mathcal{F}^{0} \vee \hat{\mathcal{F}}^{1;t}))$. Let $\pi^{W^0,W,s}:(\overline{\Omega},\overline{\mathcal{F}})\to(\bar{E},\bar{\mathscr{E}})$ be the projection map, i.e.,
\begin{align*}
    \pi^{W^0,W,s}(s,\omega^0,\omega^1) := (s,\omega^0,\hat{\omega}^1).
\end{align*}
Define $\Gamma := (\pi^{W^0,W,s}, \xi)$. From assumption, when given $\zeta$, $\overline{U}_\xi$ is independent of $\xi$. Note also that given $\zeta$, $\overline{U}_\xi$ is independent of $\hat{\mathcal{F}}^{1;t}$ (as it is $\mathcal{F}_t$-measurable). Therefore given $\zeta$, $\overline{U}_\xi$ is also independent of $\pi^{W^0,W,s}$, and thus independent of $\Gamma$, given $\zeta$. $\Gamma$ takes values in the measurable space $(E,\mathscr{E})$, with $E = [0,T]\times\Omega^0\times\hat{\Omega}^1\times\mathbb{R}^d$, $\mathscr{E} = Prog((\mathcal{F}^{0} \vee \hat{\mathcal{F}}^{1;t}))\otimes\mathcal{B}(\mathbb{R}^d)$. Working under the conditional expectation on $\mathcal{F}_t^0$ is equivalent to conditioning on the random variable $\zeta$. From above Theorem \ref{Thm_6_10_conditional_version}, there exists a function $a: \Omega^0\times[0,T]\times\Omega^0\times\hat{\Omega}^1\times\mathbb{R}^d\times[0,1]\to A$, measurable with respect to the $\sigma$-algebras $\mathcal{F}_t^0\otimes Prog((\mathcal{F}^{0} \vee \hat{\mathcal{F}}^{1;t}))\otimes\mathcal{B}(\mathbb{R}^d)\otimes\mathcal{B}([0,1])$, such that we have 
\begin{align*}
    \mathcal{L}\big(\Gamma, (a_s(\zeta,\Gamma,\overline{U}_\xi))_{s\in[0,T]}\big)\Bigg|_{\zeta} = \mathcal{L}\big(\Gamma, (\alpha_s)_{s\in[0,T]}\big)\Bigg|_{\zeta}\,\,\,\,\,\mathbb{\overline{P}}\text{-}a.e.,
\end{align*}
from which we deduce that
\begin{align*}
    &\mathcal{L}\Big(\xi,(a_s(\zeta,\Gamma,\overline{U}_\xi))_{s\in [0,T]},(W_s^0)_{s\in [0,T]},(W_{s\vee t}-W_t)_{s\in [0,T]} \Big)\Bigg|_\zeta \\
    = &\mathcal{L}\Big(\xi,(\alpha_s)_{s\in[0,T]},(W_s^0)_{s\in [0,T]},(W_{s\vee t}-W_t)_{s\in [0,T]} \Big)\Bigg|_\zeta\,\,\,\mathbb{P}^0\text{-}a.e..
\end{align*}
Now we define $b:[0,T]\times\Omega^0\times\Omega^1\times\mathbb{R}^d\times [0,1]\to A$ 
\begin{align*}
    b(s,\omega^0,\omega^1,x,u):= \alpha_s(\omega^0,\omega^1)1_{[0,t)}(s)+a_s(\omega^0_{\cdot\wedge t},s,\omega^0,\hat{\omega}^1,x,u)1_{[t,T]},
\end{align*}
then $(b_s(\xi,U_\xi))_{s\in[0,T]}$ is $\mathbb{F}^t$-progressively measurable, and
\begin{align*}
    &\mathcal{L}\Big(\xi,(b_s(\xi,U_\xi))_{s\in[0,T]},(W_s^0)_{s\in [0,T]},(W_{s\vee t}-W_t)_{s\in [0,T]} \Big)\Bigg|_\zeta \\
    = &\mathcal{L}\Big(\xi,(\alpha_s)_{s\in[0,T]},(W_s^0)_{s\in [0,T]},(W_{s\vee t}-W_t)_{s\in [0,T]} \Big)\Bigg|_\zeta\,\,\,\mathbb{P}^0\text{-}a.e..
\end{align*}
\end{proof}
\begin{lemma}
\label{independent_rv_uniform_distribution}
Assume ($\mathcal{A}1$) holds and let $\xi$ be $\mathcal{F}_t$-measurable. Suppose that
\begin{align*}
    \xi = \sum_{i=1}^n a_i 1_{B_i^0}(\omega^0) 1_{B_i^1}(\omega^1),
\end{align*}
where $\{B_i^0\times B_i^1\}_{i=1,\ldots,n}$ is a (disjoint) partition of $\Omega^0\times\Omega^1$. Then there exists an $\mathcal{F}_t$-measurable random variable $U_\xi$, for $\mathbb{P}$-a.e. $\omega$, $\mathcal{L}(U_\xi|\zeta) = U(0,1)$ and $U_\xi \indep_\zeta \xi$, where $\zeta$ is defined in (\ref{zeta}).
\end{lemma}
\begin{proof}
Note that $\xi$ can also be written as
\begin{align*}
    \xi = \sum_{i=1}^m \Xi_i(\omega^1) 1_{C_i^0}(\omega^0),
\end{align*}
where $\Xi_i$ is discrete value random variable on $(\Omega^1,\mathcal{F}^1,\mathbb{P}^1)$, and $\{C_i^0\}_{i=1,\ldots,m}$ is a disjoint partition of $\Omega^0$. For each $\Xi_i$, apply Lemma B.3 in \cite{cosso_optimal_2022}, we get $U_{\Xi_i}$ of uniform distribution on $[0,1]$ and independent of $\Xi_i$. Then
\begin{align*}
    U_\xi := \sum_{i=1}^m U_{\Xi_i}(\omega^1)1_{C_i^0}(\omega^0), 
\end{align*}
is $\mathcal{F}_t$ measurable, having uniform distribution on $[0,1]$, and independent of $\xi$, when given $\zeta$. 
\end{proof}
\begin{lemma}
\label{X_equal_in_law}
Assume ($\mathcal{A}1$) holds. If
\begin{small}
    \begin{align*}
    \mathcal{L}\Big(\xi,(\alpha_s)_{s\in[0,T]},(W_s^0)_{s\in [0,T]},(W_{s\vee t}-W_t)_{s\in [0,T]} \Big)\Bigg|_\zeta = \mathcal{L}\Big(\eta,(\beta_s)_{s\in[0,T]},(W_s^0)_{s\in [0,T]},(W_{s\vee t}-W_t)_{s\in [0,T]} \Big)\Bigg|_\zeta\,\,\mathbb{P}^0\text{-}a.e.,
\end{align*}
\end{small}
then 
\begin{align*}
    \mathcal{L}\Big((X^{t,\xi,\alpha}_s)_{s\in[t,T]},(\alpha_s)_{s\in[0,T]},(W_s^0)_{s\in [0,T]}\Big)\Bigg|_\zeta = \mathcal{L}\Big((X^{t,\eta,\beta}_s)_{s\in[t,T]},(\beta_s)_{s\in[0,T]},(W_s^0)_{s\in [0,T]}\Big)\Bigg|_\zeta\,\,\mathbb{P}^0\text{-}a.e.
\end{align*}
\end{lemma}
\begin{proof}
    We prove this by inspecting the finite dimensional distribution, and since the calculation is similar, we  only check that
    \begin{align*}
        \mathcal{L}\Big(X^{t,\xi,\alpha}_s,\alpha_s,W_s^0\Big)\Bigg|_\zeta = \mathcal{L}\Big(X^{t,\eta,\beta}_s,\beta_s,W_s^0\Big)\Bigg|_\zeta\text{ for some $s\in[t,T]$, }\mathbb{P}^0\text{-}a.e..
    \end{align*}
    We extract a subsequence such that the Picard iteration $X^{(n),t,\xi,\alpha}_s \to X^{t,\xi,\alpha}_s$ a.e. and $X^{(n),t,\eta,\beta}_s \to X^{t,\eta,\beta}_s$ $\mathbb{P}$ a.e.. For $\mathbb{P}^0$ a.e. $\omega^0$, for any bounded continuous function $g$ (it is without loss of generality to look at bounded continuous $g$ only by truncation and mollification argument), we have 
    \begin{align*}
        &\mathbb{E}\Big[g(X_s^{t,\xi,\alpha},\alpha_s,W_s^0)\big|\zeta\Big]\\
        =&\lim_n \mathbb{E}\Big[g(X_s^{(n),t,\xi,\alpha},\alpha_s,W_s^0)\big|\zeta\Big]\\
        =& \lim_n \mathbb{E}\Big[g\Big(\xi + \int_t^s b_r(X_r^{(n-1),t,\xi,\alpha},\mathbb{P}_{X_r^{(n-1),t,\xi,\alpha}}^{W^0},\alpha_r)dr+\int_t^s\sigma_r(X_r^{(n-1),t,\xi,\alpha},\mathbb{P}_{X_r^{(n-1),t,\xi,\alpha}}^{W^0},\alpha_r)dW_r \\
    &+ \int_t^s\sigma^0_r(X_r^{(n-1),t,\xi,\alpha},\mathbb{P}_{X_r^{(n-1),t,\xi,\alpha}}^{W^0},\alpha_r)dW^0_r,\alpha_s,W_s^0\Big)\big|\zeta\Big].
    \end{align*}
    We proceed by induction, for $X^{(1),t,\xi,\alpha}_s$, we have for $\mathbb{P}^0$-a.e. $\omega^0$
    \begin{align*}
        &\mathbb{E}\Big[g(X^{(1),t,\xi,\alpha}_s,\alpha_s,W_s^0)|\zeta\Big]\\=&\mathbb{E}\Big[g\Big(\xi + \int_t^s b_r(0,\delta_0,\alpha_r)dr+\int_t^s\sigma_r(0,\delta_0,\alpha_r)dW_r + \int_t^s\sigma^0_r(0,\delta_0,\alpha_r)dW^0_r,\alpha_r,W_s^0\Big)\big|\zeta\Big]\\
        =&\mathbb{E}\Big[g\Big(\eta + \int_t^s b_r(0,\delta_0,\beta_r)dr+\int_t^s\sigma_r(0,\delta_0,\beta_r)dW_r + \int_t^s\sigma^0_r(0,\delta_0,\beta_r)dW^0_r,\beta_r,W_s^0\Big)\big|\zeta\Big]\text{ (by our assumption)}\\
        =&\mathbb{E}\Big[g(X^{(1),t,\eta,\beta}_s,\beta_s,W_s^0)|\zeta\Big].
    \end{align*}
    It follows immediately from induction that $\mathbb{E}\Big[g(X^{(n),t,\xi,\alpha}_s,\alpha_s,W_s^0)|\zeta\Big] = \mathbb{E}\Big[g(X^{(n),t,\eta,\beta}_s,\beta_s,W_s^0)|\zeta\Big]$, $\mathbb{P}^0$-a.e. $\omega^0$. Thus for $\mathbb{P}^0$-a.e. $\omega^0$
    \begin{align*}
        \mathbb{E}\Big[g(X_s^{t,\xi,\alpha},\alpha_s,W_s^0)\big|\zeta\Big]
        = &\, \lim_n \mathbb{E}\Big[g(X_s^{(n),t,\xi,\alpha},\alpha_s,W_s^0)\big|\zeta\Big]
        =\lim_n \mathbb{E}\Big[g(X_s^{(n),t,\eta,\beta},\beta_s,W_s^0)\big|\zeta\Big] \\
        = &\,\,\mathbb{E}\Big[g(X_s^{t,\eta,\beta},\beta_s,W_s^0)\big|\zeta\Big].
    \end{align*}
\end{proof}
\subsection{A compact subset of the Wasserstein space-valued random variables}
\label{compact_subset}
First, we recall the compactness result in \cite{wang2019compactness}. Let $(\Omega,\mathcal{F},\mathbb{P})$ be a probability space and let $X$ be a separable Banach space with norm $\|\cdot\|_X$. 
\begin{definition}
    (Uniformly $L^p$-integrable). A set $V\subset L^p((\Omega,\mathcal{F},\mathbb{P});X)$ is called uniformly $L^p$-integrable if
    \begin{enumerate}
        \item There exists a finite $M>0$ such that for every $f\in V$,
        \begin{align*}
            \int_\Omega \|f\|_X^p d\mathbb{P}\leq M
        \end{align*}
        \item For every $\varepsilon >0$, there exists a delta $\delta>0$ such that for every $A$ in $\mathcal{F}$ with $\mathbb{P}(A)\leq \delta$ and for all $f\in V$,
        \begin{align*}
            \int_A \|f\|_X^pd\mathbb{P}\leq \varepsilon.
        \end{align*}
    \end{enumerate}
\end{definition}
\begin{definition}
    (Uniformly tight). A set $V\subset L^p((\Omega,\mathcal{F},\mathbb{P});X)$ is called uniformly tight if for every $\varepsilon>0$, there exists a compact $K\subset X$ such that for all $f\in V$,
    \begin{align*}
        \mu_f(K)\geq 1-\varepsilon,
    \end{align*}
    where $\mu_f$ is the distribution of $f$, i.e., the Radon probability measure $\mu_f$ on $X$ defined by
    \begin{align*}
        \mu_f(B) = \mathbb{P}(f\in B)\text{ for }B\subseteq X\text{ Borel.}
    \end{align*}
\end{definition}
Below, we provide the following result.
\begin{theorem}
\label{compactness_rv_wang}
    Let $1\leq p < \infty$. A subset $V\subseteq L^p((\Omega,\mathcal{F},\mathbb{P});X)$ is relatively compact in the $L^p$ norm if and only if it is uniformly $L^p$-integrable and uniformly tight.
\end{theorem}
\begin{proof}
    See \cite{wang2019compactness}.
\end{proof}
Let us now return to the settings outlined in Section \ref{setup}.
\begin{definition}
    Inspired by the work in \cite{mean_field_path}, we introduce the following definition. Let $p\geq 1$, $L>0$, $(t_0,\rho_0)\in[0,T]\times\mathcal{P}_p(\mathbb{R}^d)$, let $\xi\in L^p((\Omega,\mathcal{F},\mathbb{P});\mathbb{R}^d)$ such that $\mathcal{L}(\xi) = \rho_0$, and let $\tau\in\mathcal{T}_{t,T}^0$. For any $\mathbb{F}^{t_0}$-progressively measurable $b:[0,T]\times\Omega\to\mathbb{R}^d$, $\sigma:[0,T]\times\Omega\to\mathbb{R}^{d\times n}$, $|b|\leq L$, $|\sigma|\leq L$, and a fixed $\mathbb{F}^{t_0}$-progressively measurable $\sigma^0:[0,T]\times\Omega\to\mathbb{R}^{d\times m}$, $|\sigma^0|\leq L$, consider the following dynamics:
    \begin{align}
    \label{dynamics_L}
            X_\tau^{t_0,\xi} &= \xi +\int_{t_0}^\tau b_r dr + \int_{t_0}^\tau\sigma_r^0 dW^0_r + \int_{t_0}^\tau\sigma_r dW_r.
    \end{align}
    Define
    \begin{align*}
        \mathbb{X}_{L}^{\tau}(t_0,\rho_0):=&\Bigg\{X_\tau\in L^p((\Omega,\mathcal{F}_\tau,\mathbb{P});\mathbb{R}^d)\Big|\,\rho_0\in \mathcal{P}_p(\mathbb{R}^d),\,X_\tau\text{ follows (\ref{dynamics_L}) with $\mathcal{L}(\xi)=\rho_0$ for some}\\
        &\text{$\mathbb{F}^{t_0}$-progressively measurable $b$, $\sigma$ and the fixed $\sigma^0$ such that $\max\{|b|,|\sigma|,|\sigma^0|\}\leq L$}\Bigg\},
    \end{align*}
    \begin{align*}
        \mathscr{X}_L^{\tau}(t_0,\rho_0):= \Bigg\{X:=\sum_{i=1}^n X^i 1_{A_i}\,\Big|\,X^i\in \mathbb{X}_L^{\tau}(t_0,\rho_0)\text{ and $A_i's$ are $\mathcal{F}_{\tau}^0$ measurable, a partition of $\Omega^0$}\Bigg\},
    \end{align*}
    and
    \begin{align*}
        \mathcal{P}_L^\tau(t_0,\rho_0):=\Bigg\{\boldsymbol{\rho}:\Omega^0\to\mathcal{W}_p(\mathbb{R}^d)\,\Big|\,\boldsymbol{\rho} \text{ $\mathcal{F}_\tau^0$-measurable and }\boldsymbol{\rho}(\omega^0) = \mathcal{L}(X(\omega^0,\cdot))\text{ for some }X\in\overline{\mathscr{X}_L^\tau(t_0,\rho_0)}\Bigg\},
    \end{align*}
    where the closure $\overline{\mathscr{X}_L^\tau(t_0,\rho_0)}$ is with respect to the $L^p$ norm.
\end{definition}
\begin{theorem}
\label{compactness_rv}
    Let ($\mathcal{A}1$) hold. Let $p\geq 1$, $\rho_0\in\mathcal{P}_p(\mathbb{R}^d)$. Let $\xi\in L^p((\Omega,\mathcal{F},\mathbb{P});\mathbb{R}^d)$, and $\mathcal{L}(\xi)=\rho_0$. Then $\mathscr{X}_{L}^\tau(t_0,\rho_0)\subset L^p((\Omega,\mathcal{F},\mathbb{P});\mathbb{R}^d)$ and $\mathscr{X}_{L}^\tau(t_0,\rho_0)$ is precompact in $L^p((\Omega,\mathcal{F},\mathbb{P});\mathbb{R}^d)$.
\end{theorem}
\begin{proof}
    Let $X \in \mathscr{X}_L^{\tau}(t_0,\rho_0)$, then $X =\sum_{i=1}^n X^i 1_{A_i}$, $X^i\in \mathbb{X}_L^{\tau}(t_0,\rho_0)$ and $A_i's$ are $\mathcal{F}_{\tau}^0$ measurable, a partition of $\Omega^0$. As $X^i\in \mathbb{X}_L^{\tau}(t_0,\rho_0)$, there exists $b^i$, $\sigma^{i}$, $i=1,\ldots,n$ such that $X^i:=X^{t_0,\xi;i}$ solves
\begin{align*}
    \begin{cases}
            dX_r^{t_0,\xi;i} &= b_r^{i} dr + \sigma_r^0 dW^0_r + \sigma_r^i dW_r,\\
            X_{t_0}^{t_0,\xi;i} &= \xi.
        \end{cases}
\end{align*}
Considering $C(T, p) > 0$ as a constant that depends solely on $T$ and $p$, and possibly varies in different instances, we readily observe the following upper bound:
\begin{align}
\label{C_star}
    |X_\tau^{t_0,\xi;i}|^p \leq& C(T,p)\Big[|\xi|^p +\int_{t_0}^\tau |b_r^{i}|^p dr + \sup_{{t_0}\leq s\leq T}\Big|\int_{t_0}^s\sigma_r^{0} dW^0_r\Big|^p + \sup_{{t_0}\leq s\leq T}\Big|\int_{t_0}^s\sigma_{r}^i dW_r\Big|^p\Big]\nonumber\\
    \leq&  C(T,p)\Big[|\xi|^p +L^pT + \sup_{{t_0}\leq s\leq T}\Big|\int_{t_0}^s\sigma_r^{0} dW^0_r\Big|^p + \sup_{{t_0}\leq s\leq T}\Big|\int_{t_0}^s\sigma_{r}^i dW_r\Big|^p\Big]\nonumber\\
    :=& C^*\Big[|\xi|^p +L^pT + \sup_{{t_0}\leq s\leq T}\Big|\int_{t_0}^s\sigma_r^{0} dW^0_r\Big|^p + \sup_{{t_0}\leq s\leq T}\Big|\int_{t_0}^s\sigma_{r}^i dW_r\Big|^p\Big].
\end{align}
Therefore
\begin{align*}
    &\mathbb{E}^1\Big|\sum_{i=1}^n X^i(\omega^0,\omega^1)1_{A_i}(\omega^0)\Big|^p\\
    \leq&C(T,p)\sum_{i=1}^n\mathbb{E}^1\Bigg[\Big[|\xi|^p +L^pT + \sup_{t_0\leq s\leq T}\Big|\int_{t_0}^s\sigma_r^{0} dW^0_r\Big|^p + \sup_{t_0\leq s\leq T}\Big|\int_{t_0}^s\sigma_{r}^i dW_r\Big|^p\Big] 1_{A_i}\Bigg]\\
    \leq& C(T,p)\Bigg(\mathbb{E}^1|\xi|^p +L^pT + \mathbb{E}^1\sup_{t_0\leq s\leq T}\Big|\int_{t_0}^s\sigma_r^{0} dW^0_r\Big|^p + \sum_{i=1}^n1_{A_i}\mathbb{E}^1\sup_{t_0\leq s\leq T}\Big|\int_{t_0}^s\sigma_{r}^i dW_r\Big|^p\Bigg)\\
    \leq &C(T,p)\Bigg(\mathbb{E}^1|\xi|^p +L^pT + \mathbb{E}^1\sup_{t_0\leq s\leq T}\Big|\int_{t_0}^s\sigma_r^{0} dW^0_r\Big|^p + \sum_{i=1}^n1_{A_i}L^pT\Bigg)\\
    =&C(T,p)\Bigg(\mathbb{E}^1|\xi|^p +2L^pT + \mathbb{E}^1\sup_{t_0\leq s\leq T}\Big|\int_{t_0}^s\sigma_r^{0} dW^0_r\Big|^p\Bigg).
\end{align*}
Taking expectation with respect to $\mathbb{E}^0$ we have
\begin{align*}
    \mathbb{E}\Big|\sum_{i=1}^n X^i(\omega^0,\omega^1)1_{A_i}(\omega^0)\Big|^p \leq C(T,p)(\mathbb{E}|\xi|^p +L^p).
\end{align*}
Moreover, defining $Z:= |\xi|^p +L^pT + \sup_{t_0\leq s\leq T}|\int_{t_0}^s\sigma_r^{0} dW^0_r|^p + \sup_{t_0\leq s\leq T}|\int_{t_0}^s\sigma_{r}^i dW_r|^p$, we have
\begin{align*}
    &|X_\tau^{t_0,\xi;i}|^p 1_{|X_\tau^{t_0,\xi;i}|>K}\\
    \leq &C^*Z1_{C^*Z>K},\text{ where }C^*\text{ is defined in (\ref{C_star})}\\
     \leq &3C^*\Big[|\xi|^p1_{|\xi|>K/(3C^*)} +\Big(L^pT + \sup_{t_0\leq s\leq T}\Big|\int_{t_0}^s\sigma_r^{0} dW^0_r\Big|^p\Big)1_{L^pT + \sup_{t_0\leq s\leq T}|\int_{t_0}^s\sigma_r^{0} dW^0_r|^p>K/(3C^*)} \\
     &+ \sup_{t_0\leq s\leq T}\Big|\int_{t_0}^s\sigma_{r}^i dW_r\Big|^p1_{\sup_{t_0\leq s\leq T}|\int_{t_0}^s\sigma_{r}^i dW_r|^p>K/(3C^*)}\Big].
\end{align*}
Taking expectation of $|\sum_{i=1}^nX_\tau^{t_0,\xi;i}1_{A_i} |^p1_{|\sum_{i=1}^nX_\tau^{t_0,\xi;i}1_{A_i}|>K}$ with respect to $\mathbb{P}^1$ we have
\begin{align*}
    &\mathbb{E}^1\Big|\sum_{i=1}^nX_\tau^{t_0,\xi;i}1_{A_i} \Big|^p1_{|\sum_{i=1}^nX_\tau^{t_0,\xi;i}1_{A_i}|>K}\\
    \leq&\mathbb{E}^1\sum_{i=1}^n|X_\tau^{t_0,\xi;i}|^p1_{A_i} 1_{|X_\tau^{t_0,\xi;i}|>K}\\
    \leq &3C^*\mathbb{E}^1 \Bigg[|\xi|^p1_{|\xi|>K/(3C^*)} +\Big(L^pT + \sup_{t_0\leq s\leq T}\Big|\int_{t_0}^s\sigma_r^{0} dW^0_r\Big|^p\Big)1_{L^pT + \sup_{t_0\leq s\leq T}|\int_{t_0}^s\sigma_r^{0} dW^0_r|^p>K/(3C^*)}\Bigg] \\
     &+ 3C^*\sum_{i=1}^n 1_{A_i}\mathbb{E}^1\sup_{t_0\leq s\leq T}\Big|\int_{t_0}^s\sigma_{r}^i dW_r\Big|^p1_{\sup_{t_0\leq s\leq T}|\int_{t_0}^s\sigma_{r}^i dW_r|^p>K/(3C^*)}\\
     \leq &C(T,p)\mathbb{E}^1 \Bigg[|\xi|^p1_{|\xi|>K/(3C^*)} +\Big(L^pT + \sup_{t_0\leq s\leq T}\Big|\int_{t_0}^s\sigma_r^{0} dW^0_r\Big|^p\Big)1_{L^pT + \sup_{t_0\leq s\leq T}|\int_{t_0}^s\sigma_r^{0} dW^0_r|^p>K/(3C^*)}\Bigg] \\
     &+ \frac{C(T,p)L^{2p}}{K},
\end{align*}
Take expectation with respect to $\mathbb{P}^0$ we have
\begin{align*}
&\mathbb{E}\Big|\sum_{i=1}^nX_\tau^{t_0,\xi;i}1_{A_i} \Big|^p1_{|\sum_{i=1}^nX_\tau^{t_0,\xi;i}1_{A_i}|>K}\\
    \leq &C(T,p)\mathbb{E} \Bigg[|\xi|^p1_{|\xi|>K/(3C^*)} +\Big(L^pT + \sup_{t_0\leq s\leq T}\Big|\int_{t_0}^s\sigma_r^{0} dW^0_r\Big|^p\Big)1_{L^pT + \sup_{t_0\leq s\leq T}|\int_{t_0}^s\sigma_r^{0} dW^0_r|^p>K/(3C^*)}\Bigg] \\
     &+ \frac{C(T,p)L^{2p}}{K},
\end{align*}
for some constant $C(T,p)$ depending on $T$, $p$ only. The uniform tightness is immediate from the above as we are working under $\mathbb{R}^n$. Therefore with Theorem \ref{compactness_rv_wang} we conclude that $\mathscr{X}_L^{\tau}(t_0,\rho_0)$ is precompact in $L^p((\Omega,\mathcal{F},\mathbb{P});\mathbb{R}^d)$.
\end{proof}
\begin{theorem}
\label{P_L_compact}
    Let ($\mathcal{A}1$) hold. The set $\mathcal{P}_L^\tau(t_0,\rho_0)$ is compact in the topology of convergence in probability. Moreover, this set is closed under finite partition addition, i.e., let $A_i$, $i = 1,\ldots,n$ be a disjoint partition of $\Omega^0$, $\rho_i \in \mathcal{P}_L^\tau(t_0,\rho_0)$, then $\sum_{i=1}^n \rho_i 1_{A_i} \in \mathcal{P}_L^\tau(t_0,\rho_0)$. 
\end{theorem}
\begin{proof}
Let $\{\rho_i\}_{i\in\mathbb{N}}$ be a sequence in $\mathcal{P}_L^\tau(t_0,\rho_0)$. By definition, there exists $X^i$ such that $\rho_i(\omega^0) =\mathcal{L}(X^i(\omega^0,\cdot))$. By Theorem \ref{compactness_rv}, 
up to a subsequence there exists $X^\infty\in L^p((\Omega,\mathcal{F}_\tau,\mathbb{P});\mathbb{R}^d)$ such that $X^{i}_{}\to X^\infty$ in $L^p$ as $i\to\infty$. Denote $\rho_\infty(\omega^0) := \mathcal{L}(X^{\infty}(\omega^0,\cdot))$, we have
\begin{align*}
    &\mathbb{E}^0\Big[\mathcal{W}_p(\rho_i(\omega^0),\rho_\infty(\omega^0))^pd\mathbb{P}^0(\omega^0)\Big]\\
    \leq &\mathbb{E}^0\Big[\mathbb{E}^1\Big[|X^{i}(\omega^0,\omega^1)-X^\infty(\omega^0,\omega^1)|^pd\mathbb{P}^1(\omega^1)\Big]d\mathbb{P}^0(\omega^0)\Big]\\
    \to &0, \text{ as it converges in }L^p. 
\end{align*}
Therefore, $\rho_i \to \rho_\infty$ in probability.\\
\hfill\\
Now we prove that this set is closed under finite partition addition. Let $\{B_i\}_{i=1,\ldots,m}$ be a partition of $\Omega^0$, Without loss of generality, we can assume $m=2$. Let $\rho^1,\rho^2 \in \mathcal{P}_{L}^\tau(t_0,\rho_0)$, $\rho^1 = \lim_j\sum_{i=1}^{n^1_j}\rho^1_{i;j}1_{A_{i;j}^1}$, $\rho^2 = \lim_j\sum_{i=1}^{n^2_j}\rho^2_{i;j}1_{A_{i;j}^2}$, where for each $i$, $j$, there exists $X_{i;j}^1$, $X_{i;j}^2$ such that $\mathcal{L}(X_{i;j}^1(\omega^0,\cdot)) = \rho_{i;j}^1(\omega^0)$, $\mathcal{L}(X_{i;j}^2(\omega^0,\cdot)) = \rho_{i;j}^2(\omega^0)$, and $\{A_{i;j}^1\}_{i = 1,\ldots,n_j^1}$, $\{A_{i;j}^2\}_{i = 1,\ldots,n_j^2}$ partitions of $\Omega^0$. We have
\begin{align*}
    \mathbb{E}^0\Big[\mathcal{W}_p\Big(\sum_{i=1}^{n^1_j}\rho^1_{i;j}1_{A_{i;j}^1}1_{B_1},\rho^1 1_{B_1}\Big)^p d\mathbb{P}^0\Big]\to 0,
\end{align*}
therefore $\rho^1 1_{B_1} = \lim_j\sum_{i=1}^{n^1_j}\rho^1_{i;j}1_{A_{i;j}^1}1_{B_1}$  and similarly $\rho^2 1_{B_2} = \lim_j\sum_{i=1}^{n^2_j}\rho^2_{i;j}1_{A_{i;j}^2}1_{B_2}$, so
\begin{align*}
    \rho^1 1_{B_1}+\rho^2 1_{B_2} = \lim_j\Bigg(\sum_{i=1}^{n^1_j}\rho^1_{i;j}1_{A_{i;j}^1}1_{B_1} + \sum_{i=1}^{n^2_j}\rho^2_{i;j}1_{A_{i;j}^2}1_{B_2}\Bigg),
\end{align*}
and for every $j\in\mathbb{N}$, $\{{A_{i;j}^1}{B_1}\}_{i = 1,\ldots,n^1_j} \cup \{{A_{i;j}^2}{B_2}\}_{i = 1,\ldots,n^2_j}$ is still a partition, so $\rho^1 1_{B_1}+\rho^2 1_{B_2}\in\mathcal{P}_L^\tau(t_0,\rho_0)$ because of closure. 
\end{proof}
\subsection{Approximation by Stochastic $n$-player differential games}
\label{n_player_approximation}
This section collects all the approximation results needed in the proof of Theorem \ref{uniqueness_full}. The following results have been adapted and modified from \cite{cosso_master_2022} Appendix A to better suit our specific setting. We first introduce the infinite dimensional approximation. Let $(\Omega^{0'},\mathcal{F}^{0'},\{\mathcal{F}_t^{0'}\}_{t\geq 0},\mathbb{P}^{0'})$ be another complete filtered probability space carries a $m$-dimensional Brownian motions $\bar{B}^0$, where $\{\mathcal{F}_t^{0'}\}_{t\geq 0}$ is generated by $\bar{B}^0$ and augmented by all the $\mathbb{P}^{0'}$-null sets in $\mathcal{F}^{0'}$. Let $(\Omega^{1''},\mathcal{F}^{1''},\{\mathcal{F}^{1''}_t\}_{t\geq 0},\mathbb{P}^{1''})$ be another complete filtered probably space supporting a $m$-dimensional $\bar{W}$. $\{\mathcal{F}^{1''}_t\}_{t\geq 0}$ is generated by $\bar{W}$ and augmented by all the $\mathbb{P}^{1''}$-null sets in $\mathcal{F}^{1''}$. Define $\acute{\Omega}^0 := \Omega^0\times\Omega^{0'}$, $\acute{\mathcal{F}}^0:=\mathcal{F}^{0}\otimes\mathcal{F}^{0'}$, $\{\acute{\mathcal{F}}^0_t\}_{t\geq 0}:=\{\mathcal{F}_t^0\otimes\mathcal{F}_t^{0'}\}_{t\geq 0}$, $\acute{\mathbb{P}}^0 := \mathbb{P}^0\otimes\mathbb{P}^{0'}$, $\acute{\Omega}^1 := \Omega^1\times\Omega^{1''}$, $\acute{\mathcal{F}}^1:=\mathcal{F}^{1''}\otimes\mathcal{F}^1$, $\{\acute{\mathcal{F}_t^1}\}_{t\geq 0}:=\{\mathcal{F}_t^1\otimes\mathcal{F}_t^{1''}\}_{t\geq 0}$, $\acute{\mathbb{P}}^1 := \mathbb{P}^1\otimes\mathbb{P}^{1''}$. Moreover, recall the settings under Section \ref{setup}, define $\acute{\mathbb{F}}^{t} :=\{\acute{\mathcal{F}}^0_s\otimes\sigma(\bar{W}_{s\vee t} - \bar{W}_t)\otimes\sigma(W_{s\vee t}-W_t)\vee \mathcal{G}\}_{s\geq 0}$. Now we define the enlarged probability space
\begin{align*}
(\acute{\Omega},\acute{\mathcal{F}},\{\acute{\mathcal{F}}_t\}_{t\geq 0},\acute{\mathbb{P}}):=(\acute{\Omega}^0\times\acute{\Omega}^1,\acute{\mathcal{F}}^0\otimes\acute{\mathcal{F}}^1,\{\acute{\mathcal{F}}_t^0\otimes\acute{\mathcal{F}}_t^1\}_{t\geq 0},\acute{\mathbb{P}}^0\otimes\acute{\mathbb{P}}^1).
\end{align*}
We also denote by $\acute{\mathcal{A}}_t$ the set of control processes, the family of all $\acute{\mathbb{F}}^t$-progressively measurable processes $\acute{\alpha}:[0,T]\times \acute{\Omega}\to A$.\\
\hfill\\
For each $\varepsilon>0$, we consider the fixed $(b^N,f^N,g^N)$ in Lemma \ref{approximation_markovian}. As noted in the proof of Theorem \ref{uniqueness_full}, it is without loss of generality that we can restrict to $t\in [t_{N-1},t_N]$. For every $\varepsilon^0, \varepsilon^1>0$, $\acute{\alpha}\in\acute{\mathcal{A}}_t$, $\acute{\xi}\in L^2((\acute{\Omega},\acute{\mathcal{F}},\acute{\mathbb{P}});\mathbb{R}^d)$ , let $(\acute{X}_s^{N;\varepsilon^0,\varepsilon^1,t,\acute{\xi},\acute{\alpha}})_{t\leq s\leq T}$ be the unique solution to the following controlled McKean-Vlasov SDE
\begin{align*}
    \acute{X}_s =& \acute{\xi}+\int_t^s b^N_r(W^0_{t_1}, \ldots, W^0_{t_{N-1}},W^0_r,\acute{X}_r,\mathbb{P}_{\acute{X_r}}^{W^0},\acute{\alpha}_r)ds+\int_t^s\sigma_rdW_r + \int_t^s \sigma^0_rdW^0_r\\
    &+ \varepsilon^0(\bar{B}^0_s-\bar{B}^0_t)+ \varepsilon^1(\bar{W}_s-\bar{W}_t),\quad \forall s\in[t,T].
\end{align*}
Recall the Law Invariance result (Theorem \ref{law_invariance}), therefore $\forall(t,\mu,y)\in [t_{N-1},t_N]\times\mathcal{P}_2(\mathbb{R}^d)\times\mathbb{R}^m$ we could consider the value function 
\begin{align}
\label{infinite_dimensional_approximation}
v^N_{\varepsilon^0,\varepsilon^1}(t,\mu,y):=&\essinf_{\acute{\alpha}\in\acute{\mathcal{A}}_t}\acute{\mathbb{E}}_{\acute{\mathcal{F}}_t^0, W_t^0 = y}\Bigg[\int_t^T f^N_s(W^0_{t_1}, \ldots, W^0_{t_{N-1}},W^0_s,\acute{X}_s^{N;\varepsilon^0,\varepsilon^1,t,\acute{\xi},\acute{\alpha}},\mathbb{P}_{\acute{X}_s^{N;\varepsilon^0,\varepsilon^1,t,\acute{\xi},\acute{\alpha}}}^{W^0},\acute{\alpha}_s)ds \nonumber \\
    &+ g^N(W^0_{t_1}, \ldots, W^0_{t_{N}},\acute{X}_T^{N;\varepsilon^0,\varepsilon^1,t,\acute{\xi},\acute{\alpha}},\mathbb{P}_{\acute{X}_T^{N;\varepsilon^0,\varepsilon^1,t,\acute{\xi},\acute{\alpha}}}^{W^0})\Bigg],
\end{align}
where $\mathcal{L}(\acute{\xi}) = \mu$.\\
\hfill\\
Now, we are ready to introduce the finite dimension approximation. On top of the $(\acute{\Omega}^0,\acute{\mathcal{F}}^0,\{\acute{\mathcal{F}}^0_t\},\acute{\mathbb{P}}^0)$ defined above, we have to introduce some more probability spaces. Let $n\in\mathbb{N}$, $(\Omega^{1'},\mathcal{F}^{1'},\{\mathcal{F}^{1'}_t\}_{t\geq 0},\mathbb{P}^{1'})$ be another complete filtered probably space supporting $m$-dimensional independent Brownian motions $\bar{B}^1$, $\ldots$, $\bar{B}^n$ and $\bar{W}^1$, $\ldots$, $\bar{W}^n$. $\{\mathcal{F}^{1'}_t\}_{t\geq 0}$ is generated by $\bar{B}^i$, $\bar{W}^i$, $i=1,\ldots, n$ and augmented by all the $\mathbb{P}^{1'}$-null sets in $\mathcal{F}^{1'}$. Define $\bar{\Omega}^1 := \Omega^1\times\Omega^{1'}$, $\bar{\mathcal{F}^1}:=\mathcal{F}^{1'}\otimes\mathcal{F}^1$, $\{\bar{\mathcal{F}_t}^1\}_{t\geq 0}:=\{\mathcal{F}_t^1\otimes\mathcal{F}_t^{1'}\}_{t\geq 0}$, $\bar{\mathbb{P}}^1 := \mathbb{P}^1\otimes\mathbb{P}^{1'}$. We set the enlarged probability space
\begin{align*}
    (\bar{\Omega},\bar{\mathcal{F}},\{\bar{\mathcal{F}}_t\}_{t\geq 0},\bar{\mathbb{P}}):=(\acute{\Omega}^0\times\bar{\Omega}^1,\acute{\mathcal{F}}^0\otimes\bar{\mathcal{F}}^1,\{\acute{\mathcal{F}}_t^0\otimes\bar{\mathcal{F}}_t^1\}_{t\geq 0},\acute{\mathbb{P}}^0\otimes\bar{\mathbb{P}}^1).
\end{align*}
Recall the settings under Section \ref{setup}, define $\bar{\mathbb{F}}^{t} :=\{\acute{\mathcal{F}}^0_s\otimes\sigma(\bar{W}^1_{s\vee t} - \bar{W}^1_t)\otimes\ldots\otimes\sigma(\bar{W}^n_{s\vee t} - \bar{W}^n_t)\otimes\sigma(\bar{B}^1_{s\vee t}-\bar{B}^1_t)\otimes\ldots\otimes\sigma(\bar{B}^n_{s\vee t}-\bar{B}^n_t)\vee \mathcal{G}\}_{s\geq 0}$. For each $t\in[0,T]$, let $\bar{\mathcal{A}}^n_t$ be the family of all $\bar{\mathbb{F}}^{t}$-progressively measurable process $\bar{\alpha} = (\bar{\alpha}^1, \ldots, \bar{\alpha}^n):[0,T]\times\bar{\Omega}\to A^n$. For every $\varepsilon^0, \varepsilon^1>0$, $t\in[t_{N-1},t_N]$, $\bar{\alpha}\in\bar{\mathcal{A}}^n_t$, $\bar{\xi}^1$, $\ldots$, $\bar{\xi}^n \in L^2(\bar{\Omega},\bar{\mathcal{F}}_t,\bar{\mathbb{P}};\mathbb{R}^d)$ with $\bar{\xi}:= (\bar{\xi}^1, \ldots, \bar{\xi}^n)$. Let $(\bar{X}^{N;1,\varepsilon^0,\varepsilon^1,t,\bar{\xi},\bar{\alpha}}, \ldots, \bar{X}^{N;n,\varepsilon^0,\varepsilon^1,t,\bar{\xi},\bar{\alpha}})$ be the unique solution to the following system: 
\begin{align}
\label{approximating_system}
    \bar{X}_s^i =& \bar{\xi}^i + \int_t^s b^N_r(W^0_{t_1}, \ldots, W^0_{t_{N-1}},W^0_r,\bar{X}_r^i,\hat{\mu}_r^n,\bar{\alpha}_r^i)dr + \int_t^s \sigma_r d\bar{W}_r^i + \varepsilon^1(\bar{B}^i_s-\bar{B}^i_t) \\
    &+ \int_t^s \sigma^0_rdW^0_r + \varepsilon^0(\bar{B}^0_s-\bar{B}^0_t),\text{ with }i=1, \ldots, n,\nonumber
\end{align}
and
\begin{align*}
    \hat{\mu}_r^n := \frac{1}{n}\sum_{j=1}^n\delta_{\bar{X}_r^j}.
\end{align*}
For $\mathbb{P}_{\bar{\xi}} = \bar{\mu}$, $\bar{\mu}\in\mathcal{P}_2(\mathbb{R}^{n\times d})$, $y\in\mathbb{R}^m$. Consider the cooperative $n$-players game with the payoff
\begin{align*}
    \tilde{J}^N_{\varepsilon^0,\varepsilon^1,n}(t,\bar{\mu},y;\bar{\alpha}):=&\frac{1}{n}\sum_{i=1}^n\bar{\mathbb{E}}_{\bar{\mathcal{F}}_t^0, W_t^0 = y}\Bigg[\int_t^T f^N_s(W^0_{t_1}, \ldots, W^0_{t_{N-1}},W^0_s,\bar{X}_s^{N;i,\varepsilon^0,\varepsilon^1,t,\bar{\xi},\bar{\alpha}},\hat{\mu}_s^{n},\bar{\alpha}_s^i)ds \\
    &+ g^N(W^0_{t_1}, \ldots, W^0_{t_{N}},\bar{X}_T^{N;i,\varepsilon^0,\varepsilon^1,t,\bar{\xi},\bar{\alpha}},\hat{\mu}_T^{n})\Bigg],
\end{align*}
and the value function
\begin{align}
\label{approximating_system_value_function}     \tilde{v}^N_{\varepsilon^0,\varepsilon^1,n}(t,\bar{\mu},y):= \essinf_{\bar{\alpha}\in\bar{\mathcal{A}}^n_t} \tilde{J}^N_{\varepsilon^0,\varepsilon^1,n}(t,\bar{\mu},y;\bar{\alpha}).
\end{align}
Now we introduce smooth approximation to the state and variables. Define 
\begin{align*}
    b_{n,m}^{N;i} : [0,T]\times\mathbb{R}^{m\times N}\times\mathbb{R}^{dn}\times A \to \mathbb{R}^d,\,\,\,f_{n,m}^{N;i}:[0,T]\times\mathbb{R}^{m\times N}\times\mathbb{R}^{dn}\times A \to \mathbb{R}^d,\\
    g_{n,m}^{N;i}:\mathbb{R}^{m\times N}\times\mathbb{R}^{dn} \to \mathbb{R}^d
\end{align*}
by
\begin{align*}
    b_{n,m}^{N;i}(t,w_1,\ldots,w_N,\bar{x},a) :=& m^{nd}\int_{\mathbb{R}^{nd}} b^N\Big(t,w_1,\ldots,w_N,x_i-y_i,\frac{1}{n}\sum_{j=1}^n \delta_{x_j-y_j},a\Big)\prod_{j=1}^n\Phi(my_j)dy_j,\\
    f_{n,m}^{N;i}(t,w_1,\ldots,w_N,\bar{x},a) :=& m^{nd}\int_{\mathbb{R}^{nd}} f^N\Big(t,w_1,\ldots,w_N,x_i-y_i,\frac{1}{n}\sum_{j=1}^n \delta_{x_j-y_j},a\Big)\prod_{j=1}^n\Phi(my_j)dy_j,\\
    g_{n,m}^{N;i}(w_1,\ldots,w_N,\bar{x}) :=& m^{nd}\int_{\mathbb{R}^{nd}} g^N\Big(w_1,\ldots,w_N,x_i-y_i,\frac{1}{n}\sum_{j=1}^n \delta_{x_j-y_j}\Big)\prod_{j=1}^n\Phi(my_j)dy_j,
\end{align*}
for all $m \in \mathbb{N}$, $i=1,\ldots,n$, $w_i\in\mathbb{R}^m$, $i=1,\ldots,N$, $\bar{x} = (x_1,\ldots,x_n)\in\mathbb{R}^{nd}$, $(t,a)\in[0,T]\times A$, with $\Phi :\mathbb{R}^d\to [0,\infty)$ being $C^\infty$ functions with compact support satisfying $\int_{\mathbb{R}^d}\Phi(y)dy = 1$, and symmetric. We define the value function of this approximated system by 
\begin{align*}
       &\tilde{v}^N_{\varepsilon^0,\varepsilon^1,n,m}(t,\bar{\mu},y)\nonumber\\
    :=& \essinf_{\bar{\alpha}\in\bar{\mathcal{A}}^n_t} \frac{1}{n}\sum_{i=1}^n\bar{\mathbb{E}}_{\bar{\mathcal{F}}_t^0,W_t^0=y}\Bigg[\int_t^T f_{n,m}^{N;i}(s,W^0_{t_1}, \ldots, W^0_{t_{N-1}},W^0_s,\bar{X}_s^{N;1,m,\varepsilon^0,\varepsilon^1,t,\bar{\xi},\bar{\alpha}},\ldots,\bar{X}_s^{N;n,m,\varepsilon^0,\varepsilon^1,t,\bar{\xi},\bar{\alpha}},\bar{\alpha}_s^i)ds\nonumber\\ &+ g_{n,m}^{N;i}(W^0_{t_1}, \ldots, W^0_{t_{N}},\bar{X}_T^{N;1,m,\varepsilon^0,\varepsilon^1,t,\bar{\xi},\bar{\alpha}},\ldots,\bar{X}_T^{N;n,m,\varepsilon^0,\varepsilon^1,t,\bar{\xi},\bar{\alpha}})\Bigg],
\end{align*}
where $(\bar{X}_s^{N;i,m,\varepsilon^0,\varepsilon^1,t,\bar{\xi},\bar{\alpha}})_{t\leq s\leq T}$ solves the system (\ref{approximating_system}) with the drift being $b_{n,m}^{N;i}$.
\begin{remark}
    Note that when $\Phi$ is symmetric, we have
    \begin{align*}
        \int_{\mathbb{R}^{nd}}\frac{m^{nd}}{n}\sum_{j=1}^n|y_j|\prod_{j=1}^n\Phi(my_j)dy_j = \frac{1}{m}\int_{\mathbb{R}^d}|y_1|\Phi(y_1)dy_1.
    \end{align*}
\end{remark}
\begin{lemma}
\label{function_n_m_Lipschitz}
    Let ($\mathcal{A}1$), ($\mathcal{A}2$) and ($\mathcal{A}3$) hold. Recall the notation $\bar{x} = (x_1,\ldots,x_n)\in\mathbb{R}^{nd}$. Let $\hat{\mu}^{n,\bar{x}}$ be given by
    \begin{align*}
        \hat{\mu}^{n,\bar{x}}:= \frac{1}{n}\sum_{j=1}^n\delta_{x_j}.
    \end{align*}
    Let $i = 1,\ldots,n$, we have
    \begin{align*}
        &\lim_{m\to\infty}b_{n,m}^{N;i}(t,w_1,\ldots,w_N,\bar{x},a) = b^{N}(t,w_1,\ldots,w_N,x_i,\hat{\mu}^{n,\bar{x}},a),\\
        &\lim_{m\to\infty}f_{n,m}^{N;i}(t,w_1,\ldots,w_N,\bar{x},a) = f^N(t,w_1,\ldots,w_N,x_i,\hat{\mu}^{n,\bar{x}},a),
    \end{align*}
    uniformly for $(t,w_1,\ldots,w_N,\bar{x},a)\in [0,T]\times\mathbb{R}^{m\times N}\times\mathbb{R}^{nd}\times A$ and 
    \begin{align*}
        \lim_{m\to\infty}g_{n,m}^{N;i}(w_1,\ldots,w_N,\bar{x}) = g^{N}(w_1,\ldots,w_N,x_i,\hat{\mu}^{n,\bar{x}}),
    \end{align*}
    uniformly for $(w_1,\ldots,w_N,\bar{x}) \in \mathbb{R}^{m\times N}\times\mathbb{R}^{nd}$. We have the estimates:
    \begin{align*}
        &|b^N(t,w_1,\ldots,w_N,x_i,\hat{\mu}^{n,\bar{x}},a) - b_{n,m}^{N;i}(t,w_1,\ldots,w_N,\bar{x},a)|\\
        \vee &|f^N(t,w_1,\ldots,w_N,x_i,\hat{\mu}^{n,\bar{x}},a) - f_{n,m}^{N;i}(t,w_1,\ldots,w_N,\bar{x},a)|\\
        \vee &|g^N(w_1,\ldots,w_N,x_i,\hat{\mu}^{n,\bar{x}}) - g_{n,m}^{N;i}(w_1,\ldots,w_N,\bar{x})|\\
        \leq& Km^{nd}\int_{\mathbb{R}^{nd}} \Bigg(|y_i|+\frac{1}{n}\sum_{j=1}^n|y_j|\Bigg)\prod_{j=1}^n\Phi(my_j)dy_j.
    \end{align*}
    Finally, $\forall \bar{x},\bar{z}\in\mathbb{R}^{nd}$,
    \begin{align*}
        &|b_{n,m}^{N;i}(t,w_1,\ldots,w_N,\bar{x},a)-b_{n,m}^{N;i}(t,w_1,\ldots,w_N,\bar{z},a)|\\
        \vee &|f_{n,m}^{N;i}(t,w_1,\ldots,w_N,\bar{x},a)-f_{n,m}^{N;i}(t,w_1,\ldots,w_N,\bar{z},a)|\\
        \vee &|g_{n,m}^{N;i}(w_1,\ldots,w_N,\bar{x})-g_{n,m}^{N;i}(w_1,\ldots,w_N,\bar{z})|\\
        \leq &K\Big[|x_i-z_i| + \frac{1}{n}\sum_{j=1}^n |x_j-z_j|\Big],
    \end{align*}
   where the constant $K$ is the Lipschitz constant of our coefficient defined in assumption ($\mathcal{A}$1).
\end{lemma}
    \begin{proof}
        See Lemma A.3 in \cite{cosso_master_2022}.
    \end{proof}

\begin{definition}
    We define the following functions:
    \begin{align*}
        v^N_{\varepsilon^0,\varepsilon^1,n,m}(t,\mu,y):= \tilde{v}^N_{\varepsilon^0,\varepsilon^1,n,m}(t,\mu\otimes\ldots\otimes\mu,y),\,\,\,\,\,v^N_{\varepsilon^0,\varepsilon^1,n}(t,\mu,y):= \tilde{v}^N_{\varepsilon^0,\varepsilon^1,n}(t,\mu\otimes\ldots\otimes\mu,y).
    \end{align*}
\end{definition}
\begin{lemma}
    Let ($\mathcal{A}1$), ($\mathcal{A}2$) and ($\mathcal{A}3$) hold. Let $\varepsilon^0,\varepsilon^1 \geq 0$, for $\mathbb{P}^0$ a.e. $\omega^0$, for every $(t,\mu,y)\in [0,T]\times\mathcal{P}_2(\mathbb{R}^d)\times\mathbb{R}^m$ such that there exists $q>2$, $\mu \in \mathcal{P}_q(\mathbb{R}^d)$, we have
    \begin{align*}
        \lim_{n\to\infty}\lim_{m\to\infty}v^N_{\varepsilon^0,\varepsilon^1,n,m}(t,\mu,y) = \lim_{n\to\infty} v^N_{\varepsilon^0,\varepsilon^1,n}(t,\mu,y) = v^N_{\varepsilon^0,\varepsilon^1}(t,\mu,y),
    \end{align*}
    where $v^N_{\varepsilon^0,\varepsilon^1}(t,\mu,y)$ is defined in (\ref{infinite_dimensional_approximation}).
\end{lemma}
\begin{proof}
    The proof can be established using almost the same argument as presented in Theorem A.6 of \cite{cosso_master_2022}. However, we need to account for the additional variable $y$ in our analysis. We rewrite our system (\ref{approximating_system}) into the following form: 
   \begin{align}
    (\bar{X}_s^i,W_s^0) =& \bar{\xi}^i + \int_t^s \breve{b}^N_r(W^0_{t_1}, \ldots, W^0_{t_{N-1}},W^0_r,\bar{X}_r^i,\breve{\mu}_r^n,\bar{\alpha}_r^i)dr + \int_t^s \sigma_r dW_r^i + \varepsilon^1(\bar{B}^i_s-\bar{B}^i_t) \\
    &+ \int_t^s \sigma^0_rdW^0_r + \varepsilon^0(\bar{B}^0_s-\bar{B}^0_t),\text{ with }i=1, \ldots, n,\nonumber
\end{align}
where
\begin{align*}
    \breve{\mu}_r^n := \frac{1}{n}\sum_{j=1}^n\delta_{(\bar{X}_r^j,W_r^0)},
\end{align*}
and
\begin{align*}
    &\breve{b}^N:[0,T]\times\mathbb{R}^{m\times N}\times\mathbb{R}^{nd}\times\mathcal{P}_2(\mathbb{R}^{nd+m})\times A\to\mathbb{R},\\
    &\breve{b}^N(t,w_1,\ldots,w_N,x_1,\ldots,x_n,\mu,a):= b^N(t,w_1,\ldots,w_N,x_1,\ldots,x_n,\mu_{nd},a),
\end{align*}
$\mu_{nd}$ being the marginal distribution of $\mu$ on the first $nd$ axes. We are subject to
\begin{align*}
    \breve{v}^N_{\varepsilon^0,\varepsilon^1,n}(t,\mu,y;\bar{\alpha}):=&\essinf_{\bar{\alpha}\in\bar{\mathcal{A}}^n_t}\frac{1}{n}\sum_{i=1}^n\bar{\mathbb{E}}_{\bar{\mathcal{F}}_t^0, W_t^0 = y}\Bigg[\int_t^T \breve{f}^N_s(W^0_{t_1}, \ldots, W^0_{t_{N-1}},W^0_s,\bar{X}_s^{N;i,\varepsilon^0,\varepsilon^1,t,\bar{\xi},\bar{\alpha}},\breve{\mu}_s^{n},\bar{\alpha}_s^i)ds \\
    &+ \breve{g}^N(W^0_{t_1}, \ldots, W^0_{t_{N}},\bar{X}_T^{N;i,\varepsilon^0,\varepsilon^1,t,\bar{\xi},\bar{\alpha}},\breve{\mu}_T^{n})\Bigg],
\end{align*}
with $\mathcal{L}(\bar{\xi}) = \mu\otimes\ldots\otimes\mu$, and $\breve{f}$, $\breve{g}$ defined as $\breve{b}$. It is easy to check that the Lipschitz constant of $\breve{b},\breve{f},\breve{g}$ with respect to the measure term holds with the same constant. We can express $v^N_{\varepsilon^0,\varepsilon^1,n,m}$ as $\breve{v}^N_{\varepsilon^0,\varepsilon^1,n,m}$ using a comparable approach. Now, we can argue as \cite[Theorem A.6]{cosso_master_2022} with \cite[Theorem 3.1, Theorem 3.6]{limit_theory_tan} replacing the limit theory in \cite{cosso_master_2022}, because of the appearance of the common noise. We thus arrive at the desired conclusion.
\end{proof}
\begin{lemma}
\label{finite_dim_approximation_error}
    Let ($\mathcal{A}1$), ($\mathcal{A}2$) and ($\mathcal{A}3$) hold. We have the following estimate regarding the finite dimensional approximation of the coefficients:
    \begin{align*}
        &\Bigg|\frac{1}{n}\sum_{i=1}^n \int_{\mathbb{R}^{nd}}g_{n,m}^{N;i}(W_{t_1}^0,\ldots,W_{t_N}^0,x_1,\ldots,x_n)\mu(dx_1)\otimes\ldots\mu(dx_n)-\int_{\mathbb{R}^d}g^N(W_{t_1}^0,\ldots,W_{t_N}^0,x,\mu)\mu(dx)\Bigg|\\
        \vee&\Bigg|\frac{1}{n}\sum_{i=1}^n \int_{\mathbb{R}^{nd}}f_{n,m}^{N;i}(W_{t_1}^0,\ldots,W_{t_N}^0,x_1,\ldots,x_n,a)\mu(dx_1)\otimes\ldots\mu(dx_n)-\int_{\mathbb{R}^d}f^N(W_{t_1}^0,\ldots,W_{t_N}^0,x,\mu,a)\mu(dx)\Bigg|\\
        \vee&\Bigg|\frac{1}{n}\sum_{i=1}^n \int_{\mathbb{R}^{nd}}b_{n,m}^{N;i}(W_{t_1}^0,\ldots,W_{t_N}^0,x_1,\ldots,x_n,a)\mu(dx_1)\otimes\ldots\mu(dx_n)-\int_{\mathbb{R}^d}b^N(W_{t_1}^0,\ldots,W_{t_N}^0,x,\mu,a)\mu(dx)\Bigg|\\
        \leq&c_dK\Bigg(\int_{\mathbb{R}^d}|x|^q\mu(dx)\Bigg)^{1/q}h_n + Km^{nd}\int_{\mathbb{R}^{nd}}\Bigg(\frac{2}{n}\sum_{i=1}^n |y_i|\Bigg)\prod_{j=1}^n\Phi(my_j)dy_j\\
        \leq &c_dK\Bigg(\int_{\mathbb{R}^d}|x|^q\mu(dx)\Bigg)^{1/q}h_n + K\frac{2}{m}\int_{\mathbb{R}^d}|y_1|\Phi(y_1)dy_1,
    \end{align*}
    where $q\in(1,2]$, $c_d\geq0$ a constant depending only on $d$, and $h_n$ a sequence of real numbers, $h_n\to 0$ as $n\to\infty$. 
\end{lemma}
\begin{proof}
    See the proof in \cite{cosso_master_2022} Theorem 5.1.
\end{proof}
\begin{theorem}
\label{approximation_properties}
    Let ($\mathcal{A}1$), ($\mathcal{A}2$) and ($\mathcal{A}3$) hold. For every $\varepsilon^0,\varepsilon^1 >0$, $n,m\in \mathbb{N}$, there exists $\bar{v}^N_{\varepsilon^0,\varepsilon^1,n,m}:[t_{N-1},t_N]\times\mathbb{R}^{nd}\times\mathbb{R}^m\to \mathbb{R}$, $\bar{v}^N_{\varepsilon^0,\varepsilon^1,n,m}\in L^\infty\Big((\Omega^0,\mathcal{F}_{t_{N-1}}^0,\mathbb{P}^0);C^{1+\frac{\beta}{2},2+\beta}([t_{N-1},t_{N})\times \mathbb{R}^d)\cap C([t_{N-1},t_N]\times\mathbb{R}^d)\Big)$, for some $\beta \in (0,1)$, such that for $\mathbb{P}^0$ a.e. $\omega^0$, 
    \begin{align}
    \label{v_and_v_bar}        v^N_{\varepsilon^0,\varepsilon^1,n,m}(t,\mu,y) = \int_{\mathbb{R}^{nd}}\bar{v}^N_{\varepsilon^0,\varepsilon^1,n,m}(t,x_1,\ldots,x_n,y)\mu(dx_1)\ldots\mu(dx_n)
    \end{align}
    for every $(t,\mu,y)\in[0,T]\times \mathcal{P}_2(\mathbb{R}^d)\times\mathbb{R}^m$, and the following holds:
    \begin{enumerate}[(1)]
        \item $\bar{v}_{\varepsilon^0,\varepsilon^1,n,m}^N(t,\bar{x},y)$ solves
    \begin{align*}
        \begin{cases}
            &\displaystyle\partial_t \bar{v}^N_{\varepsilon^0,\varepsilon^1,n,m}(t,\bar{x},y)+\sum_{i=1}^n\essinf_{a_i\in A}\Bigg\{\frac{1}{n}f_{n,m}^{N;i} (t,W_{t_1}^0,\ldots,W_{t_{N-1}}^0,y,\bar{x},a_i) \\
            &\displaystyle+ \Big\langle b_{n,m}^{N;i}(t,W_{t_1}^0,\ldots,W_{t_{N-1}}^0,y,\bar{x},a_i),\partial_{x_i}\bar{v}^N_{\varepsilon^0,\varepsilon^1,n,m}(t,\bar{x},y) \Big\rangle\Bigg\}\\
        &\displaystyle+ \frac{1}{2}\tr\Big[\big(\sigma_t\sigma^\intercal_t+ (\varepsilon^1)^2 I+\sigma^0_t\sigma^{0;\intercal}_t
        +(\varepsilon^0)^2 I\big)\partial_{x_i x_i}^2 \bar{v}^N_{\varepsilon^0,\varepsilon^1,n,m}(t,\bar{x},y)\Big]
        \\&+\displaystyle\Big[\frac{1}{2}\tr\big(\partial_{yy}\bar{v}^N_{\varepsilon^0,\varepsilon^1,n,m}(t,\bar{x},y)\big)+\sum_{i=1}^n \tr\big(\sigma^0_t\partial_{x_i y}^2\bar{v}^N_{\varepsilon^0,\varepsilon^1,n,m}(t,\bar{x},y)\big)\Big]\\
        &\displaystyle+\frac{1}{2}\sum_{i,j=1,i\neq j}^n\tr\Big[\big(\sigma^0_t\sigma^{0;\intercal}_t+(\varepsilon^0)^2 I\big)\partial^2_{x_i x_j}\bar{v}^N_{\varepsilon^0,\varepsilon^1,n,m}(t,\bar{x},y)\Big]=0,\\
            &\displaystyle \bar{v}^N_{\varepsilon^0,\varepsilon^1,n,m}(T,\bar{x},y) = \frac{1}{n}\sum_{i=1}^n g_{n,m}^{N;i}(W_{t_1}^0,\ldots,W_{t_{N-1}}^0,\bar{x},y),
        \end{cases}
    \end{align*}
    for every $n,m \in\mathbb{N}$, $\bar{x} = (x_1,\ldots,x_n)\in\mathbb{R}^{dn}$ and $x_1,\ldots,x_n\in\mathbb{R}^d$.
    \item For all $(t,\bar{x},y)\in [0,T]\times\mathbb{R}^{dn}\times\mathbb{R}^m$, with $\bar{x} = (x_1,\ldots,x_n)$, $x_i \in \mathbb{R}^d$, for $i = 1,\ldots,n$, it holds that
    \begin{align}
        |\partial_{x_i}\bar{v}^N_{\varepsilon^0,\varepsilon^1,n,m}(t,\bar{x},y)|&\leq\frac{C_K}{n},\\
        |\partial_{y}\bar{v}^N_{\varepsilon^0,\varepsilon^1,n,m}(t,\bar{x},y)|&\leq\frac{C_K}{n},
    \end{align}
    where $C_K \geq 0$ depends only on $K$ and independent of $\varepsilon^0$, $\varepsilon^1$, $n$ and $m$.
    \item From (\ref{v_and_v_bar}) we have
    \begin{align}
        &\nonumber\partial_t v^N_{\varepsilon^0,\varepsilon^1,n,m}(t,\mu,y)\\ =& \int_{\mathbb{R}^{dn}}\partial_t\bar{v}^N_{\varepsilon^0,\varepsilon^1,n,m}(t,x_1,\ldots,x_n,y)\mu(dx_1)\ldots\mu(dx_n),\\
        &\partial_\mu v^N_{\varepsilon^0,\varepsilon^1,n,m}(t,\mu,y)(x)\nonumber\\ =&\sum_{i=1}^n\int_{\mathbb{R}^{d(n-1)}}\partial_{x_i}\bar{v}^N_{\varepsilon^0,\varepsilon^1,n,m}(t,x_1,\ldots,x_{i-1},x,x_{i+1},\ldots,x_n,y)\mu(dx_1)\ldots\mu(dx_{i-1})\mu(dx_{i+1})\ldots\mu(dx_n)\\
        &\nonumber\partial_x\partial_\mu v^N_{\varepsilon^0,\varepsilon^1,n,m}(t,\mu,y)(x)\\ =&\sum_{i=1}^n\int_{\mathbb{R}^{d(n-1)}}\partial_{x_i x_i}^2\bar{v}^N_{\varepsilon^0,\varepsilon^1,n,m}(t,x_1,\ldots,x_{i-1},x,x_{i+1},\ldots,x_n,y)\mu(dx_1)\ldots\mu(dx_{i-1})\mu(dx_{i+1})\ldots\mu(dx_n)\\
        &\partial_\mu^2 v^N_{\varepsilon^0,\varepsilon^1,n,m}(t,\mu,y)(x,z) \nonumber\\=& \sum_{\substack{i,j=1\\
        i\neq j}}^n\int_{\mathbb{R}^{d(n-2)}}\partial_{x_i x_j}^2\bar{v}^N_{\varepsilon^0,\varepsilon^1,n,m}(t,x_1,\ldots,x_{j-1},z,x_{j+1},\ldots,x_{i-1},x,x_{i+1},\ldots,x_n,y)\nonumber\\
        &\mu(dx_1)\ldots\mu(dx_{j-1})\mu(dx_{j+1})\ldots\mu(dx_{i-1})\mu(dx_{i+1})\ldots\mu(dx_n)
    \end{align}
    \end{enumerate}
\end{theorem}
\begin{proof}
See \cite{cosso_master_2022} Theorem A.7.
\end{proof}
\bibliography{Bib}
\bibliographystyle{plain}
\end{document}